\documentclass[11pt,a4paper,pdftex]{amsart}
\RequirePackage{booktabs}
\RequirePackage{multirow}
\RequirePackage{array}
\RequirePackage[noadjust]{cite}
\RequirePackage{amsmath}
\RequirePackage{amssymb}
\RequirePackage{graphicx}
\RequirePackage{amsthm}
\RequirePackage{graphicx}
\RequirePackage{color}
\RequirePackage{epsfig}
\RequirePackage{amscd}
\RequirePackage{xspace}
\RequirePackage[utf8]{inputenc}
\RequirePackage[all]{xy}
\usepackage{float}
\usepackage[normalem]{ulem}
\usepackage{enumitem}
\DeclareMathOperator{\id}{id}

\DeclareMathOperator{\Aut}{Aut}

\DeclareMathOperator{\End}{End}

\newcommand{\oo}{\mathcal{O}}

\newcommand{\HH}{\mathbb{H}}

\newcolumntype{C}[1]{>{\centering\let\newline\\\arraybackslash\hspace{0pt}}p{#1}}

\newcounter{gex}

\newcommand{\bbo}{\mathbb{O}}

\newcommand{\ie}{\emph{ie} }
\newcommand{\eg}{\emph{eg} }
\newcommand{\cf}{\emph{cf} }
\newcommand{\etal}{\emph{et al} }

\newcommand{\calc}{\mathcal{C}}

\newcommand{\calu}{\mathcal{U}}

\newcommand{\calv}{\mathcal{V}}

\DeclareMathAlphabet{\df}{U}{eus}{m}{n}
\DeclareMathAlphabet{\matheur}{U}{eur}{m}{n}

\makeatletter
\def\l@part{\@tocline{-1}{6pt plus2pt}{0pt}{}{\bfseries}}
\newcommand{\setindent}[1]{
  \addtocontents{toc}{
  \protect\def\protect\l@section{\protect\@tocline{1}{0pt}{#1}{}{}}}}
\newcommand{\setsubindent}[1]{
  \addtocontents{toc}{
  \protect\def\protect\l@subsection{\protect\@tocline{1}{0pt}{#1}{}{}}}}
\makeatother

\usepackage{slashed}

\newcounter{mtheorem}

\newtheorem*{theorem*}{Main Theorem}

\numberwithin{equation}{section}
\newtheorem{theorem}[equation]{Theorem}
\newtheorem{lemma}[equation]{Lemma}
\newtheorem{prop}[equation]{Proposition}
\newtheorem{corollary}[equation]{Corollary}

\theoremstyle{definition}
\newtheorem{definition}[equation]{Definition}
\newtheorem{example}[equation]{Example}

\theoremstyle{remark}
\newtheorem{remark}[equation]{Remark}
\newtheorem*{remark*}{Remark}

\newtheorem{conjecture}[equation]{Conjecture}
\newtheorem*{conjecture*}{Conjecture}
\newcommand{\se}{Sasaki--Einstein\xspace}
\newcommand{\nk}{nearly K\"ahler\xspace}
\newcommand{\nh}{nearly hypo\xspace}
\newcommand{\etase}{\eta^{\textup{se}}}
\newcommand{\omegase}{\omega^{\textup{se}}}
\newcommand{\Jse}{J^{\textup{se}}}
\newcommand{\gse}{g^{\textup{se}}}
\newcommand{\gsc}{g^{\textup{sc}}}

\newcommand{\abs}[1]{\lvert#1\rvert}
\newcommand{\norm}[1]{\left\lvert#1\right\rvert}

\newcommand{\unitary}[1]{\textup{U$(#1)$}}

\newcommand{\sunitary}[1]{\textup{SU$(#1)$}}

\newcommand{\sorth}[1]{\textup{SO$(#1)$}}
\newcommand{\lorentz}[1]{\textup{SO$_0(1,#1)$}}
\newcommand{\orchro}[1]{\textup{O$^+( 1,#1 )$}}
\newcommand{\Sp}[1]{\textup{Sp$(#1)$}}
\newcommand{\orth}[1]{\textup{O$(#1)$}}
\newcommand{\spin}[1]{\textup{Spin$(#1)$}}

\newcommand{\vol}{\operatorname{vol}}

\newcommand{\C}{\mathbb{C}}
\newcommand{\R}{\mathbb{R}}
\newcommand{\Z}{\mathbb{Z}}
\newcommand{\N}{\mathbb{N}}

\newcommand{\PP}{\mathbb{P}}
\newcommand{\CP}{CP}
\newcommand{\BCP}{\mathbb{CP}}
\newcommand{\Sph}{\mathbb{S}}

\newcommand{\ra}{\rightarrow}

\newcommand{\diag}{\operatorname{diag}}
\newcommand{\Imag}{\operatorname{Im}}
\newcommand{\Real}{\operatorname{Re}}
\newcommand{\trace}{\operatorname{Tr}}

\newcommand{\gtwo}{\ensuremath{\textup{G}_2}\xspace}

\def\co{\colon\thinspace}

\newcommand{\Lie}[1]{\mathfrak{#1}}

\newcommand{\lspan}{\operatorname{\text{Span}}}
\newcommand{\Scal}{\operatorname{\text{Scal}}}

\newcommand{\bohm}{\mathcal{B}}
\newcommand{\Ndiag}{\ensuremath{\textup{N}_{1,1}}}
\newcommand{\alphaW}{\alpha_{\text{\tiny W}}}
\newcommand{\betaW}{\beta_{\text{\tiny W}}}
\newcommand{\alphaH}{\alpha_{\text{\tiny H}}}
\newcommand{\betaH}{\beta_{\text{\tiny H}}}

\title{New \gtwo--holonomy cones and exotic nearly K\"ahler structures on $S^6$ and $S^3 \times S^3$}
\author[L.~Foscolo]{Lorenzo~Foscolo}
\author[M.~Haskins]{Mark~Haskins}

\setlist{leftmargin=*}

\begin{document}

\begin{abstract}
There is a rich theory of so-called (strict) nearly K\"ahler manifolds, almost-Hermitian manifolds generalising the famous
almost complex structure on the $6$--sphere induced by octonionic multiplication.
Nearly K\"ahler $6$--manifolds
play a distinguished role both in the general structure theory and also because of their connection with
singular spaces with holonomy group the compact exceptional Lie group $G_2$:
the metric cone over a Riemannian $6$--manifold $M$ has holonomy contained in $G_2$
if and only if $M$ is a nearly K\"ahler $6$--manifold.

A central problem in the field has been the absence of any complete inhomogeneous examples.
We prove the existence of the first complete inhomogeneous nearly K\"ahler $6$--manifolds
by proving the existence of at least one cohomogeneity one nearly K\"ahler structure
on the $6$--sphere and on the product of a pair of $3$--spheres.
We conjecture that these are the only simply connected (inhomogeneous) cohomogeneity one nearly K\"ahler structures
in six dimensions.

\end{abstract}
\maketitle

\setcounter{secnumdepth}{1}

\section{Introduction}

At least since the early 1950s (see Steenrod's 1951 book \cite[41.22]{Steenrod:Book}) it has been well known that viewing $\Sph^6$ as the unit sphere in $\Imag{\bbo}$ 
endows it with a natural nonintegrable almost complex structure $J$ defined via octonionic multiplication. Since $J$ is compatible with the round metric $g_{\textrm{rd}}$, the triple $(g_{\textrm{rd}},J,\omega)$, where $\omega (\cdot, \cdot)=g_{\textrm{rd}}(J\cdot,\cdot)$, defines an 
\mbox{almost-Hermitian} structure on $\Sph^6$.  Its torsion has very special properties: in particular, $d\omega$ is the real part of a complex volume form $\Omega$. Appropriately normalised, the pair $(\omega, \Omega)$ defines an \mbox{\sunitary{3}--structure} on $\Sph^6$ which by construction is invariant under the exceptional compact Lie group $\gtwo \simeq \Aut (\bbo)$.

Octonionic multiplication also defines a \gtwo--invariant $3$--form $\varphi$ on 
$\Imag{\bbo}$ by 
$$\varphi (u,v,w)=uv \cdot w.$$ We call this the standard \gtwo--structure on $\R^7$. Regarding $\R^7$ as the Riemannian cone over $(\Sph^6,g_{\textrm{rd}})$, $\varphi$ and its Hodge dual are given in terms of $(\omega,\Omega)$:
\begin{equation}\label{eq:G2:cone}
\varphi = r^2 dr \wedge \omega + r^3\Real{\Omega}, \qquad \ast\varphi = -r^3 dr\wedge \Imag{\Omega} + \tfrac{1}{2}r^4\omega^2.
\end{equation}
Conversely, viewing $\Sph^6$ as the level set $r=1$ in $\R^7$, the \sunitary{3}--structure $(\omega,\Omega)$ is recovered from $\varphi$ and $\ast\varphi$ by restriction and contraction by the scaling vector field $\frac{\partial}{\partial r}$.

More generally, consider a $7$--dimensional Riemannian cone $C=C(M)$ over a smooth compact manifold $(M^6,g)$. Suppose that the holonomy of the cone is contained in \gtwo. Then there exists a pair of closed (in fact parallel) differential forms $\varphi$ and $\ast\varphi$, pointwise equivalent to the model forms on $\R^7$ and homogeneous with respect to scalings on $C$. Just as above, viewing $M$ as the level set $r=1$ in $C$, the restriction and contraction by $\frac{\partial}{\partial r}$ of $\varphi$ and $\ast\varphi$ define an \mbox{\sunitary{3}--structure} $(\omega,\Omega)$ on $M$ satisfying \eqref{eq:G2:cone}. In particular, the closedness of $\varphi$ and $\ast\varphi$ is equivalent to
\begin{equation}\label{eq:NK}
\begin{cases}
d\omega = 3\Real{\Omega},\\
d\Imag{\Omega}=-2\,\omega^2.
\end{cases}
\end{equation}

\begin{definition}\label{def:NK}
A \emph{\nk} $6$--manifold is a manifold $M^6$ endowed with an \mbox{\sunitary{3}--structure} $(\omega, \Omega)$ satisfying \eqref{eq:NK}. We call \eqref{eq:NK} the \emph{\nk structure equations}. 
\end{definition}

There are other possible equivalent definitions of a \nk $6$--manifold. By relating the holonomy reduction of the cone $C(M)$ to the existence of a parallel spinor instead of a pair of distinguished parallel forms, \nk $6$--manifolds can be characterised as those $6$--manifolds admitting a real Killing spinor \cite{Grunewald}. Alternatively, one could give a definition in terms of Gray--Hervella torsion classes of almost Hermitian manifolds \cite{Gray:Hervella}. Our definition corresponds then to what are usually called \emph{strict} \nk $6$--manifolds, to distinguish them from K\"ahler manifolds which, having vanishing torsion, belong to every torsion class.

\begin{remark*}
The latter point of view allows one to define \nk manifolds in every even dimension. See \cite[\S 14.3.2]{BG:Sasaki} for some references and an early history of the subject, 
including early lesser known contributions from the Japanese school before the terminology nearly K\"ahler had been adopted. 
Even in this more general context, \nk $6$--manifolds play a distinguished role, see \cite[Theorem 1.1]{Nagy:Product}.
\end{remark*}

Since every manifold with holonomy contained in \gtwo is Ricci-flat, \nk $6$--manifolds are necessarily Einstein with scalar curvature $30$ 
(in fact, Gray proved this fact before the connection with \gtwo holonomy had been noticed \cite[Theorem 5.2]{Gray:Math:Annalen}).
In particular, a complete \nk \mbox{$6$--manifold} is compact with finite fundamental group and its universal cover is also a complete \nk manifold. In the rest of the paper we will therefore restrict to the case of simply connected \nk $6$--manifolds.

Besides the \gtwo--invariant \nk structure on the $6$--sphere $\Sph^6 = \gtwo/\sunitary{3}$, there are only three known examples of complete simply connected \nk $6$--manifolds, all of which are homogeneous: $S^3 \times S^3 = \sunitary{2}^3/\triangle\sunitary{2}$, $\CP ^3 = \textup{Sp}(2)/\unitary{1}\times\textup{Sp}(1)$ and the flag manifold $F_3=\sunitary{3}/{T}^2$. They were first constructed in 1968 by Gray and Wolf \cite{Gray:Wolf} in their work on $3$--symmetric spaces (which also yielded many
homogeneous \nk manifolds in higher dimensions). Recently Cort\'es--V\'asquez \cite{Cortes:Vasquez} have constructed and partially classified locally homogeneous nearly K\"ahler $6$--manifolds by considering finite quotients of these homogeneous nearly K\"ahler structures.

Although as early as 1958 Calabi \cite{Calabi:TAMS} exploited the fact that any oriented 
hypersurface of $\R^7$ admits an $\sunitary{3}$--structure thanks to the inclusion $\sunitary{3} \subset \gtwo$, 
the connection between \nk geometry in $6$ dimensions and \gtwo holonomy that we have emphasised seems to have been noticed only in the 1980s;  the homogeneous \nk $6$--manifolds described above then played an important role in the early development of metrics with holonomy  \gtwo. The first explicit example of a local metric with full holonomy \gtwo given by Bryant in \cite[Section 5]{Bryant:special:holonomy} is the \gtwo--cone over the flag manifold $F_3$; this appears to be the first appearance of the \nk equations \eqref{eq:NK}.
Furthermore, the first examples of complete \gtwo--metrics constructed by Bryant--Salamon \cite{Bryant:Salamon} are asymptotically conical manifolds modelled at infinity on the cones over \nk $S^3\times S^3$, $\CP^3$ and $F_3$. In each case, Bryant--Salamon construct a $1$--parameter family of (cohomogeneity one) complete metrics with holonomy \gtwo on the total space of a vector bundle over $S^3$, $S^4$ and $\CP^2$, respectively. The parameter measures the size of the zero section and as it goes to zero the manifold develops an isolated singularity and converges in the Gromov--Hausdorff sense to the corresponding \gtwo--cone.

The behaviour of these solutions exemplifies the reason for the interest in \nk $6$--manifolds from the point of view of \gtwo geometry: Riemannian cones over $6$--dimensional \nk manifolds and complete metrics with holonomy \gtwo asymptotic to them provide local models for the simplest type of singularities and for singularity formation in families of smooth \gtwo--manifolds.

The obvious natural question now is whether there are any other \nk $6$--manifolds. Thanks to the recent work of Butruille \cite{Butruille}, no such manifold can be homogeneous. The local existence of \nk structures in dimension $6$ was studied by Reyes Carrion \cite[\S 4.5]{Reyes:Carrion:thesis} and later by Bryant \cite{Bryant:almost:complex:6D}, in both cases 
using Cartan--K\"ahler theory; the conclusion is that \nk structures have the same local generality as Calabi--Yau structures and, in particular, there are many local inhomogeneous \nk structures.  The outstanding question is therefore to find new complete examples. 

An apparently promising source of \nk manifolds of any dimension is that of twistor spaces of quaternionic K\"ahler manifolds with positive scalar curvature \cite{Salamon:twistor}. In $6$ dimensions the connection with \nk geometry was first noticed by Eells--Salamon \cite{Eells:Salamon:Twistor} (later generalised to higher dimensions in \cite{NK:twistor:spaces}). Unfortunately, while this construction provides many incomplete examples, \eg starting from quaternionic K\"ahler orbifolds, so far the only complete examples it has yielded are homogeneous. By a result of Hitchin \cite{Hitchin:Kahler:Twistor} the only quaternionic K\"ahler (equivalently, self-dual Einstein) $4$--manifolds with positive scalar curvature are $S^4$ and $CP^2$ endowed with their standard metrics, in which case the \nk structures on the twistor spaces coincide with the homogeneous ones on $\CP^3$ and $F_3$, respectively. In higher dimensions the only known quaternionic K\"ahler manifolds with positive scalar curvature are Wolf's quaternionic symmetric spaces \cite{Wolf:1965}; an influential conjecture of LeBrun--Salamon \cite{LeBrun:Salamon}, known to hold up to dimension $8$, asserts the non-existence of inhomogeneous examples. 

On the other hand, the scarcity of \nk $6$--manifolds, or equivalently of \mbox{\gtwo--holonomy} cones, is surprising when compared to geometries related to other special holonomy groups: there are infinitely many Calabi--Yau cones \cite{Sparks:survey} and infinitely many hyperk\"ahler and 
$\textup{Spin}(7)$--cones \cite[\S 13.7, \S 14.3]{BG:Sasaki}.
Indeed, in this paper we show that complete \nk $6$--manifolds need not be homogeneous.

\begin{theorem*}
There exists an inhomogeneous \nk structure on $S^6$ and $S^3 \times S^3$.
\end{theorem*}

There are two very natural approaches to attempt to construct \nk $6$--manifolds. On the one hand, from the perspective of symmetries, the next step beyond the homogeneous setting would be to consider cohomogeneity one \nk $6$--manifolds, \ie those that admit an isometric action by a compact Lie group whose generic orbit is of codimension one. A completely different starting point would be to exploit the existence of a large number of $6$--dimensional \emph{singular} \nk spaces and to attempt to produce smooth \nk manifolds from these by singular perturbation methods. 
As we will explain, both points of view will play an important role in the proof of the Main Theorem.
 
Apart from orbifolds, the simplest singular Einstein spaces with positive scalar curvature are those obtained by the \emph{sine-cone} construction (also called spherical suspension): given a smooth compact Einstein manifold $(N,g_N)$ with positive scalar curvature appropriately normalised, the sine-cone $SC(N)$ over $N$ is the product $[0,\pi] \times N$ endowed with the metric $g^{\textup{sc}}=dr^2 + \sin^2{r}\, g_N$. This has two isolated singularities at $r=0$ and $\pi$ modelled on the cone over $N$. The sine-cone construction has a very simple geometric interpretation in terms of cones: the cone over $SC(N)$ 
is Ricci-flat and is isometric to the product of the Ricci-flat metric cone $C(N)$ over $N$ and the real line. Specialising to seven dimensions, the further requirement that the holonomy of the cone over $SC(N)$ be contained in \gtwo forces $C(N)$ to be a $6$--dimensional Calabi--Yau cone. The induced geometric structure on the cross-section $N$, analogous to the \nk condition for a \gtwo--cone, is called a \se structure. Hence the existence of infinitely many \se $5$--manifolds, see for example \cite{Sparks:survey}, yields infinitely many $6$--dimensional \nk sine-cones.

Not every \nk sine-cone is a good starting point for constructing smooth \nk $6$--manifolds via the desingularisation method. From the viewpoint of singular perturbation methods it is natural to consider only Calabi--Yau cones that admit asymptotically conical Calabi--Yau desingularisations. Many such examples are now known \cite{Conlon:Hein:Duke}. The simplest such cone is the conifold, which is well known to admit two Calabi--Yau  desingularisations: the Candelas--de la Ossa structure on the small resolution \cite{Candelas:delaOssa} and the structure on the smoothing due to Candelas--de la Ossa and Stenzel \cite{Candelas:delaOssa,Stenzel}. The Calabi--Yau structures on the conifold itself and on both of its desingularisations are cohomogeneity one. The group acting is $\sunitary{2}\times\sunitary{2}$ and the generic (principal) orbit is diffeomorphic to $\sunitary{2}\times\sunitary{2}/\triangle\unitary{1}\simeq S^2\times S^3$ in all three cases. The singularity of the conifold is replaced by a round totally geodesic holomorphic $S^2$ in the small resolution and by a round totally geodesic special Lagrangian $S^3$ in the smoothing. In both cases these totally geodesic spheres are the unique lower-dimensional (singular) orbits. 

The cross-section of the conifold is $S^2\times S^3$ endowed with its homogeneous \se structure. The sine-cone over it is a cohomogeneity one \nk space with two isolated singularities modelled on the conifold. One could try to construct an approximate solution to \eqref{eq:NK} on manifolds obtained by gluing rescaled copies of either desingularisation of the conifold into neighbourhoods of each singular point. Since both the singular background and the ``bubbles'' we glue in are of cohomogeneity one, this raises the question of whether complete cohomogeneity one \nk structures exist on such manifolds.

Podest\`a and Spiro initiated the study of cohomogeneity one \nk $6$--manifolds. In \cite{Podesta:Spiro:I} they classified the possible group actions, principal and singular orbits and diffeomorphism-types of complete simply connected cohomogeneity one \nk $6$--manifolds. In fact, the only case of possible interest is exactly the case described above: the principal orbit type is $\sunitary{2}\times\sunitary{2}/\triangle\unitary{1}$; there are two singular orbits, which are spheres of either $2$ or $3$ dimensions; the compact $6$--manifold is obtained by identifying neighbourhoods of the two singular orbits along their boundary $S^2 \times S^3$; these neighbourhoods are $\sunitary{2}\times\sunitary{2}$--equivariantly diffeomorphic to the small resolution or to the smoothing of the conifold; the four resulting manifolds are $S^6$,  $S^3\times S^3$, $\CP^3$ and $S^2 \times S^4$. In a second paper \cite{Podesta:Spiro:II} they studied this case in more detail, but were unable to establish the existence of new complete \nk structures.   

The \nk structures on $S^6$ and $S^3\times S^3$ that we construct in the Main Theorem are of cohomogeneity one and in fact we conjecture these are the unique (inhomogeneous) complete cohomogeneity one \nk $6$--manifolds. 
In particular, we conjecture that $S^2 \times S^4$ carries no cohomogeneity one \nk structure and $\CP^3$ only its homogeneous one.  

In the rest of the Introduction we give the plan of the paper, which serves at the same time as a detailed outline of the proof of the Main Theorem. 
The techniques we use in the paper are cohomogeneity one methods. 
Unlike the construction of, say, cohomogeneity one Sasaki--Einstein metrics \cite{gauntlett:SE,Conti:SE} we do not find explicit closed-form 
expressions for our new nearly K\"ahler structures  and in this sense the theorem 
is an abstract existence result (see however the final section of the paper). 
Instead the paper is more in the spirit of B\"ohm \cite{Bohm:Spheres,Bohm:Complete}.
The desingularisation intuition, however, provides crucial geometric insight when applying the cohomogeneity one methods, 
in particular in the consideration of certain geometrically-motivated singular limits and rescalings.

\subsection*{Plan of the paper}

In the same way that an oriented hypersurface of a $7$--manifold with a \gtwo--structure admits an induced \sunitary{3}--structure, any oriented hypersurface of a $6$--manifold endowed with an \sunitary{3}--structure $(\omega,\Omega)$ comes naturally equipped with an \sunitary{2}--structure. When $(\omega,\Omega)$ satisfies differential equations such as \eqref{eq:NK}, so does the induced \sunitary{2}--structure. The \sunitary{2}--structures induced on oriented hypersurfaces of Calabi--Yau and \nk $6$--manifolds are called \emph{hypo} and \emph{nearly hypo} structures, respectively. In the spirit of Hitchin \cite{Hitchin:stable:forms}, away from the singular orbits we regard a cohomogeneity one \nk (Calabi--Yau) $6$--manifold as a curve in the space of invariant \nh (hypo) structures on the principal orbit. This curve must satisfy a system of first order ODEs.

In Section \ref{sec:nearly:hypo} we parametrise the space of \nh structures on $S^2\times S^3$ invariant under $\sunitary{2} \times \sunitary{2}$, showing that it is a smooth connected $4$--manifold. In Section \ref{sec:Coho1:NK} we derive the ODE system \eqref{eq:nk:odes:t} whose solutions are cohomogeneity one \nk structures. We note the existence of continuous and discrete symmetries of this system; the latter will play an important role in the proof of the Main Theorem. Up to the action of these symmetries, we find a $2$--parameter family of $\sunitary{2}\times \sunitary{2}$--invariant cohomogeneity one \nk structures on the product of some interval with $S^2 \times S^3$. This gives an alternative more geometric derivation of results contained in \cite{Podesta:Spiro:II}. We have been unable to find a closed form for the general solution of these ODEs, but there are four explicit solutions which correspond to the homogeneous \nk structures on $S^6$, $\CP^3$ and $S^3 \times S^3$ and to the sine-cone over the invariant \se structure on $S^2 \times S^3$. The latter two play a role in the proof of the Main Theorem.

The generic solution in this $2$--parameter family does not extend to a complete \nk structure on a closed $6$--manifold. In Section \ref{sec:closing} we understand necessary conditions for such an extension to be possible. Based on the desingularisation philosophy, close to the sine-cone we might expect to find two $1$--parameter families of local cohomogeneity one smooth \nk structures modelled on the two different Calabi--Yau desingularisations of the conifold. We prove that this is indeed the case; the proof consists of two steps. In Section \ref{sec:closing}, studying singular initial value problems for the ODE system \eqref{eq:nk:odes:t}, we prove the existence of two \mbox{$1$--parameter} families of solutions $\{ \Psi_a \}_{a>0}$ and $\{ \Psi_b\}_{b>0}$ that extend smoothly over a singular orbit $S^2$ or $S^3$, respectively. In both cases the parameter $a$ or $b$ measures the size of the singular orbit, but, unlike the Calabi--Yau case, the parameter does not arise from an overall rescaling and instead represents a nontrivial deformation.
Any cohomogeneity one \nk structure that extends to a complete manifold must belong to (at least) one of these families. In the first part of Section \ref{sec:limits} we then confirm the expectation that these families are \nk deformations of the Calabi--Yau desingularisations of the conifold. More precisely, in the limit where the size of the singular orbit tends to zero, suitably rescaled, $\Psi_a$ and $\Psi_b$ converge to the Calabi--Yau structures on the small resolution and the smoothing, respectively.

Since any complete cohomogeneity one \nk manifold has two singular orbits, it would necessarily contain a (unique) principal orbit of maximal volume. This gives a further necessary condition for a member of $\{ \Psi_a \}_{a>0}$ or $\{ \Psi_b\}_{b>0}$ to admit a smooth \nk completion. The space of invariant \nh structures on $S^2 \times S^3$ that could potentially arise as such a maximal volume orbit is a smooth submanifold of codimension $1$ in the space of all invariant \nh structures. The main result of Section \ref{sec:Orbital:volume} is that in fact every member of both families $\{ \Psi_a \}_{a>0}, \{ \Psi_b\}_{b>0}$ admits a unique maximal volume orbit. The proof of this uses a continuity argument relying on a compactness result for the space of maximal volume orbits with a given upper bound on volume. The discrete symmetries of the ODE system \eqref{eq:nk:odes:t} play an important role in the matching construction described below. To this end we also determine the fixed locus of these symmetries in the space of maximal volume orbits.

The existence of maximal volume orbits now becomes a tool to detect which solutions in the families $\{ \Psi_a \}_{a>0}, \{ \Psi_b\}_{b>0}$ extend to a complete cohomogeneity one \nk structure. Topologically the resulting closed $6$--manifold is described as the union of neighbourhooods of the two singular orbits identified along their boundary. Hence we consider pairs of solutions in the two $1$--parameter families and try to match them across a principal orbit using the discrete symmetries of the ODE system \eqref{eq:nk:odes:t}. The maximal volume orbit provides a geometrically preferred slice to carry out this ``gluing''. The necessary matching conditions are stated in the Doubling and Matching Lemmas \ref{lem:Doubling} and \ref{lem:Matching} at the end of Section \ref{sec:Orbital:volume}. These are best formulated in terms of two continuous curves $\alpha$ and $\beta$ parametrising the maximal volume orbits of $\{ \Psi_a \}_{a>0}$ and $\{ \Psi_b\}_{b>0}$, respectively. Complete cohomogeneity one \nk manifolds are in one-to-one correspondence with intersection points of the two curves, self-intersection points of either curve and points on the curves lying in the fixed locus of the discrete symmetries in the space of invariant maximal volume orbits. The diffeomorphism type of the corresponding closed cohomogenity one \nk $6$--manifold is determined by the pair of singular orbits and the discrete symmetry used to identify the pair of solutions to \eqref{eq:nk:odes:t} across their maximal volume orbit.

The proof of the Main Theorem is now reduced to the problem of describing the behaviour of the curves $\alpha$ and $\beta$. The explicit solutions to \eqref{eq:nk:odes:t} corresponding to the homogeneous \nk structure on $S^6$, $\CP^3$ and $S^3 \times S^3$ already yield distinguished points on the curves. In Section \ref{sec:limits} we describe the limit of $\alpha$ and $\beta$ for small values of the parameter: as expected from the desingularisation philosophy, $\Psi_a$ and $\Psi_b$ converge to the sine-cone as $a$ and $b$ tend to zero. The proof of this fact makes use of a functional $\bohm$ introduced by B\"ohm in \cite{Bohm:Complete} as a Lyapunov function for the ODE system describing cohomogeneity one Einstein metrics. Since the space of invariant metrics in our setting does not reduce to relative rescalings along distinct subspaces (there are isomorphic components 
in the isotropy representation of the principal orbit and therefore ``non-diagonal'' terms in the metric), it is not clear that $\bohm$ is a Lyapunov function for the system \eqref{eq:nk:odes:t}. However, we establish that the functional $\bohm$ restricted to the space of invariant maximal volume orbits has an absolute minimum at the homogeneous \se structure on $S^2 \times S^3$. This is enough to establish the convergence of $\Psi_a$ and $\Psi_b$ to the sine-cone as $a$ and $b$ tend to zero using the convergence of the bubbles to the Calabi--Yau desingularisations we have already proved.

Using a comparison argument for solutions of Sturm--Liouville equations, in Section \ref{sec:S3xS3} we compare $\Psi_a$ and $\Psi_b$ for small $a$ and $b$ with a solution of the linearisation of \eqref{eq:NK} on the sine-cone. This comparison argument yields enough information on the curve $\beta$ to prove the existence of the complete \nk structure on $S^3 \times S^3$ given in the Main Theorem. The fact that this is inhomogeneous follows from a curvature computation.

The existence of at least two complete cohomogeneity one \nk structures on $S^3 \times S^3$ has the following consequence: an arc of the curve $\beta$ together with its image under the discrete symmetries form the boundary of a bounded region in the space of invariant maximal volume orbits containing the homogeneous \se structure on $S^2 \times S^3$. By the convergence of $\Psi_a$ to the sine-cone as $a \ra 0$, the curve $\alpha$ starts inside this region. In Section \ref{sec:S6} we prove that $\alpha$ is unbounded as the parameter $a \ra \infty$. The proof is based on a less geometric ad hoc rescaling argument suggested by the actual shape of the Taylor series of $\Psi_a$. We conclude that up to discrete symmetries the curves $\alpha$ and $\beta$ must intersect in at least two points; one of these corresponds to the homogeneous \nk structure on $S^6$. The second intersection point yields a complete cohomogeneity one \nk structure on $S^6$ which is shown to be inhomogeneous by a curvature computation.

The proof of the Main Theorem is now complete. In fact we conjecture that the theorem yields all (inhomogeneous) complete cohomogeneity one \nk structures. In Section \ref{sec:numerics} we provide some numerical evidence for this conjecture and some further information about the new complete cohomogeneity one solutions that we have 
obtained as part of a systematic numerical study. 

\medskip
The authors would like to thank Bobby Acharya, Robert Bryant, Simon Donaldson, Hans-Joachim Hein, David Morris, Johannes Nordstr\"om and Simon Salamon 
for stimulating discussions related to this paper. They would also like to thank Ron Stern for raising the possibility that $S^2 \times S^4$ could arise 
as a cohomogeneity one $\sunitary{2} \times \sunitary{2}$--invariant \nk \mbox{$6$--manifold}. 
MH would like to thank EPSRC for their continuing support of his research under Leadership Fellowship EP/G007241/1 and his Developing Leaders Grant EP/L001527/1. 
MH would also like to thank the Simons Center for Geometry and Physics at Stony Brook for hosting him several times, particularly for his long visit during Summer 2014.

\section{Cohomogeneity one \sunitary{3}--structures and homogeneous \sunitary{2}--structures}\label{sec:nearly:hypo}
As explained in the Introduction we adopt Hitchin's approach in our study of cohomogeneity one \nk $6$--manifolds, \ie 
we concentrate on the geometry induced on (invariant) hypersurfaces. 
In the context of $6$--manifolds with an \sunitary{3}--structure this approach has been pursued by Conti and Salamon \cite{Conti:Salamon} 
and by Fernandez \etal \cite{Fernandez:nearly:hypo} in the Calabi--Yau and \nk settings, respectively.  
This section thus starts by recalling the basic definitions of \sunitary{2}--structures concentrating on those that arise as hypersurfaces in 
Calabi--Yau and \nk $6$--manifolds: \emph{hypo} and \emph{nearly hypo} structures respectively. 
For applications to the construction of complete compact cohomogeneity one \nk $6$--manifolds 
it follows from the results of Podest\`a and Spiro \cite{Podesta:Spiro:I} (recalled in more detail at the beginning of Section \ref{sec:Coho1:NK}) that 
the only interesting case is that of \nh structures on $S^2 \times S^3$ invariant under the action of 
$\sunitary{2} \times \sunitary{2}$ with isotropy group $\Delta \unitary{1}$. 
The main results of the section are therefore those about $\sunitary{2} \times \sunitary{2}$--invariant \nh 
structures on $S^2 \times S^3$, especially Proposition \ref{prop:Invariant:nearly:hypo:structures}, and related results about invariant hypo structures
(Proposition \ref{prop:invariant:hypo:structures} and Theorem \ref{thm:AC:CY}).

\subsection{Hypersurfaces of $6$--manifolds with an \sunitary{3}--structure}

We recall from Conti--Salamon \cite[\S 1]{Conti:Salamon} some basic facts about the geometry of orientable hypersurfaces of a $6$--manifold endowed with an 
\sunitary{3}--structure.

Let $M^{6}$ be a $6$--manifold endowed with an \sunitary{3}--structure $(\omega, \Omega)$. An orientable hypersurface $N^{5} \hookrightarrow M$ is naturally endowed with an \sunitary{2}--structure, \ie an $\sunitary{2}$--reduction of the frame bundle of $N$. This is equivalent to the existence of a quadruple $(\eta, \omega _{1}, \omega _{1}, \omega _{3})$ where $\eta$ is a nowhere-vanishing $1$--form and $\omega_i$ are $2$--forms on $N$ satisfying 
\begin{enumerate}[label=(\roman*)]
\item \label{sd:triple} $\omega _{i} \wedge \omega _{j} = \delta _{ij}v$, where $v$ is a fixed $4$--form such that $\eta \wedge v \neq 0$;
\item \label{su2:orientation} $X \lrcorner\, \omega _{1}=Y\lrcorner\, \omega _{2} \Rightarrow \omega _{3}(X,Y) \geq 0$.
\end{enumerate}
The quadruple $(\eta,\omega_1,\omega_2,\omega_3)$ is given in terms of the \sunitary{3}--structure $(\omega,\Omega)$ and the unit normal $\nu$ to $N$  via
\begin{equation}
\label{eq:su2:hyper}
\eta = -\nu \lrcorner\, \omega, \qquad \omega _{1} = \omega|_{N}, \qquad \omega _{2}+i\omega _{3} = -i\nu \lrcorner \,\Omega.
\end{equation}
Conversely, given an \sunitary{2}--structure on $N$ we can define an \sunitary{3}--structure $(\omega,\Omega)$ on $N \times \R$ by
\[
\omega = \omega_1 + \eta \wedge dt, \qquad \Omega = (\omega_2 + i \omega_3) \wedge (\eta + i dt),
\]
where $t$ is a coordinate on $\R$.

Since $\sunitary{2} < \sorth{4} = \sunitary{2}^{+}\cdot \sunitary{2}^{-} < \sorth{5}$ there is a unique metric $g$ and orientation  on $N$ compatible with any $\sunitary{2}$--structure. At each point $x \in N$ the nowhere-zero $1$--form $\eta$ defines a splitting $T_{x}N \simeq \R \oplus \ker{\eta _{x}}$. The metric $g$ and orientation on $N$ determine a metric and orientation on the 
$4$--plane field $H:=\ker{\eta}$  and hence the space of ``horizontal'' $2$--forms $\Lambda ^{2}H^{\ast}$ splits as a direct sum of self-dual and anti-self-dual horizontal forms: 
$\Lambda ^{2}H^{\ast}=\Lambda^+ \oplus \Lambda^-$.
A triple $(\omega _{1},\omega _{2},\omega _{3})$ satisfying \ref{sd:triple} determines a trivialisation of $\Lambda ^{+}$ and therefore a reduction of the structure group from $\sorth{4}=\sunitary{2}^{+}\cdot \sunitary{2}^{-}$ to $\sunitary{2}^{-}$. In fact, we can always assume that $(\omega _{1},\omega _{2},\omega _{3})$ is an oriented basis of $\Lambda ^{+}$ with respect to the natural orientation induced from the orientation of $\ker{\eta}$; this gives condition \ref{su2:orientation} above.

\begin{remark}
\label{rmk:discrete:sym:su2}
For future reference we make the elementary observations that if $(\eta,\omega_1,\omega_2,\omega_3)$ is an \sunitary{2}--structure on $N$ then so are the quadruples
\begin{subequations}
\label{eq:involutions}
\begin{gather}
\label{eq:tau1}
\tau_1(\eta,\omega_1,\omega_2,\omega_3) := (-\eta,\omega_1,-\omega_2,-\omega_3),\\
\intertext{and}
\label{eq:tau2}
\tau_2(\eta,\omega_1,\omega_2,\omega_3):= (-\eta,-\omega_1,\omega_2,-\omega_3).
\end{gather}
\end{subequations}
The involutions $\tau_1$ and $\tau_2$ have the following interpretations at the level of \sunitary{3}--structures. 
First note that if $(\omega,\Omega)$ is an \sunitary{3}--structure then so is $(-\omega,-\overline{\Omega})$. 
Given an oriented hypersurface $(N,\nu) $ of the \sunitary{3}--manifold $(M,\omega,\Omega)$ 
besides the \sunitary{2}--structure $(\eta,\omega_1,\omega_2,\omega_3)$ defined by \eqref{eq:su2:hyper} 
we can consider two  alternative $\sunitary{2}$--structures on $N$:
one in which we endow $N$ with the opposite orientation $-\nu$ and $M$ with its original \sunitary{3}--structure
and the other in which we 
endow $N$ with its original orientation $\nu$ and $M$ with the \sunitary{3}--structure $(-\omega,-\overline{\Omega})$. Notice however that both symmetries change the orientation of the hypersurface $N \subset M$. Indeed, in the latter case changing $\omega$ into $-\omega$ changes the orientation on $M$ and therefore, since $\nu$ is kept fixed, the one on $N$. Using \eqref{eq:su2:hyper} we see that these two \sunitary{2}--structures differ from the original one by the action of the involutions $\tau_1$ and $\tau_2$ respectively. 
All three \sunitary{2}--structures clearly induce the same Riemannian metric on $N$.
\end{remark}

\begin{remark}
\label{rmk:su2:spinors}
Conti--Salamon also explain how to understand \sunitary{2}--structures on $5$--manifolds in terms of spin geometry. Underlying this is the low dimensional isomorphism $\spin{5} \cong \Sp{2}$ and the fact 
that the spinor representation of \spin{5} is isomorphic to the fundamental representation of \Sp{2} on $\HH^2$. Hence the isotropy subgroup of a nonzero spinor $\psi$ in five dimensions
 is isomorphic to $\Sp{1} \cong \sunitary{2}$.
It follows that an \sunitary{2}--structure on a $5$--manifold $N$ is equivalent to the choice of a spin structure on $N$ and a unit spinor. 
Further aspects of the geometry of \sunitary{2}--structures, including their intrinsic torsion also have formulations in terms of spin geometry. 
In this paper we find it most convenient to phrase everything in terms of differential forms rather than spinors.
However sometimes for a compact notation it will be convenient to refer to an \sunitary{2}--structure on a $5$--manifold $N$ by $\psi$,  
the corresponding \sunitary{3}--structure on $N \times \R$ by $\Psi$ and the restriction of $\Psi$ to $N \times \{t\}$ by $\psi_t$.
\end{remark}

\subsubsection{Conical \sunitary{3}--structures} 
We will be interested in conical, asymptotically conical and conically singular \sunitary{3}--structures, mainly those of Calabi--Yau or \nk type.
Consideration of conical Calabi--Yau structures leads us to the first important special class of \sunitary{2}--structures: \emph{Sasaki--Einstein} structures.

The (smooth) metric cone $C(N)$ over a smooth Riemannian manifold $(N,g_N)$ is the noncompact manifold $\R^+ \times N$ endowed with the incomplete Riemannian metric 
$g_C = dr^2 + r^2 g_N$, where $r>0$ denotes the (radial) coordinate on $\R^+$.
Smooth metric cones provide the local models for the simplest isolated singularities of Riemannian metrics. Any smooth metric cone $C(N)$  admits a $1$--parameter family of dilations preserving the cone metric $g_C$. If $N$ possesses  additional geometric structure, \eg a $G$--structure, we can make the additional demand that the $1$--parameter family of dilations of the metric cone $C(N)$ act on the extra geometric structure in the obvious way.

Motivated by this,  given an \sunitary{2}--structure $(\eta, \omega_i)$ on a $5$--manifold $N$ we define a conical $\sunitary{3}$--structure on $C(N)$ via
\begin{equation}
\label{eq:conical:su2}
\omega_C = r \eta \wedge dr + r^2 \omega_1, \quad \Omega_C = r^2 (\omega_2+i \omega_3) \wedge (r\eta + i dr).
\end{equation}
The metric induced by the conical \sunitary{3}--structure $(\omega_C,\Omega_C)$ is the cone metric $g_C$ associated with the Riemannian metric $g_N$ 
induced by the \sunitary{2}--structure on $N$.

\begin{definition}
\label{def:SE}
An \sunitary{2}--structure $(\eta, \omega_1, \omega_2, \omega_3)$ on a $5$--manifold $N$ is called \emph{Sasaki--Einstein} if it satisfies
\begin{equation}\label{eqn:Sasaki:Einstein}
d\eta = -2 \omega_1, \qquad d\omega_2 = 3\eta \wedge \omega_3, \qquad d\omega_3 = -3 \eta \wedge \omega_2.
\end{equation}
Equation \eqref{eqn:Sasaki:Einstein} is equivalent to requiring that the conical \sunitary{3}--structure defined by \eqref{eq:conical:su2} 
be Calabi--Yau, \ie $d\omega_C=d \Omega_C=0$.
\end{definition}

\begin{remark*}
The involution $\tau_2$ defined in \eqref{eq:tau2} preserves the \se equations while the involution $\tau_1$ defined in \eqref{eq:tau1} reverses the signs in all three equations in  \eqref{eqn:Sasaki:Einstein}. Furthermore, the complex $2$--form $\omega_2+i\omega_3$ can be multiplied by any complex number of unit norm while preserving \eqref{eqn:Sasaki:Einstein}. In other words, $(\omega_C,e^{i\theta}\Omega_C)$ for $e^{i\theta} \in \Sph^1$ define different conical Calabi-Yau structures on $C(N)$ inducing the same cone metric.
\end{remark*}

\subsection{$\sunitary{2} \times \sunitary{2}$--invariant \sunitary{2}--structures on $S^2 \times S^3$}

Consider now a (not necessarily complete or compact) $6$--manifold $M$ with an \sunitary{3}--structure preserved by a cohomogeneity one isometric action of a compact Lie group $G$. For reasons explained at the beginning of Section \ref{sec:Coho1:NK}, we will assume $G=\sunitary{2} \times \sunitary{2}$ and that a dense open set $M^\ast\subset M$ is diffeomorphic to the product of an interval with the $5$--dimensional homogeneous space $\Ndiag = \sunitary{2}\times\sunitary{2}/\triangle\unitary{1} \simeq S^2\times S^3$. In the spirit of Hitchin's work \cite{Hitchin:stable:forms}, we will regard the \sunitary{3}--structure on $M^\ast$ as a $1$--parameter family of \sunitary{2}--structures on $\Ndiag$ invariant under the action of $\sunitary{2}\times\sunitary{2}$.

Fix a basis $H, E, V$ of $\Lie{su}(2)$ with Lie brackets
\[
[H,E]=V, \qquad [H,V]=-E, \qquad [E,V]=\tfrac{1}{2}H.
\]
Let $U^+=(H,H)$ be the generator of the Lie algebra of $\Delta \unitary{1}$. The vectors
\begin{equation}\label{eqn:Invariant:vector:fields}
U^-=(H,-H), \qquad E_{1}=(E,0), \qquad E_{2}=(0,E), \qquad V_{1}=(V,0), \qquad V_{2}=(0,V),
\end{equation}
on $\Lie{su}(2) \oplus \Lie{su}(2)$ define left-invariant vector fields on $\Ndiag$. Let $u^-, e_{1}, e_{2}, v_{1}, v_{2}$ be the corresponding co-frame.

There is a distinguished $\sunitary{2} \times \sunitary{2}$--invariant \sunitary{2}--structure on $\Ndiag$: the Sasaki--Einstein structure on $S^2 \times S^3$ that gives rise to the conical Calabi--Yau metric on the \emph{conifold} $\{ z_{1}^{2}+z_{2}^{2}+z_{3}^{2}+z_{4}^{2} =0 \}$. In terms of the above basis the standard Sasaki--Einstein structure on $\Ndiag$ is
\begin{equation}\label{eq:invt:Sasaki:Einstein}
\etase = \frac{2}{3}u^-, \quad \omegase_1= \frac{1}{12}(e_{1} \wedge v_{1}-e_{2}\wedge v_{2}), \quad \omegase_2 = \frac{1}{12}(e_{1} \wedge v_{2} + e_{2} \wedge v_{1}), \quad \omegase_3 = \frac{1}{12}(e_{1} \wedge e_{2} + v_{1} \wedge v_{2}).
\end{equation}
Observe that \eqref{eqn:Sasaki:Einstein} are not scale-invariant, so the numerical factors here are forced on us once we fix a basis of $\Lie{su}(2)$. 

There is in fact a circle of invariant Sasaki--Einstein structures inducing the same metric due to the freedom of changing the phase of the complex $2$--form $\omegase_2+i\omegase_3$. These are all equivalent because the flow of the Reeb vector field $U^-$ preserves $\etase$ and $\omegase_1$ and acts as a rotation in the $(\omegase_2,\omegase_3)$--plane.

\enlargethispage{2\baselineskip}
\begin{lemma}\label{lem:invariant:forms}
$\etase$ is the unique (up to scale) invariant $1$--form on $\Ndiag$. In particular, the distribution $\ker{\eta}$ is independent of the choice of invariant \sunitary{2}--structure on $\Ndiag$.

Fix the volume form $v=\omegase_1\wedge\omegase_1$ on $\ker{\etase}$. The space of invariant self-dual $2$--forms on $\ker{\etase}$ is $3$--dimensional, spanned by $\omegase_1, \omegase_2, \omegase_3$, and there exists a unique invariant anti-self-dual $2$--form on $\ker{\etase}$ up to scale,
\[
\omegase_0=\frac{1}{12}(e_{1}\wedge v_{1} + e_{2} \wedge v_{2}).
\]
Furthermore, $\omegase_0$ is closed.
\proof
Write $\lspan (U^-,E_{1},E_{2},V_{1},V_{2})$ as $\R U^- \oplus \Lie{n}_{1} \oplus \Lie{n}_{2}$, where $\Lie{n}_{i}=\lspan (E_{i},V_{i})$. Then $\Lie{n}_{i} \simeq \Lie{n}$ as $\Delta \unitary{1}$--representations, where $\Lie{n}$ is the complex $1$--dimensional representation of \unitary{1} with weight $1$. Therefore as \unitary{1}--representations $\Lambda ^{2}\Lie{n}_{i}^{\ast} \simeq \R (e_{i} \wedge v_{i})$ and $\Lie{n}_{1} \otimes \Lie{n}_{2} \simeq \Lambda ^{2}\Lie{n}^{\ast} \oplus \R \text{id} \oplus \text{Sym}_{0}(\Lie{n})$, where $\text{Sym}_{0}(\Lie{n})$ is the complex $1$--dimensional representation of weight $2$.

$\Ndiag$ is a circle bundle over $S^2 \times S^2$. The anti-self-dual $2$--form $\omegase_0$ is a multiple of the pull-back of the K\"ahler--Einstein metric on $\BCP^1\times\BCP^1$ and is therefore closed.
\endproof
\end{lemma}

Endow the $4$--dimensional vector space of invariant $2$--forms on $\ker{\etase}$ with the inner product $\langle \omega, \omega' \rangle v = \omega \wedge \omega '$ of signature $(1,3)$.

\begin{prop}\label{prop:invariant:SU(2):structures}
The space of invariant $\sunitary{2}$--structures on $\Ndiag$ inducing the same orientation as the standard Sasaki--Einstein structure can be identified with $\R ^{+} \times \R ^{+} \times \lorentz{3}$.
\proof
If $(\eta, \omega_i)$ is an invariant \sunitary{2}--structure on $\Ndiag$ then 
\begin{equation}
\label{eq:eta:invt}
\eta = \lambda \etase
\end{equation}
for some $\lambda \neq 0$. The choice of orientation implies that $\lambda >0$. The $2$--forms $\omega_1, \omega_2, \omega_3$ yield an oriented orthogonal basis of the space-like $3$--dimensional subspace $\lspan{(\omega _{1},\omega _{2},\omega _{3})}$ of $\Lambda ^{2}(\ker{\etase})^{\ast}$. Moreover $|\omega_i|$ is independent of $i$. Therefore up to scale the triple $(\omega_1, \omega_2, \omega_3)$ represents a point in the Lorentzian Stiefel manifold $V_{0}(3;\R ^{1,3}) \simeq \lorentz{3}$. Indeed, we can always complete $(\omega _{1},\omega _{2},\omega _{3})$ to a basis of $\R ^{1,3}$ with the unique invariant $2$--form $\omega _{0}$ satisfying $\langle \omega _{0}, \omega _{i} \rangle =0$, $|\omega _{0}|^{2}=-|\omega _{i}|^{2}$ and $\langle \omega _{0}, \omegase_0 \rangle >0$. Then 
\begin{equation}
\label{eq:omega:invt}
(\omega _{0}, \omega_1, \omega_2, \omega _{3})= \mu A\, (\omegase _{0}, \omegase_1, \omegase_2, \omegase _3)
\end{equation}
for some $\mu >0$ and $A \in \lorentz{3}$. Here $A\in \lorentz{3}$ is interpreted as an endomorphism of the space of invariant $2$--forms on $\ker{\etase}$ with respect to the basis $\{\omegase _{0}, \omegase_1, \omegase_2, \omegase _3\}$.
\endproof
\end{prop}

\begin{remark}
\label{rmk:invt:su2:para}
Given any $(\lambda, \mu, A) \in \R ^{+} \times \R ^{+} \times \lorentz{3}$ 
we will denote the invariant \sunitary{2}--structure $(\eta, \omega_1, \omega_2, \omega_3)$ satisfying \eqref{eq:eta:invt} and \eqref{eq:omega:invt} by $\psi_{\lambda,\mu,A}$.
The choice of notation $\psi$ here is motivated by the spinor reformulation of $\sunitary{2}$--structures alluded to in Remark \ref{rmk:su2:spinors} and the desire for 
a compact notation. From now on we make the standing assumption that our invariant \sunitary{2}-structures satisfy $\lambda>0$, \ie $\psi_{\lambda,\mu,A}$ induces the same orientation as the Sasaki--Einstein structure.
\end{remark}

\begin{remark}
\label{rmk:involution:invt}
The involutions $\tau_1$ and $\tau_2$ defined in \eqref{eq:involutions} act on the set of invariant \sunitary{2}--structures. Adopting the notation of the previous remark $\tau_1$ and $\tau_2$ act as follows:
\[
\tau_1^* \psi_{\lambda, \mu, A} = \psi_{-\lambda, \mu, A T_1}, \quad \tau_2^* \psi_{\lambda, \mu, A} = \psi_{-\lambda, \mu, A T_2}, 
\]
where $T_1, T_2 \in \lorentz{3}$ are defined by $T_1:=\diag{(1,-1,1,-1)}$ and $T_2:=\diag{(1,1,-1,-1)}$. Observe that neither $\tau_1$ nor $\tau_2$ preserves the normalisation $\lambda>0$, but their composition does.
\end{remark}

The left-invariant vector field $U^-$ generates the group of inner automorphisms of $\sunitary{2} \times \sunitary{2}$ that fix $\Delta \unitary{1}$. $U^-$ induces a circle action on the space of invariant \sunitary{2}--structures given by a rotation in the $(\omegase_2, \omegase_3)$--plane. There are also discrete symmetries, \cf Proposition \ref{prop:symmetries} below.

Any invariant \sunitary{2}--structure on $\Ndiag$ determines uniquely an invariant metric. The next result describes  this map explicitly in terms of the parametrisation of invariant \sunitary{2}--structures given in Proposition \ref{prop:invariant:SU(2):structures}.

\enlargethispage{\baselineskip}

\begin{prop}[\mbox{\cf \cite[Lemma 4.1]{Alekseevsky:Dotti:Ferraris}}]\label{prop:Invariant:metrics}\hfill{}
\begin{enumerate}[label=(\roman*)]
\item
The set of invariant metrics on $\Ndiag$ is parametrised by $\R^+ \times \R^+ \times S^+$ where
\[S^+:= \{ w \in \R^{1,3} :  \norm{w}^2=-1, \, w_0>0\}
\]
is the upper hyperboloid in $\R^{1,3}$. We denote the corresponding invariant metric $g_{\lambda,\mu,w}$.
\item
The invariant metrics corresponding to $g_{\lambda, \mu, w}$ and $g_{\lambda, \mu, w'}$, where 
\[w'=(w_{0},w_{1},\cos{\theta}\, w_{2}-\sin{\theta}\, w_{3}, \sin{\theta}\, w_{2}+\cos{\theta}\, w_{3}),\] are isometric.
\item
The map from the invariant \sunitary{2}--structure $\psi_{\lambda,\mu,A}$ to its invariant metric $g_{\lambda,\mu,w}$ 
is given by 
\[
(\lambda, \mu, A) \mapsto (\lambda,\mu,pr_1(A)),
\]
where $pr_1(A) \in S^+$ is the projection of the matrix $A \in \lorentz{3}$ onto its first column.
In particular the set of invariant \sunitary{2}--structures is a principal \sorth{3}--bundle over the space of invariant metrics.
\end{enumerate}
\end{prop}
\begin{proof}
Denote by $\Delta\Lie{u}(1) ^\perp$ the orthogonal complement of  $\Delta\Lie{u}(1)$ in $\Lie{su}_{2} \oplus \Lie{su}_{2}$. In the proof of Lemma \ref{lem:invariant:forms} we have already observed that 
\[
\Delta\Lie{u}(1)^\perp = \lspan{(U^-,E_{1},V_{1},E_{2},V_{2})} \simeq \R \oplus \Lie{n}_{1} \oplus \Lie{n}_{2}
\]
where $\Lie{n}_{i}$ is isomorphic to the standard representation $\Lie{n}$ of \unitary{1}. 
An invariant metric $g$ on $\Ndiag$ can be written as 
\[
g=\lambda ^{2} \etase \otimes \etase + g^{T},
\] where the transverse metric $g^T$ can be thought of as a \unitary{1}--invariant inner product on $\Lie{n}_{1} \oplus \Lie{n}_{2}$. When $g$ is induced by an invariant \sunitary{2}--structure $\psi_{\lambda,\mu,A}$ then $g(U^-,U^-)=\eta(U^-)^2$, which motivates the notation. Furthermore, in this case we also define transverse almost complex structures $J_{i}$ such that $g^{T}(u,v) = \omega_i(u,J_i v)$.

The main observation is that the \unitary{1}--invariance of $g^T$ forces additional structure on the transverse geometry. Indeed, $\Jse_{0}=[U^+,\,\cdot\,]$ defines a complex structure on $\Lie{n}_{1} \oplus \Lie{n}_{2}$ with $\Jse_{0}E_{i}=V_{i}$. Since $U^+$ generates the action of $\Delta \unitary{1}$, \unitary{1}--invariant endomorphisms of $\Lie{n}_{1} \oplus \Lie{n}_{2}$ are precisely those commuting with $\Jse_0$. In particular, $\Jse_0J_i = J_i\Jse_0$ and $g^{T}$ is an almost-Hermitian metric with respect to $\Jse_0$. Therefore we can define the associated Hermitian form $\omega_0 (X,Y) = g^{T}(\Jse_0 X,Y)$. Observe also that $\Jse_0$ induces the opposite orientation with respect to the volume form $v=\omegase_{1}\wedge\omegase_{1}$ and therefore $\omega_0$ is an anti-self-dual form. It follows that the map from invariant \sunitary{2}--structures to invariant metrics is surjective and is a principal $\sorth{3}$--bundle as claimed. The explicit expression for $\omega_0$ in terms of the parametrisation $(\lambda,\mu,A)$ is 
\[
\omega _{0} =\mu A\omegase_0= \mu \left( w_{0}\,\omegase_0 + w_{1}\,\omegase_1 + w_{2}\,\omegase_2 + w_{3}\,\omegase_3 \right),
\]
where $w=(w_i)$ is the first column of the matrix $A \in \lorentz{3}$.

The circle action in (ii) is the induced action of the flow of the Reeb vector field $U^-$ on invariant metrics.
\end{proof}

We are interested in \sunitary{2}--structures induced on a hypersurface of a Calabi--Yau or nearly K\"ahler $6$--manifold; these have been dubbed \emph{hypo} and \emph{nearly hypo} structures, respectively.  We will study the $\sunitary{2} \times \sunitary{2}$--invariant ones on $\Ndiag$. As we will see in a moment, hypo and nearly hypo structures can be defined by a set of constraints on the exterior differentials of the forms defining the \sunitary{2}--structure. By \eqref{eqn:Sasaki:Einstein} and Lemma \ref{lem:invariant:forms}, for the invariant \sunitary{2}--structure $\psi_{\lambda, \mu, A}$  we have
\begin{equation}\label{eqn:Torsion:invariant:structures}
d\eta = -2\lambda \omegase_1, \qquad d\omega _{i} = -2\tfrac{\mu}{\lambda} \eta \wedge TA\omegase_i, \qquad d(\eta \wedge \omega _{i}) = -2\lambda \mu \langle A\omegase_i, \omegase_i \rangle v,
\end{equation}
where
$T \in \text{End}(\R ^{1,3})$ is
\begin{equation}\label{eqn:Torsion:standard:SE}
T= \left( \begin{array}{cccc}
0 & 0 & 0 & 0\\
0 & 0 & 0 & 0\\
0 & 0 & 0 & -3\\
0 & 0 & 3 & 0\\
\end{array} \right).
\end{equation}

\begin{remark}\label{rmk:invt:SE}
As an immediate consequence we see that the invariant Sasaki--Einstein structures on $\Ndiag$ are those $\psi_{\lambda,\mu,A}$ with $\lambda=\mu=1$ and $A \in \sorth{2} \subset \lorentz{1} \times \sorth{2} \subset \lorentz{3}$. Indeed, the equations \eqref{eqn:Sasaki:Einstein} are not scale invariant and the only geometric degree of freedom is to rotate the form $\omega _{2} + i \omega _{3}$ by a phase $e^{i\theta}$. This freedom of rotation in the plane spanned by $\omegase_2$ and $\omegase_3$, as already noted, is nothing but the action of the flow of the Reeb vector field on invariant $2$--forms; in general the flow of the Reeb vector field on invariant metrics on $\Ndiag$ is nontrivial but in the case of the Sasaki--Einstein metric $\gse$  the Reeb vector field is an additional Killing field. In particular $\gse$ is invariant under the larger group $\unitary{1} \times \sunitary{2} \times \sunitary{2}$.
\end{remark}

\subsection{Hypo structures}
In this section, following Conti--Salamon \cite[Definition 1.5]{Conti:Salamon}, we consider the class of $\sunitary{2}$--structures that arise on oriented 
hypersurfaces in Calabi--Yau $3$--folds. To this end consider a $1$--parameter family of \sunitary{2}--structures $\big( \eta , \omega _{i}\big)(t)$ such that 
\begin{equation}\label{eqn:SU(2):to:SU(3):structures}
\omega = \eta \wedge dt + \omega _{1}, \qquad \Omega = (\omega _{2}+i\omega _{3}) \wedge (\eta + idt),
\end{equation}
is a Calabi--Yau structure on $N \times I$ for some interval $I \subset \R$, \ie an \sunitary{3}--structure such that $\omega$ and $\Omega$ are both closed. The condition $d\omega =0=d\Omega$ is equivalent to
\begin{subequations}
\begin{equation}\label{eqn:Hypo}
d\omega _{1} =0, \qquad d( \eta \wedge \omega _{2} )=0, \qquad d( \eta \wedge \omega _{3} )=0,
\end{equation}
together with the evolution equations
\begin{equation}\label{eqn:Hypo:evolution}
\partial _{t}\omega _{1}=-d\eta, \qquad \partial _{t}(\eta \wedge \omega _{2}) = -d\omega _{3}, \qquad \partial _{t}(\eta \wedge \omega _{3}) = d\omega _{2}.
\end{equation}
\end{subequations}
\begin{definition}
\label{def:hypo}
An \sunitary{2}--structure $(\eta, \omega_1, \omega_2, \omega_3)$ on a $5$--manifold $N$ satisfying \eqref{eqn:Hypo} is called a \emph{hypo} structure. We call equations \eqref{eqn:Hypo:evolution} the \emph{hypo evolution equations}. 
\end{definition}
\begin{remark*}
The involutions $\tau_1$ and $\tau_2$ defined in Remark \ref{rmk:discrete:sym:su2} both preserve the hypo equations. 
Moreover, $(\eta,\omega_1,\omega_2,\omega_3)(t)$ solves the hypo evolution equations 
if and only if $\tau_2(\eta,\omega_1,\omega_2,\omega_3)(t)$ does (and similarly for 
$\tau_1(\eta,\omega_1,\omega_2,\omega_3)(-t)$). Furthermore, \eqref{eqn:Hypo} and \eqref{eqn:Hypo:evolution} are both invariant under changing the phase of the complex $2$--form $\omega_2+i\omega_3$.
\end{remark*}

We now describe the space of \emph{invariant} hypo structures on $\Ndiag$. Recall again the notation $\psi_{\lambda,\mu,A}$ where $(\lambda,\mu,A) \in \R^+ \times \R^+ \times \lorentz{3}$ adopted in Remark \ref{rmk:invt:su2:para} to describe invariant \sunitary{2}--structures on $\Ndiag$. 

\begin{prop}
\label{prop:invariant:hypo:structures}
For any $(f,g) \in \R^+ \times \R^+$ the invariant hypo structures on $\Ndiag$ are invariant under the rescalings
\[
\eta \mapsto f \eta, \quad \omega_i \mapsto g\, \omega_i \quad\text{\ for \ } i=1,2,3.
\]
In particular by rescaling it suffices to describe the invariant hypo structures with $\lambda=\mu=1$.

The set of invariant hypo structures on $\Ndiag$ with $\lambda=\mu=1$ is the union of two smooth manifolds:
\begin{enumerate}
\item
A two-dimensional manifold diffeomorphic to $\textup{S}\left( \orchro{1} \times \orth{2} \right) \subset \lorentz{3}$.
\item 
A three-dimensional manifold diffeomorphic to $\orchro{2}$ embedded in $\lorentz{3}$ as the subgroup that fixes the line spanned by $(0,1,0,0)$.
\end{enumerate}
In each case there are two connected components, which are interchanged by $\tau_1 \circ \tau_2$. The two manifolds intersect in two circles, the one parametrising the $1$--parameter family of invariant Sasaki--Einstein structures on $\Ndiag$ described in Remark \ref{rmk:invt:SE} and its image under $\tau_1 \circ \tau_2$. 

\end{prop}
\begin{proof}
Let $\psi_{\lambda,\mu,A}$ be an invariant hypo structure on $\Ndiag$. 
Observe first that the hypo equations \eqref{eqn:Hypo} are invariant under $\eta \mapsto f \eta$ and $\omega _{i} \mapsto g\, \omega _{i}$ for any $f>0, g>0$. 
Therefore without loss of generality we now assume that $\lambda=\mu=1$. Writing $\omega _{i}=A\,\omegase_i$ and using \eqref{eqn:Torsion:invariant:structures} we see that \eqref{eqn:Hypo} is satisfied if and only if
$A$ is of the form
\begin{equation}
\label{eqn:invt:hypo:A}
A= \left( \begin{array}{cccc}
w_{0} & x_{0} & y_{0} & z_{0}\\
w_{1} & x_{1} & 0     &  0 \\
w_{2} & 0     & y_{2} & z_{2}\\
w_{3} & 0     & y_{3} & z_{3}\\
\end{array} \right).
\end{equation}
Imposing the condition that $A \in \lorentz{3}$ leads us to distinguish two cases: $y_0=z_0=0$ or otherwise. 
In the former case, up to the action of $\tau_1\circ\tau_2$, \eqref{eqn:invt:hypo:A} together with $A  \in \lorentz{3}$ forces
\begin{equation}
\label{eqn:invt:hypo:A1}
A= \left( \begin{array}{cccc}
\cosh{s}  & \sinh{s} & 0  & 0\\
\sinh{s} & \cosh{s} & 0  &  0 \\
0 & 0 & \cos{\theta} & -\sin{\theta}\\
0 & 0 & \sin{\theta} & \cos{\theta}\\
\end{array} \right),
\end{equation}
while in the latter we must have
\begin{equation}
\label{eqn:invt:hypo:A2}
A= \left( \begin{array}{cccc}
w_{0} & 0 & y_{0} & z_{0}\\
0     & 1 & 0     &  0 \\
w_{2} & 0 & y_{2} & z_{2}\\
w_{3} & 0 & y_{3} & z_{3}\\
\end{array} \right) \in \lorentz{2}.
\end{equation}
The final statement about the intersection of the two components follows immediately. 
\end{proof}

We are now going to solve the hypo evolution equations \eqref{eqn:Hypo:evolution} in the two components (i) and (ii) of invariant hypo structures on $\Ndiag$ given in the previous Proposition. The invariance assumption 
reduces these evolution equations to ODEs on the space of invariant hypo structures that we will solve explicitly. The limiting case corresponding to the intersection of the manifolds of Proposition \ref{prop:invariant:hypo:structures} is of course the conifold, the Calabi--Yau cone over the \se structure \eqref{eq:invt:Sasaki:Einstein}. Up to scale there exist two $\sunitary{2} \times \sunitary{2}$--invariant complete  Calabi--Yau metrics 
asymptotic to the conifold: one on the small resolution of the conifold \cite{Candelas:delaOssa}, the total space of the vector bundle $\oo (-1) \oplus \oo (-1)$ over $\PP^1$, and one on the smoothing of the conifold \cite{Candelas:delaOssa,Stenzel}, diffeomorphic to $T^\ast S^3$. These have been recently proven to be the unique complete asymptotically conical Calabi--Yau metrics with tangent cone at infinity the conifold  \cite{Conlon:Hein:III}. 

\begin{theorem}[Candelas--de la Ossa \cite{Candelas:delaOssa}; Stenzel \cite{Stenzel} only for part (ii)]
\label{thm:AC:CY}
$ $

\begin{enumerate}
\item Up to scale there exists a unique smooth invariant Calabi--Yau structure on the total space of $\oo (-1) \oplus \oo (-1)$ over $\PP^1$.
\item
Up to scale there exists a unique smooth invariant Calabi--Yau structure on $T^\ast S^3$.
\end{enumerate}
Both Calabi--Yau structures are complete and asymptotic to the conifold in the sense made precise below.
\end{theorem}

We give a detailed outline of the proof of the theorem using the language of invariant hypo structures, partly as a warm-up for the more complicated analysis in the \nk case, and because these asymptotically conical Calabi--Yau manifolds will play a role in that analysis as limiting objects.

\proof[Proof of Theorem \ref{thm:AC:CY}(i)]
We begin with the complete invariant Calabi--Yau structures that arise form the invariant hypo structures 
described in Proposition \ref{prop:invariant:hypo:structures}(i), \ie where $y_0=z_0=0$.

By acting with the flow of the Reeb vector field, \ie choosing $\theta=-\tfrac{\pi}{2}$ in equation \eqref{eqn:invt:hypo:A1}, 
we can always assume that a hypo structure with $y_{0}=0=z_{0}$ satisfies
\begin{equation}\label{eq:hypo:Stenzel}
\eta = \lambda \etase, \qquad \omega_1 = u_{0}\, \omegase_0 + u_{1}\,  \omegase_1, \qquad \omega_2 = -\mu\, \omegase_3, \qquad \omega_3 = \mu\,\omegase_2,
\end{equation}
with $$-u_{0}^{2}+u_{1}^2=\mu ^{2}.$$ Then the flow equations \eqref{eqn:Hypo:evolution} are equivalent to 
\begin{equation}\label{eqn:Small:resolution:odes}
\dot{u}_{0}=0, \qquad \dot{u}_{1}=2\lambda, \qquad \partial _{t}(\lambda \mu) = 3\mu.
\end{equation}
The solution with $u_0=0$ is the conifold. When $u_0 \neq 0$, by scaling and the discrete symmetry $\tau_4$ of Proposition \ref{prop:symmetries} that exchanges the two factors of $\sunitary{2} \times \sunitary{2}$, we can assume without loss of generality that $u_0=1$. It is convenient to introduce a new variable $r$ such that $u_1=r^2$. Here up to the action of $\tau_1\circ\tau_2$ we can always assume $u_1 \geq 0$. Hence $\mu^2=-u_{0}^{2}+u_{1}^2 = r^4-1$ and $\lambda = r\dot{r}$. 
In terms of the new variable $r$ from \eqref{eqn:Small:resolution:odes} we obtain 
\[
\frac{d}{dr} (\lambda \mu)^2 = 2 \lambda \mu \frac{d}{dr}(\lambda \mu) = \frac{6\lambda\mu^2}{\dot{r}} = 6r\mu^2 = 6r^5-6r.
\]
In Section \ref{sec:closing} we will discuss necessary conditions for a cohomogeneity one \sunitary{3}--structure to extend smoothly over a singular orbit. In view of Lemma \ref{lem:Smooth:extension:S2}, we require that, with respect to the variable $t$, $u_1(0)=1$ and $\lambda (0)=0$. Thus $r \geq 1$ and 
\begin{equation}\label{eqn:Small:resolution}
\mu = \sqrt{r^4-1}, \qquad \lambda = \sqrt{\frac{r^6 - 3r^2 +2}{r^4-1}}.
\end{equation}

In particular as $r\ra\infty$ we have 
\begin{subequations}
\label{eqn:small:resn:coeffs:expand}
\begin{gather}
\lambda \mu = r^3 \sqrt{1-3r^{-4}+2r^{-6}} = r^3 + O(r^{-1}), \\
\frac{r\mu}{\lambda} = (r^2-r^{-2})(1-3r^{-4}+2r^{-6})^{-1/2} = r^2 + O(r^{-2}).
\end{gather}
\end{subequations}

For $r\sim 1$ using $\lambda=r\dot{r}$ to transform back to the variable $t$ yields $t\sim \frac{2}{\sqrt{3}}\sqrt{r-1}$ and therefore
\[
u_0 -u_1 = 1 -\tfrac{3}{2}t^2 + O(t^4), \qquad \lambda = \tfrac{3}{2}t +O(t^3).
\]
Lemma \ref{lem:Smooth:extension:S2}(i) guarantees that the  resulting cohomogeneity one Calabi--Yau structure $(\omega,\Omega)$ extends smoothly at $r=1$ over a $2$--sphere $\sunitary{2}\times\sunitary{2}/\unitary{1}\times\sunitary{2}$. Moreover $(\omega,\Omega)$ is asymptotic as $r \ra \infty$ to the conifold. 
To see this explicitly recall from \eqref{eqn:SU(2):to:SU(3):structures} how the \mbox{\sunitary{3}--structure} $(\omega,\Omega)$ is obtained from 
the $1$--parameter family of \sunitary{2}--structures $(\eta(t), \omega_i(t))$.
Converting from $t$ to $r$ using $\lambda dt = \lambda \tfrac{dr}{\dot{r}} = rdr$
we obtain 
\begin{subequations}
\label{eqn:SU3:small}
\begin{gather}
\omega = \omega_1 + \eta \wedge dt = u_0\, \omegase_0 + u_1\, \omegase_1 + \lambda \etase \wedge dt = 
\omegase_0 + r^2 \omegase_1 + r \etase \wedge dr,\\
\Real{\Omega} = \omega_2 \wedge  \eta - \omega_3 \wedge dt = -\lambda \mu \,\omegase_3 \wedge \etase - \tfrac{\mu r}{\lambda} \omegase_2 \wedge dr,\\
\Imag{\Omega} = \omega_3 \wedge \eta + \omega_2 \wedge dt = \lambda\mu \, \omegase_2 \wedge \etase - \tfrac{\mu r}{\lambda} \omegase_3 \wedge dr.
\end{gather}
\end{subequations}
Substituting the expansions from \eqref{eqn:small:resn:coeffs:expand} into \eqref{eqn:SU3:small} and comparing with \eqref{eq:conical:su2} we obtain
\begin{subequations}
\label{eqn:SU3:small:expansion}
\begin{align}
\omega &= \omega_C + O(r^{-2}),\\
\Real{\Omega} &= \Real{\Omega_C} + O(r^{-4}),\\
\Imag{\Omega} &= \Imag{\Omega_C} + O(r^{-4}),
\end{align}
\end{subequations}
where we used the cone metric to compute norms. We stress that up to scaling and discrete symmetries, \eqref{eqn:SU3:small} is the unique solution to \eqref{eqn:Hypo:evolution} that extends smoothly over a singular orbit $S^2$.
\endproof

\proof[Proof of Theorem \ref{thm:AC:CY}(ii)]
Consider instead the evolution of an invariant hypo structure with $(y_{0}, z_{0}) \neq (0,0)$. By acting with the flow of the Reeb vector field and by changing the phase of $\omega _{2} + i \omega _{3}$, we can always assume that $y_0=y_2=z_3=0$ in \eqref{eqn:invt:hypo:A2} and hence
\[
\eta = \lambda \etase, \qquad \omega_1 = \mu\, \omegase_1, \qquad \omega_2 = -\mu\, \omegase_3, \qquad \omega_3 = v_{0}\, \omegase_0 + v_{2}\,\omegase_2,
\]
with $$-v_{0}^{2}+v_{2}^2=\mu ^{2}.$$ 
Then the \nh evolution equations \eqref{eqn:Hypo:evolution} become
\begin{equation}\label{eqn:Stenzel:odes}
\dot{\mu}=2\lambda, \qquad \partial _{t}(\lambda \mu) = 3v_{2}, \qquad \partial _{t}(\lambda v_{0})=0, \qquad \partial_{t}(\lambda v_{2}) = 3\mu.
\end{equation}
If we introduce a new dependent variable $s$ defined by 
\[
\frac{ds}{dt}=\frac{1}{\lambda}>0,
\]
then we can integrate the resulting system of ODEs explicitly as follows.
In terms of $s$ the ODE system \eqref{eqn:Stenzel:odes} is equivalent to 
\begin{equation}
\label{eqn:Stenzel:odes:s}
\frac{d}{ds}(\mu^3) = 6(\mu \lambda)^2, \quad \frac{d}{ds} (\mu \lambda) = 3\lambda v_2, \quad \frac{d}{ds}(\lambda v_0) = 0, \quad \frac{d}{ds}(\lambda v_2) = 3 \mu \lambda.
\end{equation}
In particular both $\mu \lambda$ and $\lambda v_2$ satisfy the second order ODE $f{''} = 9f$.
Applying Lemma \ref{lem:Smooth:extension:S3} one can determine the various constants of integration that ensure that the resulting 
cohomogeneity one \sunitary{3}--structure extends smoothly over the exceptional orbit  (a totally geodesic round $S^3$) that without loss of generality we assume occurs at $s=0$.
Up to the action of the discrete symmetries of Proposition \ref{prop:symmetries}, the solution of \eqref{eqn:Stenzel:odes:s} takes the form
\[
\lambda v_0 = -\kappa, \qquad \lambda \mu = \kappa \sinh{3s}, \qquad \lambda v_2 = \kappa \cosh{3s}, \qquad \mu ^3 = \kappa^2(\sinh{3s}\cosh{3s}-3s),
\]
where $\kappa$ is a positive real parameter and $s \in [0,+\infty)$. The parameter $\kappa>0$ can be changed by scaling: the choice $\kappa = \tfrac{2}{3}$ is equivalent to the normalisation $\lambda(0)=1$; geometrically this corresponds to the exceptional orbit at $s=0$ being a unit $3$--sphere.

Solving for the coefficients $\lambda$, $\mu$, $v_1$ and $v_2$ we obtain
\begin{equation}\label{eqn:Stenzel}
\begin{gathered}
\lambda = \left( \frac{2}{3}\right) ^{\frac{1}{3}} \frac{\sinh{3s}}{(\sinh{3s}\cosh{3s}-3s)^{\frac{1}{3}}}, \qquad \mu = \left( \frac{2}{3}\right) ^{\frac{2}{3}}(\sinh{3s}\cosh{3s}-3s)^{\frac{1}{3}},\\
v_{0} = -\left( \frac{2}{3}\right) ^{\frac{2}{3}}\frac{ (\sinh{3s}\cosh{3s}-3s)^{\frac{1}{3}} }{ \sinh{3s} }, \qquad v_{2} = \left( \frac{2}{3}\right) ^{\frac{2}{3}}\frac{ (\sinh{3s}\cosh{3s}-3s)^{\frac{1}{3}} }{ \tanh{3s} }.
\end{gathered}
\end{equation}

In order to analyse the asymptotics of the resulting Calabi--Yau structure as $s \rightarrow \infty$, consider the change of variable $r^2 = \mu (s)$ which, since $\omega=d\left( -\frac{\mu}{2}\etase \right)$, defines a symplectomorphism between $(0,\infty) \times \Ndiag$ and the conifold. Thus
\[
r^2 \sim \frac{e^{2s}}{3^{\frac{2}{3}}} \left( 1 + O(s e^{-6s}) \right)^{\frac{1}{3}},
\]
and therefore
\[
\omega = \omega_C, \qquad \Omega = \Omega _C + \xi + O\left( \frac{\log{r}}{r^6} \right),
\]
where $\xi = \frac{2}{3r}\omegase_0 \wedge \left( dr-ir\etase\right) = O(r^{-3})$ and we used the cone metric to compute norms.
\endproof

\enlargethispage{2\baselineskip}
\subsection{Nearly Hypo structures} 
In the previous theorem we described cohomogeneity one Calabi--Yau $3$--folds as the evolution of invariant hypo structures under \eqref{eqn:Hypo:evolution}. Following the same philosophy, in this section we consider the \nk analogue of hypo \sunitary{2}--structures: in order to understand cohomogeneity one \nk manifolds we study the class of \sunitary{2}--structures induced on hypersurfaces in a \nk 6--manifold. 
Fernandez \etal \cite[Definition 3.1]{Fernandez:nearly:hypo} named these \emph{nearly hypo} structures. 

To this end consider a family of \sunitary{2}--structures on $N$ induced by a nearly K\"ahler structure on $N \times \R$. By \eqref{eqn:SU(2):to:SU(3):structures} and the definition of a nearly K\"ahler structure in terms of $(\omega, \Omega)$, the \sunitary{2}--structures on $N$ must then satisfy
\begin{subequations}
\begin{equation}\label{eqn:Nearly:Hypo}
d\omega _{1} = 3\eta \wedge \omega _{2}, \qquad d(\eta \wedge \omega _{3}) = -2\omega _{1}^{2},
\end{equation}
and evolve according to the evolution equations
\begin{equation}\label{eqn:Nearly:Hypo:evolution}
\partial _{t}\omega _{1}=-3\omega _{3}-d\eta, \qquad \partial _{t}(\eta \wedge \omega _{2}) =-d\omega _{3}, \qquad \partial _{t}(\eta \wedge \omega _{3}) = d\omega _{2} + 4\eta \wedge \omega _{1}.
\end{equation}
\end{subequations}
\begin{definition}
An \sunitary{2}--structure $(\eta, \omega_1, \omega_2, \omega_3)$ satisfying \eqref{eqn:Nearly:Hypo} 
is called \emph{nearly hypo}. 
We call equations \eqref{eqn:Nearly:Hypo:evolution} the \emph{nearly hypo evolution equations}.
\end{definition}

\begin{remark}
\label{rmk:nh:discrete:sym}
Observe that the involutions $\tau_1$ and $\tau_2$ defined in Remark \ref{rmk:discrete:sym:su2} both preserve the \nh equations. 
In the latter case this corresponds to the fact that $(\omega,\Omega)$ satisfies the \nk equations \eqref{eq:NK} 
if and only if $(-\omega,-\overline{\Omega})$ does. 
Moreover, $(\eta,\omega_1,\omega_2,\omega_3)(t)$ solves the \nh evolution equations 
if and only if $\tau_2(\eta,\omega_1,\omega_2,\omega_3)(t)$ does (and similarly for 
$\tau_1(\eta,\omega_1,\omega_2,\omega_3)(-t)$). Unlike the hypo case, we are not free to change the phase of the holomorphic volume form. 
\end{remark}

As a first example, we now give a simple but fundamental class of mildly singular \nk structures associated with any \se structure on a compact smooth $5$--manifold $N$ via the so-called \emph{sine-cone} construction and explain its relation both to \gtwo geometry in dimension 7 and to \nh structures in dimension 5.

\subsubsection{The sine-cone construction of singular \nk spaces}
The metric cone $C(N)$ over $N$ endowed with an \sunitary{2}--structure equipped with the conical \sunitary{3}--structure
$(\omega_C,\Omega_C)$ defined in \eqref{eq:conical:su2} is Calabi--Yau if and only if the \sunitary{2}--structure is \se.
Hence the metric product $\R \times C(N)$ is a (non smooth) metric cone $C'(N)$ whose holonomy is contained in $\sunitary{3} \subset \gtwo$. 
Because of the $\R$--invariance of $C'(N)$ its cross-section is not smooth but is instead the sine-cone (or metric suspension) over $N$, 
\ie $SC(N):= [0,\pi] \times N$ endowed with the Riemannian metric $\gsc = dr^2 + \sin^2{r}\, g_N$. 

Unless the cone $C(N)$ is isometric to $\C^3$ the sine-cone $SC(N)$ is a compact but singular metric space with two isolated conical singularities at $r=0$ and $r=\pi$ both modelled on $C(N)$. 
Since $C'(N)$ is a (singular) Ricci-flat cone, its cross-section $SC(N)$ with the metric $\gsc$ is a singular Einstein space with scalar curvature 30.
Moreover, since $C'(N)$ has holonomy contained in $\gtwo$, its cross-section $SC(N)$ admits a \nk structure compatible 
with the sine-cone metric $\gsc$ and therefore every orientable hypersurface of $SC(N)$ must admit a \nh structure. 
In particular, the (totally geodesic) hypersurface $\{\tfrac{\pi}{2}\} \times N$ admits a \nh structure whose induced metric 
is isometric to the Sasaki--Einstein metric $g_N$. 

In fact we notice as an immediate consequence of the Sasaki--Einstein and \nh structure equations \eqref{eqn:Sasaki:Einstein} and \eqref{eqn:Nearly:Hypo} respectively 
that if $(\eta,\omega_1, \omega_2, \omega_3)$ 
is Sasaki--Einstein then the ``rotated'' \sunitary{2}--structure $(\eta,-\omega_3,\omega_2,\omega_1)$ is in fact \nh and  
they both induce the same Riemannian metric. More generally, we find that 
the following explicit $1$--parameter family of \nh structures solves the 
\nh  evolution equations \eqref{eqn:Nearly:Hypo:evolution} and induces the ``rotated \se'' \nh structure when $t=\pi/2$:
\begin{equation}\label{eqn:SC}
\begin{gathered}
\eta = \sin{t}\,\etase, \qquad \omega _{1} = \sin^{2}{t} \left( \cos{t}\,\omegase_1 - \sin{t}\,\omegase_3\right),\\
\omega _{2} = \sin^{2}{t}\, \omegase_2, \qquad \omega _{3} = \sin^{2}{t} \left( \sin{t}\,\omegase_1 + \cos{t}\,\omegase_3\right),
\end{gathered}
\end{equation}
for $t \in [0,\pi]$.
Clearly the metric induced on $\{t\} \times N$ by this \nh structure is $\sin^2{t}\, g_N$ as required for the 
\nh structure induced by the \nk structure on the sine-cone. Nearly K\"ahler sine-cones were introduced in \cite{Fernandez:nearly:hypo} generalising the construction of (non-smooth) $\textup{Spin}(7)$--cones that had appeared in the physics literature \cite{Acharya:sine:cone,Bilal:Metzger}; see also \cite[\S 14.4]{BG:Sasaki} for further references. 

Sine-cones (also called spherical or metric suspensions) have also played a key role in the structure theory for 
spaces with lower Ricci curvature bounds. They provide non-smooth metric spaces that have extremal properties 
analogous to the round metric on spheres (which is of course the sine-cone over a lower-dimensional round sphere of the appropriate size)
and therefore appear in several ``almost rigidity" statements, \eg Cheeger and Colding's Almost Maximal Diameter Theorem
\cite[Theorems 5.12 \& 5.14]{Cheeger:Colding:almost:rigidity}
that generalises the classical Maximal Diameter Theorem of Cheng to (singular) limit spaces.

\subsubsection{Invariant \nh structures on $\Ndiag$}
We now specialise to the case of invariant \nh structures on $\Ndiag$. Since we will construct cohomogeneity one \nk manifolds by studying the nearly hypo evolution equations \eqref{eqn:Nearly:Hypo:evolution} restricted to invariant \nh structures, the following result will play a crucial role in the rest of the paper.

\begin{prop}\label{prop:Invariant:nearly:hypo:structures}
Invariant nearly hypo \sunitary{2}--structures on $\Ndiag$ are parametrised by the product of a circle with the open set $\calu$ in $\lorentz{2}$ defined by \eqref{eqn:so12}. Here the embedding of $\lorentz{2} \subset \R^+ \times \R^+ \times \lorentz{3}$ is given by equations \eqref{eqn:invt:nh:A} and \eqref{eqn:invt:nh:lambda:mu} and the \sorth{2} factor corresponds to the orbits of the action of the Reeb vector field $U^-$. Moreover, $\calu$ is a (trivial) real line bundle over $\R^2$. In particular, the space of invariant \nh structures is a smooth connected $4$--manifold.
\begin{proof}
For an invariant \sunitary{2}--structure $\psi_{\lambda,\mu,A}$ on $\Ndiag$ the defining equations \eqref{eqn:Nearly:Hypo} are equivalent to
\[
TA\omegase_1=3\lambda A\omegase_2, \qquad \langle A\omegase_3, \omegase_1 \rangle = \tfrac{\mu}{\lambda},
\]
where $T$ is the matrix defined in \eqref{eqn:Torsion:standard:SE}. Together with the requirements $\langle A\omegase_i,A\omegase_2 \rangle =0$ for $i=1,3$ and $|A\omegase_2|=1$ this implies that $A$ is of the form
\begin{equation}
\label{eqn:invt:nh:A}
A  = \left( \begin{array}{cccc}
1 & 0  & 0  & 0\\
0 & 1 & 0  & 0\\
0 & 0 & \cos{\theta} & -\sin{\theta}\\
0 & 0  & \sin{\theta} & \cos{\theta}\\
\end{array} \right) \left( \begin{array}{cccc}
w_{0} & x_{0}  & 0  & y_{0}\\
w_{1} & x_{1} & 0   & \frac{\mu}{\lambda}\\
w_{2} & -\lambda  & 0 & y_{2}\\
0 & 0  & -1 & 0\\
\end{array} \right) \in \sorth{2} \times \lorentz{2}.
\end{equation}

The matrix
\[
\left( \begin{array}{ccc}
w_{0} & x_{0}     & y_{0}\\
w_{1} & x_{1}     & \frac{\mu}{\lambda}\\
w_{2} & -\lambda  &  y_{2}\\
\end{array} \right)
\]
lies in the open set $\calu$ of \lorentz{2} defined by
\begin{equation}
\label{eqn:so12}
\calu = \{ B \in \lorentz{2} \, |\, B_{32} <0, B_{23} >0 \},
\end{equation}
where $B_{ij}$ denotes the $(i,j)$ entry of the matrix $B$. The map $\sorth{2} \times \calu \rightarrow \R ^+ \times \R^+ \times\lorentz{3}$ defined by \eqref{eqn:invt:nh:A} and
\begin{equation}
\label{eqn:invt:nh:lambda:mu}
\calu \ni B \longmapsto \left( -B_{32}, -B_{23}B_{32} \right) \in \R ^+ \times \R^+
\end{equation}
is clearly injective.

We conclude by describing the open set $\calu$ more precisely. Identify $\lorentz{2}$ with the unit tangent bundle of the hyperbolic plane $\lorentz{2}/\sorth{2}$. The projection of a matrix $B \in \lorentz{2}$ to its first column is the bundle projection and the fibrewise circle action is given by
\begin{equation}\label{eq:Circle:fibre:proj:H}
B \longmapsto B \left( \begin{array}{ccc}
1 & 0 & 0\\
0 & \cos{\varphi}     &  \sin{\varphi}\\
0 & -\sin{\varphi}  &  \cos{\varphi}\\
\end{array} \right).
\end{equation}
The inequalities $B_{32} <0$ and  $B_{23} >0$ define two open intervals of length $\pi$ in each circle fibre. These two intervals must intersect either in a connected open subinterval or a pair of points. In the latter case, however, these two points correspond to matrices of the form 
\[
B=\left( \begin{array}{ccc}
w_0 & x_0 & y_0\\
w_1 & x_1 & 0\\
w_2 & 0 & y_2\\
\end{array} \right) \in \lorentz{2}.
\]
Since $x_1 y_2 = w_0>0$, rotating by $\pm \frac{\pi}{2}$ in the circle containing $B$ we can arrange that both inequalities are satisfied. Thus the first case occurs and $\calu$ is an interval-bundle over $\lorentz{2}/\sorth{2} \simeq \R^2$ as claimed.
\end{proof}
\end{prop}

As an immediate corollary we obtain a characterisation of the invariant nearly hypo structures embedded as hypersurfaces of the sine-cone.

\begin{corollary}\label{cor:characterisation:SC}
Let $\psi_{\lambda, \mu,A}$ be an invariant \nh structure on $\Ndiag$ such that $x_0=0=y_0$ (equivalently, $w_1=0=w_2$) in \eqref{eqn:invt:nh:A}. Then $(\Ndiag,\psi_{\lambda, \mu,A})$ is an invariant hypersurface of the sine-cone over an invariant \se structure on $\Ndiag$.  
\end{corollary}

\section{Cohomogeneity one \nk manifolds}\label{sec:Coho1:NK}

In this section we begin the study of cohomogeneity one \nk $6$--manifolds proper. 
After quickly reviewing basic facts about smooth compact simply connected cohomogeneity one spaces 
we recall Podest\`a and Spiro's classification of possible compact simply connected cohomogenity one 
\nk $6$--manifolds. The only potentially interesting cases all turn out to have principal orbit 
$S^2 \times S^3$ invariant under $\sunitary{2} \times \sunitary{2}$ with isotropy group $\Delta \unitary{1}$.
Using our results on $\sunitary{2} \times \sunitary{2}$--invariant \nh structures  on $S^2 \times S^3$ 
we specialise the \nh evolution equations \eqref{eqn:Nearly:Hypo:evolution}  
to this invariant setting and derive the fundamental ODEs \eqref{eq:nk:odes:t} 
satisfied by any $\sunitary{2} \times \sunitary{2}$--invariant \nk structure 
on the (open dense) subset of principal orbits.
We note the continuous and discrete symmetries of the 
fundamental ODEs, explaining their geometric origins. 
The discrete symmetries in particular play an important role 
in our construction of new complete cohomogeneity one \nk structures. 
We have been unable to find a closed form for the general solution 
to \eqref{eq:nk:odes:t}; however, four explicit solutions are described 
and their geometric significance explained. 
Two of these four solutions also play important roles in the proof of the Main Theorem.

\medskip
Let $M$ be a complete \nk $6$--manifold (as always in the sense of Definition \ref{def:NK}) acted upon isometrically by a connected compact Lie group $G$ with cohomogeneity one, \ie the orbit space $M/G$ is $1$--dimensional. Since $M$ is Einstein with positive Einstein constant, $M$ is compact with finite fundamental group. By \cite[Theorem 9.3, Chapter I]{Bredon} the universal cover of $M$ is also a cohomogeneity one \nk manifold; hence there is no loss of generality in assuming $M$ simply connected, and we will do so throughout the rest of this paper.

From the general theory of cohomogeneity one spaces \cite[Theorem 8.2, Chapter IV]{Bredon}, $M/G$ is then a connected closed interval $[0,T] \subset \R$. The open set $M^\ast \subset M$ corresponding to the interior of $M/G$ is diffeomorphic to $(0,T) \times G/K$. We call $K \subset G$ the principal isotropy subgroup and $G/K$ the \emph{principal orbit}. Corresponding to the boundary points of $M/G$ there are two lower-dimensional \emph{singular orbits} with 
isotropy subgroups $K_1$ and  $K_2$ respectively. We call $K_1$ and $K_2$ the singular isotropy subgroups. They both contain the principal isotropy subgroup $K$ and the coset $K_i/K$ is diffeomorphic to a sphere. Moreover, there are representations $K_i \ra \orth{V_i}$ on Euclidean spaces $V_i$ such that a neighbourhood of the singular orbit $G/K_i$ in $M$ is $G$--equivariantly diffeomorphic to a neighbourhood of the zero section of the vector bundle $G\times _{K_i}V_i \ra G/K_i$. In fact, $M$ is obtained by identifying the two disc bundles $G\times_{K_i}D_i$, $D_i \subset V_i$, along their common boundary $G/K$.

\enlargethispage{\baselineskip}
The set of inclusions $K \subset K_1, K_2 \subset G$ is called the \emph{group diagram} of $M$. Two cohomogeneity one manifolds are $G$--equivariantly diffeomorphic if their group diagrams can be obtained one from the other with the following operations:
\begin{enumerate}
\item interchanging $K_1$ and $K_2$;
\item conjugating $K, K_1, K_2$ by the same element of $G$;
\item replacing $K_1$ with $h K_1 h^{-1}$, where $h$ is an element of the connected component of the normaliser of $K$ in $G$.
\end{enumerate}

In \cite[Theorem 1.1]{Podesta:Spiro:I} Podest\`a and Spiro classified all possible group diagrams of cohomogeneity one \nk $6$--manifolds; 
the list is given in Table \ref{tab:Group:diagrams}. 
\begin{table}[!h]
\centering
\begin{tabular}{ccccc}
$G$ &  $K$  &  $K_1$  &  $K_2$  & $M$ \\ \hline
 $\sunitary{2} \times \sunitary{2}$  &  $\triangle\unitary{1}$ & $\triangle\sunitary{2}$ & $\triangle\sunitary{2}$ & $S^3 \times S^3$\\ 
 $\sunitary{2} \times \sunitary{2}$  &  $\triangle\unitary{1}$ & $\triangle\sunitary{2}$ & $\unitary{1} \times \sunitary{2}$ & $S^6$\\
  $\sunitary{2} \times \sunitary{2}$  &  $\triangle\unitary{1}$ & $\unitary{1} \times \sunitary{2}$ & $\sunitary{2}\times\unitary{1}$ & $CP^3$\\
$\sunitary{2} \times \sunitary{2}$  &  $\triangle\unitary{1}$ & $\unitary{1} \times \sunitary{2}$ & $\unitary{1} \times \sunitary{2}$ & $S^2\times S^4$\\
$\sunitary{3}$ & $\sunitary{2}$ & $\sunitary{3}$ & $\sunitary{3}$ & $S^6$\\
\end{tabular}
\vspace{1em}
\caption{Group diagrams of cohomogeneity one \nk 6--manifolds}
\label{tab:Group:diagrams}
\end{table}
The last case is the sine-cone over the round \se structure on $\Sph^5 \simeq SU(3)/SU(2)$. Since the space of invariant metrics on $\Sph^5$ is $1$--dimensional, it is clear that this is the unique \nk structure arising in that case.

The interesting case is therefore $G=\sunitary{2} \times \sunitary{2}$ with principal orbit $\Ndiag =\sunitary{2} \times \sunitary{2}/ \triangle\unitary{1}$, which motivates our interest in invariant \nh structures on this homogeneous space. We will see later in the section that the first three group diagrams in the list are realised by known homogeneous \nk manifolds; on the other hand, no \nk structure is known to exist on $S^2 \times S^4$. In fact, this case is overlooked in \cite{Podesta:Spiro:I}.

\begin{remark*} B\"ohm \cite{Bohm:Spheres} constructed infinitely many cohomogeneity one Einstein metrics on some of the manifolds of Table \ref{tab:Group:diagrams}. By \cite[Lemma 3.1]{Podesta:Spiro:I} these cannot be induced by a \nk structure since on a \nk $6$--manifold not isometric to the round $6$--sphere any isometry must also preserve the almost complex structure. In particular, the isotropy group of the principal orbit must be contained in \sunitary{2}. For example, the metrics constructed by B\"ohm on $S^3\times S^3$ are invariant under the action of $\sorth{3}\times\sorth{4}$ with principal isotropy group $\sorth{2}\times\sorth{3} \subset \sorth{5}$.
\end{remark*}

\subsection{The fundamental ODE system}

In order to study the \nh evolution equations \eqref{eqn:Nearly:Hypo:evolution} restricted to invariant 
\nh structures it is convenient to introduce an alternative parametrisation to the one of Proposition \ref{prop:Invariant:nearly:hypo:structures}, whose main advantage is to reduce the problem to the study of a first order ODE system, rather than the mixed differential and algebraic system given by \eqref{eqn:Nearly:Hypo} and \eqref{eqn:Nearly:Hypo:evolution}.

Given an invariant \nh structure $\psi_{\lambda, \mu, A}$ with $A$ as in \eqref{eqn:invt:nh:A} we write
\begin{subequations}\label{eq:invt:nh}
\begin{equation}
\label{eq:invt:nh:eta:omega1}
\eta = \lambda \etase, \quad \omega_1 = u_0\, \omegase_0 + u_1\, \omegase_1 + u_2\cos{\theta}\, \omegase_2 + u_2\sin{\theta}\,\omegase_3.
\end{equation}
The first equation in \eqref{eqn:Nearly:Hypo} and \eqref{eqn:Nearly:Hypo:evolution}, respectively, is equivalent to
\begin{equation}
\label{eq:invt:nh:omega2:omega3}
 \omega_2 = -\frac{u_2}{\lambda}\sin{\theta}\,\omegase_2 + \frac{u_2}{\lambda}\cos{\theta}\, \omegase_3, \quad 
\omega_3 = \frac{v_0}{\lambda} \omegase_0 + \frac{v_1}{\lambda} \omegase_1 + \frac{v_2}{\lambda}\cos{\theta}\,\omegase_2 + \frac{v_2}{\lambda}\sin{\theta}\,\omegase_3,
\end{equation}
\end{subequations}
where $v_0, v_1, v_2$ are determined by $\lambda$ and $u_0, u_1, u_2$ via
\begin{subequations}\label{eq:nk:odes:t:1}
\begin{gather}
\lambda \dot{u}_0 + 3v_0=0,\\
\lambda \dot{u}_1 + 3v_1 - 2\lambda^2=0,\\
\lambda \dot{u}_2 + 3v_2=0.
\end{gather}
\end{subequations}
Here $\dot{}$ denotes differentiation with respect to the arc length parameter $t$ along a geodesic orthogonal to the principal orbits. Thus if $B \in \calu$ is the matrix
\[
B=\left( \begin{array}{ccc}
w_0 & x_0 & y_0\\
w_1 & x_1 & y_1\\
w_2 & x_2 & y_2\\
\end{array} \right)
\]
then the change of variables from the parametrisation of invariant \nh structures in Proposition \ref{prop:Invariant:nearly:hypo:structures} to the one in \eqref{eq:invt:nh} is
\begin{equation}\label{eq:change:parametrisation:nh}
u_i=\mu\, x_i,\qquad v_i=\lambda\mu \,y_i,\quad i=0,1,2,\qquad \text{ where } \lambda = -x_2, \quad \mu = -x_2 y_1.
\end{equation}

The second equation of \eqref{eqn:Nearly:Hypo:evolution} implies
\[
\dot{\theta} =0, \qquad \lambda \dot{u}_2 + 3v_2=0.
\]
Since we are free to change $\theta$ by acting by the flow of the Reeb vector field $U^-$, we assume without loss of generality that $\theta=0$. Then the last equation of \eqref{eqn:Nearly:Hypo:evolution} is
\begin{subequations}\label{eq:nk:odes:t:2}
\begin{gather}
\dot{v}_0 -4\lambda u_0=0,\\
\dot{v}_1-4\lambda u_1=0,\\
\lambda \dot{v}_2 - 4\lambda^2u_2 + 3u_2=0.
\end{gather}
\end{subequations}

Besides \eqref{eq:nk:odes:t:1} and \eqref{eq:nk:odes:t:2}, necessary conditions for $\psi_{\lambda, \mu, A}$ to define a \nh structure are the algebraic constraints
\begin{subequations}
\label{eq:constraints}
\begin{gather}
\label{eq:omega13}
I_1=\langle u,v \rangle = 0,\\
I_2=\lambda^2 |u|^2-u_2^2=0,\\
I_3=\lambda ^2 \abs{u}^2-\abs{v}^2=0,\\
\label{eq:NK2}
I_4=v_1 - |u|^2=0,
\end{gather}
\end{subequations}
which correspond to $\omega_1 \wedge \omega_3=0$, $\omega_1^2=\omega_2^2$, $\omega_1^2=\omega_3^2$ and the second equation of \eqref{eqn:Nearly:Hypo}, respectively; $\omega_2 \wedge \omega_1 =0=\omega_2\wedge\omega_3$ follow immediately from \eqref{eq:invt:nh}. Here $\abs{\cdot}$ and $\langle \cdot, \cdot \rangle$ denote the metric of signature $(-++)$ on $\R^{1,2}$. Furthermore, the pair of vectors $u,v \in \R^{1,2}$ has to satisfy the sign constraint
\begin{equation}\label{eq:constraints:sign}
u_1v_2-u_2v_1>0,
\end{equation}
which corresponds to the requirement that the matrix $A$ in \eqref{eqn:invt:nh:A} lies in the restricted Lorentz group.

Following \cite[Proposition 5.1]{Podesta:Spiro:II}, we now derive a differential equation for $\lambda$ which will imply that equations \eqref{eq:constraints} have to be imposed only at the initial time and will then be conserved for all time. Differentiating \eqref{eq:constraints} using \eqref{eq:nk:odes:t:1} and \eqref{eq:nk:odes:t:2} we obtain:
\begin{subequations}
\begin{gather}\label{eq:constraints:derivative}
\lambda\dot{I}_1=3I_2 + 3I_3 +2\lambda^2I_4,\\
\lambda\dot{I}_2=-6\lambda^2 I_1 + 2\lambda^2\abs{u}^2\dot{\lambda} +4\lambda^4 u_1 + 6u_2 v_2,\\
\lambda\dot{I}_3=-14\lambda^2 I_1,\\
\lambda\dot{I}_4=6I_2.
\end{gather}
\end{subequations}
Thus if $\lambda$ satisfies the first order differential equation 
\begin{equation}\label{eq:nk:odes:t:3}
\lambda^2\abs{u}^2\dot{\lambda} +2\lambda^4 u_1 + 3u_2 v_2=0,
\end{equation}
then $I=(I_1,\dots,I_4$) satisfies a homogeneous linear system (with coefficients depending on $\lambda$) and is therefore uniquely determined by the initial conditions.

The following proposition follows immediately from this discussion. 

\begin{prop}\label{prop:nk:odes}
Let $(\lambda,u,v)$ be a solution of the ODE system 
\begin{subequations}
\label{eq:nk:odes:t}
\begin{gather}
\lambda \dot{u}_0 + 3v_0=0,\\
\lambda \dot{u}_1 + 3v_1 - 2\lambda^2=0,\\
\lambda \dot{u}_2 + 3v_2=0,\\
\dot{v}_0 -4\lambda u_0=0,\\
\dot{v}_1-4\lambda u_1=0,\\
\lambda \dot{v}_2 - 4\lambda^2u_2 + 3u_2=0,\\
\lambda^2\abs{u}^2\dot{\lambda} +2\lambda^4 u_1 + 3u_2 v_2=0.
\end{gather}
\end{subequations}
defined on an interval $(a,b)\subset\R$ on which $u_2<0$, $\lambda, \mu^2:=|u|^2>0$ and \eqref{eq:constraints:sign} is satisfied. Assume also that $I_1(t_0)=\dots=I_4(t_0)=0$ for some $a < t_0 < b$. Then \eqref{eq:invt:nh} defines an invariant \nk structure on $(a,b) \times \Ndiag$. Conversely, any \nk structure on $(a,b) \times \Ndiag$ invariant under the action of $\sunitary{2} \times \sunitary{2}$ on $\Ndiag$ takes the form \eqref{eq:invt:nh} for a solution $(\lambda,u,v)$ of \eqref{eq:nk:odes:t} satisfying the given sign constraints and with $I_1=I_2=I_3=I_4=0$ for all time.

Moreover, given an invariant \nh structure $\psi_{\lambda,\mu,A}$ on $\Ndiag$ there exists a unique solution of \eqref{eq:nk:odes:t} with initial condition $\psi_{\lambda,\mu,A}$ (in particular, $I_1(0)=\dots=I_4(0)=0$). In particular, up to the action of the flow of the Reeb vector field there exists a $2$--parameter family of local invariant \nk structures on $(a,b)\times \Ndiag$.
\end{prop}

\begin{remark*}
In \cite[Thorem 5]{Conti} Conti shows that any compact real analytic \nh 5--manifold can be embedded into a (local) real analytic \nk  6--manifold. The final statement of the Proposition, which follows from standard ODE theory, is a specialisation of Conti's result to the case of invariant \nh structures on $\Ndiag$.
\end{remark*}

\subsection{Symmetries of the fundamental ODE system}

In the rest of the paper we will make repeated use of various symmetries of the fundamental ODE system \eqref{eq:nk:odes:t}; 
the discrete symmetries of \eqref{eq:nk:odes:t} in particular will turn out to play a crucial role in our construction of new  complete \nk  metrics on  $S^3 \times S^3$ and $S^6$.  

To facilitate the description of these symmetries 
we introduce the alternative notation $(\lambda, u_0, u_1, u_2, v_0, v_1, v_2, t)$ for a solution 
$\Psi= \big( \lambda (t), u_0(t), u_1(t),u_2(t),v_0(t),v_1(t),v_2(t)\big)$ 
of \eqref{eq:nk:odes:t}.

\begin{prop}[\mbox{\cf \cite[Proposition 4.2]{Podesta:Spiro:II}}]
\label{prop:symmetries}
The system \eqref{eq:nk:odes:t} is invariant under the following symmetries.
\begin{enumerate}
\item Time translation $t \mapsto t+t_0$, $t_0 \in \R$.
\item Time reversal
\[
\tau_1 \co (\lambda, u_0, u_1, u_2, v_0, v_1, v_2, t) \longmapsto (-\lambda, u_0, u_1, u_2, v_0, v_1, v_2, -t).
\]
\item The involutions $\tau_2, \tau_3, \tau_4$ defined by 
\begin{align*}
&\tau_2 \co (\lambda, u_0, u_1, u_2, v_0, v_1, v_2, t) \longmapsto (-\lambda, -u_0, -u_1, -u_2, v_0, v_1, v_2, t),\\
&\tau_3 \co (\lambda, u_0, u_1, u_2, v_0, v_1, v_2, t) \longmapsto (\lambda, u_0, u_1, -u_2, v_0, v_1, -v_2, t),\\
&\tau_4 \co (\lambda, u_0, u_1, u_2, v_0, v_1, v_2, t) \longmapsto (-\lambda, u_0, -u_1, u_2, -v_0, v_1, -v_2, t).
\end{align*}
\end{enumerate}
$\tau_2$ preserves the constraint $u_2=-\lambda \mu$, while the remaining $\tau_i$ send this constraint into $u_2=\lambda \mu$. 
All the symmetries preserve the constraints \eqref{eq:constraints}.
\proof
The fact that these transformations are symmetries of \eqref{eq:nk:odes:t} is straightforward to verify by direct computation. 
Instead we concentrate on explaining the geometric origin of each of these symmetries. The involutions $\tau_1$ and $\tau_2$ are simply the specialisation to the invariant case
of the involutions $\tau_1$ and $\tau_2$ defined in Remark \ref{rmk:discrete:sym:su2} and we already noted in Remark \ref{rmk:nh:discrete:sym} 
that both involutions preserve the \nh condition. The existence of the remaining involutions $\tau_3$ and $\tau_4$ is 
specific to the case of $\sunitary{2}\times\sunitary{2}$--invariant \sunitary{2}--structures: $\tau_3$ and $\tau_4$ are induced by automorphisms of $\sunitary{2} \times \sunitary{2}$ that fix $\triangle\unitary{1}$. More precisely, $\tau_4$ is the action on invariant \sunitary{2}--structures of the outer automorphism of $\sunitary{2} \times \sunitary{2}$ that exchanges the two factors. 
The Reeb vector field $U^-$ generates the group of inner automorphisms of $\sunitary{2} \times \sunitary{2}$ fixing $\triangle \unitary{1}$. By normalising nearly hypo structures so that $w_3=0$, \ie $\theta =0$ or $\pi$ in \eqref{eqn:invt:nh:A}, we quotient out this action except for a residual $\Z _2$--action generated by $\tau_3$.
\endproof
\end{prop}

\subsection{Special solutions of the fundamental ODE system}\label{sec:Homogeneous}

There are four distinguished solutions to the ODE system \eqref{eq:nk:odes:t}: the sine-cone over the homogeneous Sasaki--Einstein metric on $\Ndiag$ and the three homogeneous \nk solutions for which there is a subgroup of the full isometry group isomorphic to $\sunitary{2} \times \sunitary{2}$, namely 
\nk  $S^6$, $CP^3$ and $S^3 \times S^3$.
These special solutions will play a role in our analysis of general cohomogeneity one $\sunitary{2} \times \sunitary{2}$--invariant \nk metrics and provide explicit solutions of \eqref{eq:nk:odes:t} that we will record below. It is immediate to verify that the given expressions define solutions of \eqref{eq:nk:odes:t} but we also refer to \cite[\S 4.2]{Podesta:Spiro:II} where these expressions are derived. 

\begin{example}[The sine-cone]\label{Sine:Cone}
The sine-cone over the standard \se structure on $\Ndiag$ has already been discussed in \eqref{eqn:SC}. In terms of \eqref{eq:invt:nh}, we have
\[
\begin{gathered}
\lambda = \sin{t}, \quad u_0=0, \quad u_1=\sin^2{t}\cos{t}, \quad u_2=-\sin^3{t},\\
v_0=0, \quad v_1=\sin^4{t}, \quad v_2=\sin^3{t}\cos{t},
\end{gathered}
\]
for $t \in [0,\pi]$. The two conical singularities of the sine-cone occur at $t=0$ and $t=\pi$.
\end{example}

\begin{example}[The round sphere]\label{S6}
In this case
\[
\begin{gathered}
\lambda = \frac{3}{2} \cos{t}, \quad u_0 = -\frac{3}{2}\sin{t} \left( 2-5 \cos^2{t}\right), \quad u_1=-3\sin{t}\left(1-2\cos^2{t}\right), \quad u_2=-\frac{9}{2} \sin{t}\cos^2{t},\\
v_0 = \frac{9}{4}\cos^{2}{t}\left( 4-5\cos^2{t}\right), \quad v_1=9\sin^{2}{t}\cos^2{t}, \quad v_2=\frac{9}{4}\cos^2{t}\left( 3\cos^{2}{t}-2\right),
\end{gathered}
\]
for $t \in [0,\tfrac{\pi}{2}]$. The open set of principal orbits is compactified by adding a $3$--sphere $\sunitary{2} \times \sunitary{2}/\triangle\sunitary{2}$ at $t=0$ and a $2$--sphere $\sunitary{2} \times \sunitary{2}/\sunitary{2}\times\unitary{1}$ at $t=\frac{\pi}{2}$.
\end{example}

\begin{example}[Homogeneous \nk structure on $S^3\times S^3$]\label{S3xS3}
The homogeneous \nk structure on $S^3 \times S^3$ corresponds to the solution
\[
\begin{gathered}
\lambda = 1, \quad u_0=u_1 = \frac{1}{\sqrt{3}} \sin{(2\sqrt{3}t)}, \quad u_2 = -\frac{2}{\sqrt{3}}\sin{(\sqrt{3}t)},\\
v_0=-\frac{2}{3}\cos{(2\sqrt{3}t)}, \quad v_1=\frac{2}{3}\left( 1-\cos{(2\sqrt{3}t)}\right), \quad v_2=\frac{2}{3}\cos{(\sqrt{3}t)},
\end{gathered}
\]
for $t \in [0,\tfrac{\pi}{\sqrt{3}}]$. (Notice that there is a typo in \cite[\S 4.2.3]{Podesta:Spiro:II}: the range of $t$ with their normalisations should be $t \in [0,\tfrac{\pi}{\sqrt{6}}]$.) The singular orbits are two $3$--spheres, $\sunitary{2} \times \sunitary{2}/\triangle\sunitary{2}$ at $t=0$ and $\sunitary{2} \times \sunitary{2}/\phi_3(\triangle\sunitary{2})$ at $t=\frac{\pi}{\sqrt{3}}$, where $\phi_3$ is the inner automorphism of $\sunitary{2}\times\sunitary{2}$ generated by $\frac{\pi}{4}U^-$.
\end{example}

\begin{example}[Homogeneous \nk structure on $\CP^3$]\label{CP3}
In this case,
\[
\begin{gathered}
\lambda = \frac{3\sqrt{2}}{4}\sin{(\sqrt{2}t)}, \quad u_0=\frac{3}{8}\left( 3\cos^2{(\sqrt{2}t)}-1 \right) , \quad u_1 = \frac{3}{4}\cos{(\sqrt{2}t)}, \quad u_2 = -\frac{9}{8}\sin^2{(\sqrt{2}t)},\\
v_0=-\frac{9}{8}\cos{(\sqrt{2}t)}\sin^2{(\sqrt{2}t)}, \quad v_1=\frac{9}{8}\sin^2{(\sqrt{2}t)}, \quad v_2=\frac{9}{8}\cos{(\sqrt{2}t)}\sin^2{(\sqrt{2}t)},
\end{gathered}
\]
for $t \in [0,\tfrac{\pi}{\sqrt{2}}]$. The singular orbits are both $2$--spheres, $\sunitary{2} \times \sunitary{2}/\unitary{1}\times\sunitary{2}$ at $t=0$ and $\sunitary{2} \times \sunitary{2}/\sunitary{2}\times\unitary{1}$ at $t=\frac{\pi}{\sqrt{2}}$.
\end{example}

\section{Solutions that extend smoothly over the singular orbits}
\label{sec:closing}

In the previous sections we have described the subset consisting of principal orbits of a cohomogeneity one \nk manifold as a $1$--parameter family of \nh structures. In this section we will discuss (singular) boundary conditions for the ODE system \eqref{eq:nk:odes:t}, \ie study conditions under which a cohomogeneity one \nk structure extends smoothly across a singular orbit. Recall from Table \ref{tab:Group:diagrams} that in the $\sunitary{2}\times\sunitary{2}$--invariant case there are only three types of singular orbits, a $3$--sphere $\sunitary{2}\times\sunitary{2}/\triangle\sunitary{2}$ and $2$--spheres $\sunitary{2}\times\sunitary{2}/\unitary{1}\times\sunitary{2}$ and $\sunitary{2}\times\sunitary{2}/\sunitary{2}\times\unitary{1}$. The latter two are exchanged by the outer automorphism of $\sunitary{2}\times\sunitary{2}$.
The main results of the section are Theorems \ref{thm:Singular:IVP:S^2} and \ref{thm:Singular:IVP:S^3}:
these establish the existence of two $1$--parameter families of local cohomogeneity one \nk $6$--manifolds 
that close smoothly on a singular orbit that is a round sphere of dimension two or dimension three, respectively. 
In both cases the parameter is the size of the singular orbit. In subsequent sections, by studying 
the behaviour of these two $1$--parameter families as the size of the singular orbit shrinks to zero,  we will show that they should be viewed 
as \nk deformations of the Calabi--Yau structures on the small resolution and on the smoothing of the conifold respectively. 

\medskip
To understand the conditions under which invariant tensors on a cohomogeneity one manifold extend smoothly across a singular orbit we will appeal to a method due to Eschenburg and Wang \cite[\S 1]{Eschenburg:Wang}. For the convenience of the reader, we describe it  briefly here.  

Let $M^{n}$ be a smooth manifold with a cohomogeneity one isometric action of a compact Lie group $G$. Let $Q = G/K'$ be a singular orbit. Set $V=T_{q}M/T_{q}Q$ and recall that a neighbourhood of $Q$ in $M$ is $G$--equivariantly diffeomorphic to a neighbourhood of the zero section of the normal bundle $E = G \times _{K'} V \ra Q$.

In view of our applications and for concreteness, we only discuss conditions under which a $G$--invariant section $h \in \Gamma \left( E;\End (TE) \right)$ extends smoothly over the zero section, but the method generalises to arbitrary tensors. $G$--invariance implies that $h$ is determined by its restriction to $V \simeq E_{q}$. The choice of a complement $\Lie{p}$ of the Lie algebra of $K'$ in the Lie algebra of $G$ determines a trivialisation $TE|_{V} = V \oplus \Lie{p}$. Fix a point $v_{0} \in V$ with $|v_{0}|=1$ (having fixed an invariant inner product on $V$) and denote by $K$ its stabiliser in $G$. The principal orbits of $M$ are diffeomorphic to $G/K = G \times _{K'} S^{d-1} \ra Q$, where $S^{d-1}$ is the unit sphere in $V$. 

Introduce polar coordinates $(t, \sigma ) \in [0,\infty ) \times S^{d-1} \simeq V$. Then $h \in \Gamma \left( V;\End (V \oplus \Lie{p})\right)$ can be thought of as a $1$--parameter family of maps $h_{t} \in \Gamma \left( S^{d-1};\End (V \oplus \Lie{p})\right)$. Since $K'$ acts transitively on $S^{d-1}$, the space $\mathcal{W}$ of $K'$--equivariant maps $h_{t} \co S^{d-1} \ra \End (V \oplus \Lie{p})$ is isomorphic to $\End (V \oplus \Lie{p}) ^{K}$ via the evaluation at $v_{0}$. Denote by $\mathcal{W}_{p}$ the subspace of $\mathcal{W}$ consisting of the restriction of homogeneous polynomials of degree $p$. Notice that we can always increase the degree of $h \in \mathcal{W}_{p}$ by $2$ by multiplying $h$ by $v \mapsto |v|^{2}$. However, by finite dimensionality of $\mathcal{W}$ and polynomial approximation, we can find a basis of the vector space $\End (V \oplus \Lie{p}) ^{K}$ such that every element corresponds to a $K'$--equivariant homogeneous polynomial of minimum degree $p \geq 0$. Then a curve $h_{t} \in \mathcal{W}$ defined for $t \in [0,T)$ represents a smooth section $h \in C^{\infty} \left( E;\End (TE) \right)$ if and only if it has Taylor series at $t=0$, $h_{t} = \sum _{p \geq 0}{h_{p}t^{p}}$, with $h_{p} \in \mathcal{W}_{p}$ for all $p$ \cite[Lemma 1.1]{Eschenburg:Wang}. In this way the problem is reduced to a representation-theoretic computation.

We now specialise our discussion to the case $n=6$, $G= \sunitary{2} \times \sunitary{2}$,  $K=\triangle\unitary{1}$ and 
$K'=\Delta \sunitary{2}, \unitary{1} \times \sunitary{2}$ or $\sunitary{2}\times\unitary{1}$.

\subsection{Closing smoothly on an $S^2$}

We first consider the case $K'=\unitary{1} \times \sunitary{2}$ and $Q \simeq S^{2}$. Then $\Lie{p} = \Lie{n}_1$ $=\textup{Span}\{E_1,V_1\}$ in the notation of \eqref{eqn:Invariant:vector:fields} and $V \simeq \HH$, where $\unitary{1}\times\sunitary{2}$ acts on $V$ via $(e^{i\theta},q) \cdot x = qxe^{-i\theta}$. In particular, the vector bundle $E = \left( \sunitary{2}\times\sunitary{2} \right)\times_{\unitary{1}\times\sunitary{2}}V$ is isomorphic to $\mathcal{O}(-1)\oplus\mathcal{O}(-1) \ra S^2$.

As a $\Delta\unitary{1}$--representation $V = \R^2 \oplus \Lie{n}_{2}$, where $\Lie{n}_{2} \simeq \Lie{n}$ is the standard representation of \unitary{1} and $\R ^2$ is a trivial $2$--dimensional representation. It follows that $\End (V \oplus \Lie{p}) ^{\unitary{1}} = \End(\R^2) \oplus 4\End (\Lie{n})^{\unitary{1}}$. Moreover, $\End (\Lie{n})^{\unitary{1}}$ is $2$--dimensional, spanned by the identity and the complex structure.

Since \sunitary{2} already acts transitively on the unit sphere in $V$, any $\unitary{1}\times\sunitary{2}$--equivariant polynomial $h$ on $\Sph^3 \subset V$ with values in $\End (V \oplus \Lie{p})$ must satisfy $h(q) = M^{-1} h(1) M$, where $h(1) \in \End (V \oplus \Lie{p}) ^{\unitary{1}}$, $M = (1,q) \in \unitary{1} \times \sunitary{2}$ and we identified $\Sph^3$ with \sunitary{2}. A simple computation then shows that
\begin{enumerate}
\item $\id _{\Lie{n}_1}, j_{\Lie{n}_1}, \id _{\R ^2} + \id _{\Lie{n}_2}, j _{\R ^2} + j _{\Lie{n}_2}$ correspond to constant polynomials (they are preserved by $\unitary{1} \times \sunitary{2}$);
\item $\End (\Lie{n}_i, \Lie{n}_j) ^\unitary{1}$ with $i \neq j$ correspond to polynomials of degree $1$;
\item $\id _{\R ^2} - \id _{\Lie{n}_2}, j _{\R ^2} - j _{\Lie{n}_2}$ and $\text{Sym}^{2}_{0}(\R ^2)$ correspond to degree $2$ polynomials.
\end{enumerate}
Here $j_X$ denotes the standard complex structure on $X$.

The general theory of \cite{Eschenburg:Wang} now gives conditions for the smooth extension of any $\unitary{1}\times\sunitary{2}$--equivariant $\End (V \oplus \Lie{p})$--valued polynomial on $V$. However, in order to interpret these conditions as initial data for the ODE system \eqref{eq:nk:odes:t} it is necessary to change coordinates from this description of a neighbourhood of the singular orbit to the coframe $dt, u^-, e_{i}, v_{i}$ that we introduced on the set of principal orbits. To this end embed $S^2 \times V$ in $\Imag{\HH} \times \HH$ and let $\sunitary{2} \times \sunitary{2}$ act via $(q_{1},q_{2}) \cdot (x,y) = (q_1 x q_1^\ast, q_{2}y q_{1}^\ast)$. Along the ray $\gamma (t)= (i,t)$ the vector fields $U^-,E_{i},V_{i}$ of \eqref{eqn:Invariant:vector:fields} are
\[
\begin{gathered}
U^-=(0,-it), \qquad E_{1} = \frac{1}{2\sqrt{2}}(-2k,-jt), \qquad V_{1} = \frac{1}{2\sqrt{2}}(2j,-kt),\\
\qquad E_{2} = \frac{1}{2\sqrt{2}}(0,jt), \qquad V_{2} = \frac{1}{2\sqrt{2}}(0,kt).
\end{gathered}
\]
Thus if $(t,q_1,q_2,q_3)$ are coordinates on $V = \HH$, along $\gamma$ we have $dt=dq_0$, $tu^-=dq_1$, $te_{2} = 2\sqrt{2}dq_2$ and $tv_{2}=2\sqrt{2}dq_3$, while $e_1, v_1$ are $1$--forms along $S^2$.

In the notation of \eqref{eq:invt:nh:eta:omega1} we write an invariant $2$--form on $(0,\epsilon) \times \Ndiag$ as
\[
\omega = \lambda \etase \wedge dt + u_0\,\omegase_0 + u_1\,\omegase_1 + u_2\,\omegase_2+u_3\,\omegase_3,
\]
where $\lambda, u_0, \dots, u_3$ are functions depending on time only. Recalling \eqref{eq:invt:Sasaki:Einstein}, the change of variables above yields
\[
\begin{gathered}
\omega = -\frac{2\lambda}{3t} dq_0 \wedge dq_1 + \frac{u_0+u_1}{12}e_1\wedge v_1 + \frac{2(u_0-u_1)}{3t^2}dq_2\wedge dq_3 +\\ \frac{\sqrt{2}u_2}{6t}\left( e_1\wedge dq_3 + dq_2 \wedge v_1\right) + \frac{\sqrt{2}u_3}{6t}\left( e_1\wedge dq_2 - dq_3 \wedge v_1\right).
\end{gathered}
\]

It is now straightforward to use the representation-theoretic computation to deduce conditions on the coefficients of $\omega$ so that it extends smoothly at $t=0$ along the singular orbit.

\begin{lemma}\label{lem:Smooth:extension:S2} Let $\omega = \lambda dt \wedge \etase + u_0\, \omegase_0 + u_1\, \omegase_1 + u_2\, \omegase_2 + u_3\, \omegase_3$ be an invariant $2$--form on $(0,T) \times \Ndiag$. Then:
\begin{enumerate}
\item $\omega$ extends smoothly over a singular orbit $\sunitary{2}\times\sunitary{2}/\unitary{1}\times\sunitary{2}$ at $t=0$ if and only if
\begin{enumerate}
\item $u_0, u_1, u_2, u_3$ are even and $\lambda$ is odd;
\item $u_2 (0) = 0 = u_3 (0)$ and $u_0 (t) - u_1 (t) = -\dot{\lambda}(0) t^2 + O(t^4)$, where $\dot{\lambda}$ denotes the derivative of $\lambda$ with respect to $t$.
\end{enumerate}
\item $\omega$ extends smoothly over a singular orbit $\sunitary{2}\times\sunitary{2}/\sunitary{2}\times\unitary{1}$ at $t=0$ if and only if
\begin{enumerate}
\item $u_0, u_1, u_2, u_3$ are even and $\lambda$ is odd;
\item $u_2 (0) = 0 = u_3 (0)$ and $u_0 (t) + u_1 (t) = \dot{\lambda}(0) t^2 + O(t^4)$.
\end{enumerate}
\end{enumerate}
\proof
The statement in (i) follows immediately from the representation-theoretic computation, since $u_0+u_1$ must be even, $\frac{u_2}{t}$ and $\frac{u_3}{t}$ odd, $-\frac{2\lambda}{3t}$ and $\frac{2(u_0-u_1)}{3t^2}$ are both even and well-defined at $t=0$ where they must take the same value. Since $\unitary{1}\times\sunitary{2}$ and $\sunitary{2}\times\unitary{1}$ are exchanged by the outer automorphism of $\sunitary{2}\times\sunitary{2}$, the statement in (ii) follows immediately from (i) after applying the discrete symmetry $\tau_4$ of Proposition \ref{prop:symmetries}.
\endproof
\end{lemma}

\subsection{Closing smoothly on an $S^3$}

When $K'=\Delta\sunitary{2}$ and $Q \simeq S^{3}$ , $V$ and $\Lie{p}$ are both isomorphic to the adjoint representation of \sunitary{2}. The vector bundle $E$ is trivial and we can identify it with $T^\ast S^3$.

Since $\Lie{su}_{2}= \R \oplus \Lie{n}$ as a \unitary{1}--representation, $\End (V \oplus \Lie{p}) ^{\unitary{1}} = 4 \End (\R \oplus \Lie{n})^{\unitary{1}}$. It is easy to see that $\End (\R \oplus \Lie{n})^{\unitary{1}}$ is a $3$--dimensional vector space, spanned by $\id_{\R}$, $\id_{\Lie{n}}$ and $j_{\Lie{n}}$. Under the identification of the action of \sunitary{2} on its Lie algebra with the standard action of \sorth{3} on $\R ^{3}$, \sorth{3}--equivariant polynomials on $\R ^{3}$ with values in $\End (\R ^{3})$ are generated by the constant polynomial $\id$, the degree two polynomial $v \mapsto \langle \,\cdot\, , v \rangle v$ and the degree one polynomial $v \mapsto \,\cdot\, \times v$, which by evaluation at $(1,0,0)$ correspond to $\id _{\R} + \id_{\Lie{n}}$, $\id _{\R}$ and $j_{\Lie{n}}$, respectively.

In order to understand how to change coordinates from this description of the neighbourhood of the singular orbit to the coframe $dt, u^-, e_{i}, v_{i}$ on the set of principal orbits, think of $S^3 \times \R ^3$ as embedded in $\HH \times \Imag{\HH}$ and let $\sunitary{2} \times \sunitary{2}$ act via $(q_{1},q_{2}) \cdot (x,y) = (q_2 x q_1^\ast, q_{1}y q_{1}^\ast)$. Along the ray $\gamma (t)= (1,it)$ the vector fields $U^-,E_{i},V_{i}$ of \eqref{eqn:Invariant:vector:fields} are
\[
\begin{gathered}
U^-=(-i,0), \qquad E_{1} = \frac{1}{2\sqrt{2}}(-j,-2kt), \qquad V_{1} = \frac{1}{2\sqrt{2}}(-k,2jt),\\
\qquad E_{2} = \frac{1}{2\sqrt{2}}(j,0), \qquad V_{2} = \frac{1}{2\sqrt{2}}(k,0).
\end{gathered}
\]
Thus if $(t,x,y)$ are coordinates on $\R ^3 = \Imag\HH$, along $\gamma$ we have $e_{+} = -\frac{\sqrt{2}}{t}dy$ and $v_{+}=\frac{\sqrt{2}}{t}dx$, where $e_{\pm}, v_\pm$ and are the $1$--forms dual to the vector fields $E_{1}\pm E_2$ and $V_1 \pm V_2$, respectively. On the other hand, $u^-, e_-, v_-$ gives instead a coframe on the singular orbit $S^3$.

\begin{lemma}[\mbox{\cf \cite[Proposition 6.1]{Podesta:Spiro:II}}]
\label{lem:Smooth:extension:S3}
An invariant $2$--form $\omega = \lambda dt\wedge \etase + u_0\, \omegase_0 + u_1\, \omegase_1 + u_2\, \omegase_2 + u_3\, \omegase_3$ on $(0,T) \times \Ndiag$ extends smoothly over a singular orbit $\sunitary{2}\times\sunitary{2}/\triangle\sunitary{2}$ if and only if
\begin{enumerate}[label=(\roman*)]
\item $u_0, u_1, u_2$ are odd and $\lambda, u_3$ are even;
\item $u_0 + u_2 = O(t^3)$, $u_3 = O(t^2)$ and $u_1 (t) = 2\lambda (0) t + O(t^3)$.
\end{enumerate}
\end{lemma}

\begin{remark*}
Lemmas \ref{lem:Smooth:extension:S2} and \ref{lem:Smooth:extension:S3} give conditions for an invariant $2$--form $\omega$ to extend smoothly along a singular orbit. Suppose now we have a solution $\Psi = ( \lambda, u,v)$ of the fundamental ODE system \eqref{eq:nk:odes:t} such that $\lambda, u_0, u_1, u_2$ satisfy the conditions of either of these lemmas (with $u_3=0$) at $t=0$ and the constraints $I_1=\dots=I_4=0$ of \eqref{eq:constraints} hold for all time. Thus $\Psi$ defines an invariant \nk structure $(\omega,\Omega)$ on $(0,T) \times \Ndiag$ and $\omega$ extends smoothly at $t=0$. Since $3\Real{\Omega}=d\omega$ and $\Imag{\Omega}$ is determined algebraically by $\Real{\Omega}$ \cite[\S 2]{Hitchin:stable:forms}, it follows that the whole $\sunitary{3}$--structure $(\omega,\Omega)$ extends smoothly at $t=0$.
\end{remark*}

\subsection{Extension of symmetries over a singular orbit}

Before dealing with the existence of solutions of the fundamental ODE system \eqref{eq:nk:odes:t} closing smoothly on a singular orbit, we discuss here which symmetries of Proposition \ref{prop:symmetries} extend over the singular orbits of Table \ref{tab:Group:diagrams}.

\begin{lemma}\label{lem:symmetries:singular:orbits}
Let $\Psi(t)$ be a solution of the fundamental ODE system \eqref{eq:nk:odes:t} defined on $[0,T)$.
\begin{enumerate}
\item If $\Psi$ extends smoothly over a singular orbit $S^2 = \sunitary{2} \times \sunitary{2} / \unitary{1} \times \sunitary{2}$  at $t=0$ then so does $\tau_i(\Psi)$ for $i=1,2,3$, while $\tau_4(\Psi)$ extends smoothly over a singular orbit $\sunitary{2} \times \sunitary{2} / \sunitary{2} \times \unitary{1}$.
\item If $\Psi$ extends smoothly over a singular orbit $S^3 = \sunitary{2} \times \sunitary{2} / \triangle\sunitary{2}$  at $t=0$ then so does $\tau_i(\Psi)$ $i=1,2,4$, while $\tau_3(\Psi)$ extends smoothly over a singular orbit $\sunitary{2} \times \sunitary{2} / \phi_{3}\left(\triangle\sunitary{2}\right)$, where $\phi_3=\text{Ad}\left( \exp{\left(\frac{\pi}{4}U^- \right)} \right)$.
\end{enumerate}
\proof
It is obvious that $\tau_1$ and $\tau_2$ preserve the boundary conditions for closing smoothly on a singular orbit of any type. It remains to check whether the automorphisms $\phi_3$ and $\phi_4$ corresponding to $\tau_3$ and $\tau_4$, respectively, fix not only $\triangle \unitary{1}$ but also the stabiliser $K'$ of a point on the singular orbit.

Denote by $(\phi_j)_\ast$ the infinitesimal action of $\phi_j$, $j=3,4$, on $\R U^- \oplus \Lie{n}_1 \oplus \Lie{n}_2$, identified with the tangent space of $\Ndiag$ at a point. Then
\[
(\phi_3)_\ast = \text{Ad}\left( \exp{\left(\frac{\pi}{4}U^- \right)} \right) =\left( \begin{array}{ccc} 1 & 0 & 0\\ 0 & i & 0\\ 0 & 0 & -i\\ \end{array} \right), \qquad
(\phi_4)_\ast = \left( \begin{array}{ccc} -1 & 0 & 0\\ 0 & 0 & 1\\ 0 & 1 & 0\\ \end{array} \right).
\]

The Lie algebras of $K_1=\unitary{1} \times \sunitary{2}$ and $K_2=\sunitary{2} \times \unitary{1}$ are $\R U^+ \oplus \R U^- \oplus \Lie{n}_1$ and $\R U^+ \oplus \R U^- \oplus \Lie{n}_2$, respectively. Thus $\phi_3$ preserves both $K_1$ and $K_2$, while $\phi_4$ exchanges them. Similarly, $(\phi_3)_\ast$ exchanges $\Lie{n}_+$ and $\Lie{n}_-$, where $\Lie{n}_\pm$ is the subspace generated by $E_\pm = E_1 \pm E_2$ and $V_\pm = V_1 \pm V_2$, while $\phi_4$ preserves $\triangle\sunitary{2}$.
\endproof
\end{lemma}

\subsection{Existence of solutions extending smoothly over a singular orbit}

The main results of this section are the existence of two $1$--parameter families of solutions to the fundamental ODE system \eqref{eq:nk:odes:t} closing smoothly on a singular orbit.

\begin{theorem}[Invariant \nk metrics extending smoothly over a singular orbit $S^2$]
\label{thm:Singular:IVP:S^2}
For every $a>0$ there exists a unique solution $\Psi_a$ to \eqref{eq:nk:odes:t} satisfying the initial conditions
\[
\lambda (0)=u_2 (0)=0, \quad u_0 (0)=u_1(0)=a^2, \quad v_0(0)= v_1 (0)=v_2 (0)=0,
\]
and such that $\lambda, u_0, u_1, u_2$ satisfy the conditions of Lemma \ref{lem:Smooth:extension:S2}(i). In other words, for every $a>0$ there exists a unique $\sunitary{2} \times \sunitary{2}$--invariant smooth \nk structure defined in a sufficiently small neighbourhood of the zero section of $\mathcal{O}(-1) \oplus \mathcal{O}(-1) \ra S^2$ and such that the zero section has volume $a^2$ with respect to the induced metric. Moreover, $\Psi_a$ depends continuously on $a \in (0, \infty)$.
\end{theorem}

The singular orbit $S^2$ here is $\sunitary{2}\times\sunitary{2}/\unitary{1}\times\sunitary{2}$. The composition $\tau_2\circ\tau_3\circ\tau_4$ preserves the sign constraints $\lambda>0$, $u_2<0$, $v_1>0$ and by Lemma \ref{lem:symmetries:singular:orbits} it sends $\Psi_{a}$ to a solution of \eqref{eq:nk:odes:t} which extends smoothly across a singular orbit $\sunitary{2}\times\sunitary{2}/\sunitary{2}\times\unitary{1}$.

\begin{theorem}[Invariant \nk metrics extending smoothly over a singular orbit $S^3$]
\label{thm:Singular:IVP:S^3}
For every $b>0$ there exists a unique solution $\Psi_b$ to \eqref{eq:nk:odes:t} satisfying the initial conditions
\[
\lambda (0)=b,\quad v_0(0)=-v_2(0)=-\frac{2}{3}b^3,\quad v_1(0)=0, \quad u_i (0)=0, \quad i=0,1,2.
\]
In other words, for every $b>0$  there exists a unique $\sunitary{2} \times \sunitary{2}$--invariant smooth \nk structure 
defined in a sufficiently small neighbourhood of the zero section of $T^\ast S^3$ such that the zero section has volume $b^{3}$ with respect to the induced metric. Moreover, $\Psi_b$ depends continuously on $b \in (0, \infty)$.
\end{theorem}

\begin{remark}\label{rmk:a:b:homogeneous}
The solutions of Examples \ref{S6}, \ref{S3xS3} and \ref{CP3} of course belong to the two \mbox{$1$--parameter} families of the theorems. More precisely, the round \nk structure on $S^6$, Example \ref{S6}, is $\Psi_a$ with $a=\sqrt{3}$ and $\Psi_b$ with $b=\frac{3}{2}$, the homogeneous \nk structure on $S^3\times S^3$, Example \ref{S3xS3}, is $\Psi_b$ with $b=1$ and the homogeneous \nk structure on $\CP^3$, Example \ref{CP3}, is $\Psi_a$ with $a=\frac{\sqrt{3}}{2}$.
\end{remark}

The main technical tool for proving these theorems is the following general result about first order singular initial value problems. We will appeal to it repeatedly in the paper.

\begin{theorem}
\label{thm:Singular:IVP}
Consider the singular initial value problem
\begin{equation}\label{eq:Singular:IVP}
\dot{y}=\frac{1}{t}M_{-1}(y)+M(t,y), \qquad y(0)=y_0,
\end{equation}
where $y$ takes values in $\R^k$, $M_{-1}\co \R^k\ra \R^k$ is a smooth function of $y$ in a neighbourhood of $y_0$ and $M\co\R\times\R^k\ra\R^k$ is smooth in $t,y$ in a neighbourhood of $(0,y_0)$. Assume that
\begin{enumerate}
\item $M_{-1}(y_0)=0$;
\item $h\text{Id}-d_{y_0}M_{-1}$ is invertible for all $h \in \N$, $h \geq 1$.
\end{enumerate}
Then there exists a unique solution $y(t)$ of \eqref{eq:Singular:IVP}. Furthermore $y$ depends continuously on $y_0$ satisfying (i) and (ii).
\end{theorem}

The condition (ii) guarantees the existence of a unique formal power series solution $y(t)$ to \eqref{eq:Singular:IVP}. Once a formal power series solution has been shown to exist, one can follow the arguments of \cite[Theorem 7.1]{Malgrange}, \cite[\S 5]{Eschenburg:Wang} and \cite[\S 4]{Ferus:Karcher}: use a truncation of the power series of sufficiently high degree as an approximate solution to \eqref{eq:Singular:IVP} and deform it to a genuine solution by applying a contraction mapping fixed point argument. As for the continuous dependence on the initial conditions, one argues as in \cite[\S 4]{Ferus:Karcher}: the coefficients of the formal power series solution $y(t)$ depend differentiably on $y_0$ satisfying (i) and (ii) and the operator used in the fixed point argument is uniformly contracting with respect to the initial conditions.

\proof[Proof of Theorem \ref{thm:Singular:IVP:S^2}]
We first reformulate the problem in the form \eqref{eq:Singular:IVP}.  

By Lemma \ref{lem:Smooth:extension:S2}.(i) if $\Psi$ is a solution to \eqref{eq:nk:odes:t} that extends smoothly along a singular orbit $S^2 = \sunitary{2} \times \sunitary{2}/\unitary{1} \times \sunitary{2}$ we can write
\[
\begin{gathered}
u_0(t)=a^2 + t^2 y_1 (t), \qquad u_1 (t) = a^2 + t^2 y_2 (t),\qquad u_2(t) = t^2 y_3(t),\\
v_0 (t) = t^2 y_4(t), \qquad v_1(t)=t^2 y_5(t), \qquad v_2(t) = t^2 y_6(t), \qquad \lambda (t) = t\, y_7(t),
\end{gathered}
\]
for some $a > 0$.

In terms of $y=(y_1, \dots, y_7)$ the system \eqref{eq:nk:odes:t} becomes
\begin{gather*}
\begin{align*}
&\dot{y}_1=-\frac{1}{t}\left( 2y_1 + 3\frac{y_4}{y_7} \right),\qquad &\dot{y}_4=-\frac{1}{t}\left( 2y_4 - 4a^2 y_7 \right) + 4t y_1 y_7,\\
&\dot{y}_2=-\frac{1}{t}\left( 2y_2 + 3\frac{y_5}{y_7}-2y_7\right),\qquad &\dot{y}_5=-\frac{1}{t}\left( 2y_5 - 4a^2 y_7 \right) + 4t y_2 y_7,\\
&\dot{y}_3=-\frac{1}{t}\left( 2y_3 + 3\frac{y_6}{y_7}\right),\qquad &\dot{y}_6=-\frac{1}{t}\left( 2y_6 + 3\frac{y_3}{y_7} \right) + 4t y_3 y_7,
\end{align*}\\
\dot{y}_7=-\frac{1}{t}\left( y_7 + \frac{y_7^2}{y_2-y_1} + \frac{3y_3 y_6}{2a^2y_7^2 (y_2-y_1)}\right) + M_7 (t,y),
\end{gather*}
where
\[
\begin{gathered}
tM_7 (t,y) = \frac{y_7^2}{y_2-y_1} + \frac{3y_3 y_6}{2a^2y_7^2 (y_2-y_1)} -\\
\frac{2y_7^2(a^2+t^2y_2)}{2a^2(y_2-y_1)+t^2(-y_1^2+y_2^2+y_3^2)} - \frac{3y_3y_6}{y_7^2\left( 2a^2(y_2-y_1)+t^2(-y_1^2+y_2^2+y_3^2) \right)}.
\end{gathered}
\]
Hence $y$ solves an ODE system as in \eqref{eq:Singular:IVP}.

It remains to check the hypotheses of Theorem \ref{thm:Singular:IVP}. First, the requirement $M_{-1}\left( y_0 \right) =0$ uniquely determines the initial condition $y_0$ in terms of $a$:
\begin{equation}
\label{eq:initial:condition:y:a0}
y_0 = \left( -3a^2, -3a^2+\frac{3}{2}, -\frac{3\sqrt{3}}{2}a, 3a^2, 3a^2, \frac{3\sqrt{3}}{2}a, \frac{3}{2} \right),
\end{equation}
where we assumed $u_2 <0$ without loss of generality.

For $y$ sufficiently close to $y_0$, $a$ bounded away from $0$ and $t$ sufficiently small we see that $M_{-1}$ and $M$ are real analytic functions of all of their arguments and depend smoothly on $a>0$.

Finally, the linearisation of $M_{-1}$ at $y_0$ is 
\[
d_{y_0}M_{-1}=\left( \begin{array}{ccccccc}
-2 & 0 & 0 & -2 & 0 & 0 & 4a\\
0 & -2 & 0 & 0 & -2 & 0 & 2+4a\\
0 & 0  & -2  & 0 & 0 & -2 & 2\sqrt{3a}\\
0 & 0 & 0 & -2 & 0 & 0 & 4a\\
0 & 0 & 0 & 0 & -2 & 0 & 4a\\
0 & 0 & -2 & 0 & 0 & -2 & -2\sqrt{3a}\\
1 & -1 & -\frac{2}{\sqrt{3a}} & 0 & 0 & \frac{2}{\sqrt{3a}} & -7\\
\end{array}\right).
\]
The determinant of $h \text{Id} - d_{y_{0}}M_{-1}$ is
\[
(h+1)(h+2)^3(h+4)^3 \geq 512 > 0 
\]
for all integer $h \geq 0$.
\endproof

\begin{remark}\label{rmk:Power:Series:a0}
The first few terms of the Taylor series of $\Psi_{a}$ at $t=0$ are:
\begin{subequations}
\begin{gather*}
\lambda(t) = \tfrac{3}{2}t  - \frac{2a^2+3}{12a^2} t^3 + \frac{116a^4 - 381a^2 + 261}{1440a^4}t^5 + \frac{5500a^6 - 26523a^4 + 34209a^2 - 13149}{90720a^6}t^7 + \dots,\\
u_0(t)= a^2  -3a^2 t^2 + \frac{52a^2-3}{24}t^4 -\frac{172a^4 + 3a^2 - 18}{270a^2} t^6 + \dots,\\
u_1(t)= a^2 -\tfrac{3}{2}(2a^2-1)t^2 + \frac{52a^4 - 32a^2 - 3}{24a^2} t^4 -\frac{2752a^6 - 1688a^4 + 93a^2 - 261}{4320a^4}t^6+ \dots,\\
u_2(t) = \tfrac{-3\sqrt{3}}{2}a t^2 + \frac{\sqrt{3}(16 a^2 - 3)}{12 a} t^4 + \frac{\sqrt{3} (-3412 a^4  + 267 a^2 + 423)}{8640 a^3}t^6 + \dots,\\
v_0(t) = 3a^2 t^2 - \left(\frac{1}{4}+\frac{14a^2}{3}\right)t^4 + \frac{5516a^4+429a^2+261}{2160a^2}t^6 + \dots,\\
v_1(t)= 3a^2 t^2 + \left(2 - \frac{14a^2}{3}\right)t^4 + \frac{5516a^4-2541a^2-549}{2160a^2}t^6 + \dots,\\
v_2(t)= \tfrac{3}{2}\sqrt{3}at^2 - \frac{\sqrt{3}(34a^2-3)}{12a}t^4 + \sqrt{3}\frac{13492a^4+273a^2 -423}{8640 a^3}t^6+ \dots.
\end{gather*}
\end{subequations}
The particular form of the coefficients will play an important role in Proposition \ref{prop:Limit:large:a0}.
\end{remark}

The proof of Theorem \ref{thm:Singular:IVP:S^3} is similar, \cf also \cite[Theorem 6.4]{Podesta:Spiro:II} and Proposition \ref{prop:Bubbling:AC:CY}.

\section{The orbital volume function and maximal volume orbits}\label{sec:Orbital:volume}

It follows from previous work that any complete $\sunitary{2}\times \sunitary{2}$--invariant \nk $6$--manifold must arise as the completion 
of some element (or possibly an element from both) of the two $1$--parameter families 
$\{ \Psi_a \}_{a>0}, \{ \Psi_b\}_{b>0}$ constructed in the previous section. 
Our strategy to understand whether any of these solutions closes smoothly along a \emph{second} singular orbit at some time $T>0$ 
will be to consider pairs of solutions from these two families and to try to match them across a principal orbit. In this section we find a geometrically preferred slice to use as a tool in this matching. 

To this end we study the properties of the orbital volume function $V(t)$, \ie the volume of the hypersurfaces $\{ t \} \times \Ndiag$, 
for these two $1$--parameter families of solutions. 
The main result of this section, Proposition \ref{prop:Existence:maximal:volume:orbit}, establishes that the orbital volume function 
of every solution constructed in Theorems \ref{thm:Singular:IVP:S^2} and \ref{thm:Singular:IVP:S^3} has a unique (strict) maximum. 
An important ingredient of the proof of Proposition \ref{prop:Existence:maximal:volume:orbit} is 
Proposition \ref{prop:Maximal:volume:orbits} which establishes key properties of the space of possible maximal volume orbits $\calv$, 
foremost of which are: (i) the orbital volume restricted to $\calv$ is bounded below by the volume of the invariant \se structure on $\Ndiag$ and is 
achieved only by a ``rotated" invariant \se structure; (ii) for any $C\ge 1$ the subset of maximal volume orbits with volume 
bounded above by $C$ is compact and nonempty.

Proposition \ref{prop:Existence:maximal:volume:orbit}  allows us to define two curves $\alpha$ and $\beta: (0,\infty) \to \calv$
that parametrise the maximal volume orbits of the two $1$--parameter families $\{ \Psi_a \}_{a>0}$ and $\{ \Psi_b\}_{b>0}$ respectively 
inside the space of all maximal volume orbits. 
The maximal volume orbit is the preferred principal orbit on which to investigate matching conditions; such matching conditions are then
described in Lemmas \ref{lem:Doubling} and \ref{lem:Matching} (which we call the Doubling and Matching Lemmas respectively)  
in terms of properties of the curves $\alpha$ and $\beta$ (or a certain projection of them). 
Establishing enough information about the curves $\alpha$ and $\beta$ to prove that they satisfy the conditions 
of the Doubling or Matching Lemmas in some cases will occupy the rest of the paper.
\medskip

According to Proposition \ref{prop:Invariant:metrics}, the volume of $\Ndiag$ with respect to the metric induced by an invariant \sunitary{2}--structure 
$\psi_{\lambda,\mu,A}$ has the simple form $V = V_0 \,\lambda \mu^2$, where $V_0$ is the volume of $\Ndiag$ with respect to the standard \se metric. For notational convenience in the rest of the paper we will write $V$ for $(V/V_0)$. While it is possible to derive evolution equations for the orbital volume function $V$ and related quantities directly from the fundamental ODE system \eqref{eq:nk:odes:t}, for some purposes 
it is more convenient to consider the system of ODEs describing arbitrary cohomogeneity one Einstein metric: 
computations will be minimised and, more importantly, we will be able to recognise often complicated algebraic expressions 
involving the \sunitary{2}--structure $\psi$ as basic geometric quantities.

\subsection{The Einstein equations for families of equidistant hypersurfaces}\label{sec:Einstein:odes}

Let $g_t$ be a family of Riemannian metrics on an $n$--dimensional oriented manifold $N$ and consider the metric $\hat{g}$ on $\R_t \times N$ defined by $\hat{g}=dt^2 + g_t$. 
We recall the derivation of the Einstein equations for $\hat{g}$ given by Eschenburg--Wang in \cite[Proposition 2.1]{Eschenburg:Wang}.

Denote by $\nu$ the unit normal of the family of hypersurfaces $N_t:=\{ t \} \times N$ and by $L$ the Weingarten operator given by
\[L(X)=\hat{\nabla} _X \nu, \quad \text{for every\ } X \in TN,
\] 
where $\hat{\nabla}$ is the Levi-Civita connection of the metric $\hat{g}$. Then $L=\frac{1}{2}g^{-1}g'$ and 
\begin{equation}\label{eqn:Riccati}
L'+L^2+\hat{R}_\nu=0
\end{equation}
is the Riccati equation. Here $\hat{R}_\nu$ is the normal-tangential component of the curvature of $\hat{g}$, \ie $\hat{R}_\nu(X)=\hat{R}(X,\nu)\nu$ for every $X \in TN$, and $'$ denotes differentiation with respect to $t$.

Assume now that $\hat{g}$ is Einstein with scalar curvature $(n+1)\Lambda$. Then the Gauss equation for the hypersurface $N_t$ implies that
\begin{equation}\label{eqn:Einstein:odes}
L'+lL-r+\Lambda\id =0,
\end{equation}
where $l=\trace{L}$ is the mean curvature of the hypersurface $N_t$ and $r$ is the Ricci-endomorphism of the metric $g_t$. Moreover, 
if we regard $L$ as a $TN$--valued $1$--form and $d^\nabla$ is the exterior differential induced by the Levi-Civita connection of $g_t$
then the vanishing of the Ricci curvature in normal-tangential directions can be expressed using the Codazzi equation as 
\begin{equation}
\label{eq:trace:codazzi}
\trace (X \lrcorner d^\nabla L) =0, \quad \text{for all\ } X \in TN.
\end{equation}
Taking the trace of both \eqref{eqn:Riccati} and \eqref{eqn:Einstein:odes} yields
\begin{subequations}\label{eqn:Mean:Curvature:odes}
\begin{gather}
l'+|L|^2+\Lambda =0,\\
\Scal -(n-1)\Lambda = l^2 - |L|^2,
\end{gather}
\end{subequations}
where $\Scal$ is the scalar curvature of $g_t$ and $|L|^2=\trace{L^2}$.

The Einstein equations are therefore a first order system for a pair $(g,L)$ consisting of a Riemannian metric $g$ on $N$ and a symmetric endomorphism $L$ of $TN$ subject to the additional constraints \eqref{eq:trace:codazzi} and (\ref{eqn:Mean:Curvature:odes}a). (\ref{eqn:Mean:Curvature:odes}b) is then a conserved quantity of this system.

Consider now a solution $\Psi = (\psi _t)_{0 < t < T}$ of the fundamental ODE system \eqref{eq:nk:odes:t}. Each invariant nearly hypo structure $\psi _t$ on $\Ndiag$ determines an invariant metric $g$. Furthermore, differentiating the map from invariant \nh structures to invariant metrics along the direction of the vector field \eqref{eq:nk:odes:t} defines the Weingarten operator $L$. The pair $(g,L)$ then satisfies the system above with $n=5=\Lambda$.

The positivity of $\Lambda$ has a number of immediate consequences.  

\begin{prop}\label{prop:crit:vol}
Let $\Psi$ be a solution of the fundamental ODE system \eqref{eq:nk:odes:t} defined on $(-T_1,T_2)$ for $T_1,T_2>0$. Moreover, let $V$ be the orbital volume function $V=\lambda\mu^2$, where $\mu^2 = |u|^2$ in the parametrisation of \eqref{eq:invt:nh}, and $l$ be the mean curvature $l=\trace{L}=\frac{\dot{V}}{V}$.
\begin{enumerate}
\item Every critical point of $V$ is a strict maximum.
\item $T_1+T_2\leq \pi$.
\item Fix $t_0 \in (0,\pi)$ such that $l(0)=5\cot{(t_0)}$. Then
\begin{align*}
l(t)  & \leq  5 \cot{(t+t_0}), \quad \text{for\ }t>0;\\
l(t)  & \geq  5 \cot{(t+t_0)}, \quad \text{for\ }t<0.
\end{align*}
\item For the same $t_0$ as above and any $t \in (-T_1,T_2)$
\[
V(t) \leq V(0) \frac{ \sin^5{(t+t_0)} }{ \sin^5{(t_0)} }.
\]
\end{enumerate}
\proof
Since $V'=lV$ and therefore $V''=l'V$ at any critical point of $V$,  (\ref{eqn:Mean:Curvature:odes}a) implies (i).

Decomposing $L$ into its trace $l$ and its traceless part $\mathring{L}$, we write $|L|^2=\frac{1}{5}l^2 + |\mathring{L}|^2$. Then \eqref{eqn:Mean:Curvature:odes} implies that $l$ satisfies the inequality 
\[l'+\frac{1}{5}l^2 + 5 \leq 0.
\]
Comparing $l$ with a solution of the scalar Riccati-type equation $l'+\frac{1}{5}l^2 + 5 = 0$ we obtain (ii) and (iii) by \cite[Theorem 4.1]{Eschenburg:Comparison:hypersurfaces}. Direct integration of $V' = lV$ then yields (iv).
\endproof
\end{prop}

We will now write down explicit formulae for the geometric quantities appearing in (\ref{eqn:Mean:Curvature:odes}b) in terms of an invariant \nh structure $\psi_{\lambda,\mu,A}$ as an illustration of how complicated these algebraic expressions can be. Two of these explicit formulae will be useful in the next section.

\begin{lemma}\label{lem:Scalar:curvature:nearly:hypo}
Let $\psi_{\lambda,\mu,A}$ be one of the invariant nearly hypo structures on $\Ndiag$ parametrised in Proposition \ref{prop:Invariant:nearly:hypo:structures}. Let $x_i$, $y_i$ and $w_i$ be given by \eqref{eqn:invt:nh:A} and, acting with the flow of the Reeb vector field if necessary, assume that $\theta=0$.
\begin{enumerate}
\item The scalar curvature of the induced invariant metric $g$ on $
\Ndiag$ is
\[
\Scal = 20 - 4\frac{\lambda ^2 x_1^2}{\mu^2} + 24\frac{x_1 y_2}{\mu}-4\frac{\lambda^2w_1^2}{\mu ^2} -9\frac{w_2^2}{\lambda^2}.
\]
\item The mean curvature $l$ is
\[
l=2\frac{x_1}{y_1}-3\frac{y_2}{x_2}.
\]
\item The norm squared of the traceless part of the Weingarten operator $L$ is
\[
|\mathring{L}|^2= \frac{36}{5}\left( \frac{\lambda x_1}{\mu } -\frac{y_2}{\lambda}\right) ^2 + 4\frac{\lambda^2 w_1^2}{\mu ^2} +9\frac{w_2^2}{\lambda^2}.
\]
\end{enumerate}
\proof
First we give an expression for the scalar curvature of a general invariant \sunitary{2}--structure on $\Ndiag$ and then specialise to the case of invariant \nh structures. Bedulli and Vezzoni \cite[Theorem 3.4]{Bedulli:Vezzoni} give a formula for the scalar curvature of the metric induced by an \sunitary{2}--structure on a $5$--manifold.
Since $\etase$ spans the space of invariant $1$--forms on $\Ndiag$ and invariant functions on $\Ndiag$ are constant, this formula simplifies considerably 
in the case of invariant \sunitary{2}--structures:
\begin{equation}\label{eqn:Scalar:curvature}
\Scal = -5\phi_{11}^{2} - \sum _{i=1}^3 { \phi_i ^2 } -4\phi_1\phi_{23} -4\phi_2\phi_{31} -4\phi_3\phi_{12}-\sum_{i=0}^3{\frac{1}{2}|\sigma_i|^2}, 
\end{equation}
where $\phi_i$, $\phi_{ij}$ and $\sigma_i$ are the \sunitary{2}--irreducible components of the intrinsic torsion $\Theta$ of the \sunitary{2}--structure as described 
in Conti--Salamon \cite[Proposition 2.3]{Conti:Salamon}.
These irreducible components of $\Theta$ are determined by the exterior derivatives of $\eta$ and the triple $\omega_i$ via 
\begin{equation}
\label{eq:torsion:su2:deriv}
d\omega _{i} = \alpha _{i} \wedge \omega _{i} + \sum _{j=1}^{3}{\phi _{ij} \eta \wedge \omega _{j}} + \sigma _{i}, \qquad d\eta = \eta \wedge \alpha _{4} + \sum _{i=1}^{3}{\phi_{i} \omega _{i}} + \sigma _{4}.
\end{equation}
To apply these two formulae we need to express the components of the intrinsic torsion of the invariant \sunitary{2}--structure $\psi_{\lambda,\mu,A}$ in terms of the parameters 
$(\lambda, \mu, A) \in \R ^{+} \times \R ^+ \times \lorentz{3}$. Recall that from equations \eqref{eq:eta:invt} and \eqref{eq:omega:invt} we have
$\eta = \lambda \etase$, $\omega _i = \mu A \omegase_i$. From \eqref{eqn:Torsion:invariant:structures} we obtain
\[
d\eta = -2\frac{\lambda}{\mu}\omega _1, \qquad d\omega _i = \frac{1}{\lambda} \eta \wedge (A^{-1}TA)\omega_i,
\]
where $T$ is defined in \eqref{eqn:Torsion:standard:SE}. 

We now specialise to the case of invariant \nh \sunitary{2}--structures. An explicit computation of $A^{-1}$ and $A^{-1}TA$ yields
\[
\begin{gathered}
d\eta = -2\frac{\lambda}{\mu}\left( -w_1 \omega_0 + x_1 \omega_1 +\frac{\mu}{\lambda} \omega_3 \right), \qquad d\omega _1 = 3\eta \wedge \omega_2,\\
d\omega_2 = \eta \wedge \frac{1}{\lambda }\left( -3 w_2 \omega_0 - 3 \lambda \omega_1 +3 y_2 \omega_3 \right), \qquad d\omega _3 = -3\frac{y_2}{\lambda}\eta \wedge \omega_2.
\end{gathered}
\]
The first part of the lemma follows from \eqref{eqn:Scalar:curvature} and \eqref{eq:torsion:su2:deriv} using the fact that $|\omega_0|^2=2$ with respect to $g$.

The expression for the mean curvature $l$ follows from \eqref{eq:nk:odes:t} and the change of variables \eqref{eq:invt:nh} since $l=\frac{\dot{V}}{V}=\frac{\dot{\lambda}}{\lambda}+2\frac{\dot{\mu}}{\mu}$. Once (i) and (ii) are known, (iii) follows from (\ref{eqn:Mean:Curvature:odes}b).
\endproof
\end{lemma}

Formulae in the parametrisation $(\lambda, u, v)$ of \eqref{eq:invt:nh} follow immediately by applying the change of variables \eqref{eq:change:parametrisation:nh}.

\subsection{The space of invariant maximal volume orbits}
If a solution of \eqref{eq:nk:odes:t} gives rise to a \emph{complete} invariant \nk metric on a closed $6$--manifold $M$ then clearly it would 
contain a unique orbit of maximal volume. 
In general the two $1$--parameter families of invariant \nk metrics provided by Theorems \ref{thm:Singular:IVP:S^2} and \ref{thm:Singular:IVP:S^3} give rise to \emph{incomplete} invariant \nk metrics. Nevertheless, in this section we show that every solution $\Psi$ in the two \mbox{$1$--parameter} families of Theorems \ref{thm:Singular:IVP:S^2} and \ref{thm:Singular:IVP:S^3} has a unique \emph{maximal volume orbit}, \ie the orbital volume function $V=\lambda\mu^2$ reaches a strict maximum before $\Psi$ blows up.

Let $l$ be the mean curvature of an invariant \nh structure $\psi \in \calu \times \sorth{2}$, where $\calu$ is the open set of $\lorentz{2}$ defined in Proposition \ref{prop:Invariant:nearly:hypo:structures}. Here for any $\delta>0$ sufficiently small we extend $\psi$ to an invariant nearly K\"ahler structure $\Psi$ on $(-\delta,\delta) \times \Ndiag$ by Proposition \ref{prop:nk:odes} and define $l$ by $\dot{V}(0)=l\,V(0)$. In this manner, we regard $l$ as a smooth function on the space of invariant \nh structures and denote by $\calv$ its zero level set. $\calv$ is clearly invariant under the circle action generated by the Reeb vector field and therefore, in the parametrisation of Proposition \ref{prop:Invariant:nearly:hypo:structures}, $\calv = \calv_0 \times \sorth{2}$ for a subset $\calv_0$ of $\calu$.

\begin{prop}
\label{prop:Maximal:volume:orbits}
The space $\calv = \calv_0 \times \sorth{2}$ of invariant nearly hypo structures on $\Ndiag$ that satisfy the additional constraint $l=0$ is a smooth manifold.

Let $\pi\co\calv_0 \rightarrow \R^+ \times \R^+$ be the map given by $\psi_{\lambda,\mu,A} \mapsto (\lambda,\mu)$. Then $\pi (\psi)=\pi(\psi')$ if and only if $\psi'$ lies in the orbit of $\psi$ under the group of discrete symmetries generated by the involutions $\tau_1, \tau_2, \tau_3, \tau_4$ of Proposition \ref{prop:symmetries}. The image of $\calv_0$ under $\pi$ is the wedge 
\[W=\{ (\lambda, \mu)\in \R^+ \times \R^+ \, |\, \mu \geq \lambda \geq 1 \}
\] 
and the projection $\pi\co \calv_0 \ra W$ is a $4$--fold cover branched along the boundary of $W$.

In particular, the volume $V=\lambda\mu^2$ of an invariant \nh structure $\psi \in \calv$ is bounded below by $1$ with equality if and only if $\psi$ is an invariant \se structure on $\Ndiag$ and the set $\calv \cap \{ V \leq C\}$  is compact for any $C \geq 1$.
\proof
By the first equation in \eqref{eqn:Mean:Curvature:odes} at any point the differential of $l$ evaluated on the tangent vector corresponding to the evolution equations \eqref{eqn:Nearly:Hypo:evolution} is strictly negative and therefore $\calv$ is a smooth submanifold of $\calu \times \sorth{2}$. We will now describe $\calv$ in detail.

Let $\psi$ be an invariant \nh structure written in the form \eqref{eq:invt:nh} with $\theta=0$. By Lemma \ref{lem:Scalar:curvature:nearly:hypo}(ii) and the change of variables \eqref{eq:change:parametrisation:nh}, $\psi$ is a critical point of $V$ if and only if
\begin{equation}
\label{eq:crit:vol}
2\lambda^4 u_1 = 3u_2v_2.
\end{equation}
We now use \eqref{eq:crit:vol} to rewrite the constraints \eqref{eq:constraints} in terms of $\lambda$, $\mu$ and $u_1$. 
First of all, substituting $u_2 = -\lambda \mu$ in \eqref{eq:crit:vol} yields
\begin{equation}
\label{eq:v2:crit}
v_2 = -\frac{2\lambda^3 u_1}{3\mu}.
\end{equation}
Notice that the inequality \eqref{eq:constraints:sign} then reads
\begin{equation}
\label{eq:V:crit:sign}
u_1^2<\frac{3\mu^4}{2\lambda^2}.
\end{equation}
We now substitute the expression \eqref{eq:v2:crit} for $v_2$ along with those for $u_2$ and $v_1$ in terms of $\lambda$ and $\mu$ into the constraints \eqref{eq:constraints}. We obtain
\begin{subequations}
\label{eq:V:crit}
\begin{gather}
u_0^2 = u_1^2 + \mu^2(\lambda^2-1),\\
\frac{9\mu^2}{4\lambda^6} v_0^2 = u_1^2 +  \frac{9\mu^4}{4\lambda^6}(\mu^2-\lambda^2),\\
\frac{3\mu}{2\lambda^3} u_0 v_0 = \frac{\mu}{2\lambda^3} u_1 (3\mu^2+2\lambda^4).
\end{gather}
\end{subequations}

Equating the product of the first two equations with the square of the third we see that $x=u_1^2$ satisfies the quadratic equation
\begin{equation}
\label{eq:u1:quadratic}
x^2-cx+d=0,
\end{equation}
where $c = \mu^2\left(1+\frac{3\mu^2}{\lambda^2}+\frac{9\mu^2}{4\lambda^4}\right)>0$ and $d=\frac{9\mu^6}{4\lambda^6}(\lambda^2-1)(\mu^2-\lambda^2)$. The discriminant $\Delta = c^2-4d$ of the quadratic \eqref{eq:u1:quadratic} is
\[
\Delta = \frac{\mu^4}{\lambda^8} \left( \frac{45}{2}\mu^4\lambda^2 + 15 \mu^2\lambda^6 + \frac{81}{16} \mu^4 + \lambda^8 - \frac{9}{2} \mu^2 \lambda^4\right) \ge \frac{\mu^4}{\lambda^8} \left( \frac{45}{2} \mu^4 \lambda^2 + 15 \mu^2 \lambda^6\right) \ge 0,
\]
where equality holds if and only if $\lambda=\mu=0$. Hence \eqref{eq:u1:quadratic} always has two distinct real roots.

The larger root $x_+$ always has to be discarded because it contradicts \eqref{eq:V:crit:sign}. When $d<0$ the smallest root $x_-=\frac{1}{2}(c-\sqrt{\Delta})<0$. Assume then that $d\geq 0$. We distinguish two cases: (i)  $\mu < \lambda <1$; (ii) $\mu \geq \lambda \geq 1$.

In case (i) we show that the choice $x=x_-$ is incompatible with \eqref{eq:V:crit}. Indeed, if the first two equalities of \eqref{eq:V:crit} were satisfied with this choice of $u_1^2$, then we must have $c + 2\mu^2 (\lambda^2 -1) \geq \sqrt{\Delta}$ and $c + \tfrac{9\mu^4}{2\lambda^6} (\mu^2-\lambda^2) \geq \sqrt{\Delta}$. However, the squares of the expressions on the left hand side of these inequalities are $\Delta + \tfrac{\mu^4}{\lambda^6}(3\mu^2 + 2 \lambda^4)^2(\lambda^2-1)$ and $\Delta + \tfrac{9\mu^6}{4\lambda^{12}}(3\mu^2 + 2 \lambda^4)^2(\mu^2 - \lambda^2)$, both strictly less than $\Delta$ by hypothesis.

In case (ii) the solution $x=x_-$ to \eqref{eq:u1:quadratic} is both non-negative and compatible with \eqref{eq:V:crit:sign} and \eqref{eq:V:crit}. Thus an admissible solution to the quadratic equation \eqref{eq:u1:quadratic} only exists (and is unique) when $(\lambda, \mu) \in W$. In this case $u_1$ is determined up to sign by the pair $(\lambda, \mu)$. The first two equations in \eqref{eq:V:crit} therefore determine $u_0$ and $v_0$ up to sign in terms of $(\lambda, \mu)$. Acting by $\tau_4\circ\tau_1$ if necessary we can assume without loss of generality that $u_1 \geq 0$. Then $u_1, u_2=-\lambda \mu, v_1=\mu^2, v_2, |u_0|, |v_0|$ are completely determined by $(\lambda, \mu)$. Moreover, the third equation of \eqref{eq:V:crit} implies that $u_0 v_0 \geq 0$. The involution $\tau_2 \circ \tau_3 \circ \tau_4$ maps $(u_0,v_0) \mapsto (-u_0,-v_0)$ while keeping $\lambda,u_1,u_2,v_1,v_2$ fixed and therefore can be used to remove the remaining ambiguity in the choice of sign of $u_0, v_0$.
\endproof
\end{prop}

Thanks to the compactness statement we deduce the following crucial proposition.

\begin{prop}
\label{prop:Existence:maximal:volume:orbit}
Let $\Psi$ be one of the solutions of \eqref{eq:nk:odes:t} given by Theorems \ref{thm:Singular:IVP:S^2} and \ref{thm:Singular:IVP:S^3}. Then $\Psi$ intersects $\calv$ in a unique point.
\proof
The proof is identical in the two cases and we therefore give it only for the family $\Psi_{a}$. Let $S$ be the set of $a \in (0,\infty)$ such that $\Psi_{a}$ intersects $\calv$. The solution of \eqref{eq:nk:odes:t} of Example \ref{S6}, \ie the standard \nk structure on $S^6$, admits a maximal volume orbit and a singular orbit $S^2$ (and $S^3$). Hence $S$ is non-empty.

$S$ is open: by Proposition \ref{prop:crit:vol}(i) every nearly hypo structure with $l=0$ is a non-degenerate maximum of the orbital volume function. Moreover, by Proposition \ref{prop:Maximal:volume:orbits}, $l=0$ defines a smooth hypersurface in the space of invariant \nh structures and therefore if $\Psi_{a}$ has a maximal volume orbit, so does $\Psi_{a'}$ for any $a'$ sufficiently close to $a$.

$S$ is closed: suppose that a sequence $a_i$ in $S$ converges to some $a \in (0, \infty)$. By the continuous dependence on the initial conditions, we can find some time $t>0$ sufficiently small such that the orbital volume $V(t)$ and the mean curvature $l(t)$ remain uniformly bounded for all $a_i$. By Proposition \ref{prop:crit:vol}(iii), the maximal volume orbits of $\Psi_{a_i}$ have uniformly bounded volume. By Proposition \ref{prop:Maximal:volume:orbits} the set of maximal volume orbits with an upper bound on the volume is compact and therefore $\Psi_{a}$ must also contain a maximal volume orbit.

We conclude that $S=(0,\infty)$. The intersection point is unique by Proposition \ref{prop:crit:vol}(i).
\endproof 
\end{prop}

Since the solutions $\Psi_a, \Psi_b$ of Theorems \ref{thm:Singular:IVP:S^2} and \ref{thm:Singular:IVP:S^3} depend continuously on $a,b>0$, respectively, Proposition \ref{prop:Existence:maximal:volume:orbit} yields two continuous curves $\alpha, \beta\co (0,\infty) \ra \calv_0$. Using Proposition \ref{prop:Maximal:volume:orbits} we project these two curves onto the wedge $W$.

\begin{definition}\label{def:Maximal:volume:orbits}
Let $\alphaW, \betaW\co (0,\infty) \ra W$ be the two continuous curves which parametrise the maximal volume orbit of the solutions $\Psi_{a}$, $a >0$, and $\Psi_{b}$, $b >0$, respectively, up to the action of discrete symmetries.
\end{definition}
The fact that $\alphaW$ and $\betaW$ parametrise maximal volume orbits up to discrete symmetries follows from Proposition \ref{prop:Maximal:volume:orbits}.

We now use the curves $\alphaW$ and $\betaW$ to state conditions for pairs of solutions $\Psi_\epsilon$, $\Psi_{\epsilon'}$ to match across their maximal volume orbit and therefore define a complete invariant \nk structure on closed manifold. The task of the rest of paper is to study the behaviour of the curves $\alphaW$ and $\betaW$ to establish whether these matching conditions are ever satisfied.

\subsection{Doubling and matching}

Suppose that $\Psi_1, \Psi_2$ are two of the solutions of \eqref{eq:nk:odes:t} given by Theorems \ref{thm:Singular:IVP:S^2} and \ref{thm:Singular:IVP:S^3}. In the parametrisation of \eqref{eq:invt:nh} we write
\[
\begin{gathered}
\Psi_1(t)=\Big( \lambda(t), u_0(t),u_1(t),u_2(t), v_0(t), v_1(t), v_2(t)\Big),\\
\Psi_2(t)=\left( \hat{\lambda}(t), \hat{u}_0(t), \hat{u}_1(t), \hat{u}_2(t), \hat{v}_0(t), \hat{v}_1(t), \hat{v}_2(t)\right).
\end{gathered}
\]
Suppose that the maximal volume orbit of $\Psi_1, \Psi_2$ occurs at $t=T_1,T_2$, respectively, and that $\left( \lambda(T_1),\mu(T_1)\right) = \left( \hat{\lambda}(T_2), \hat{\mu}(T_2) \right)$, \ie the two maximal volume orbits coincide up to the action of discrete symmetries. In particular, $u_i (T_1), v_i(T_1)$ coincide with $\hat{u}_i(T_2), \hat{v}_i(T_2)$ up to some sign.

Acting by $\tau_2\circ\tau_3\circ\tau_4$ on $\Psi_1$ we define
\[
\widetilde{\Psi}_1(t) = \Big( \lambda(t), -u_0(t), u_1(t), u_2(t), -v_0(t), v_1(t), v_2(t)\Big).
\]
By the uniqueness of solutions of \eqref{eq:nk:odes:t} with given initial conditions, if $u_1(T_1)=\hat{u}_1(T_2)$ then necessarily $\Psi_2=\Psi_1$ or $\Psi_2=\widetilde{\Psi}_1$ depending on the sign of $u_0 (T_1), \hat{u}_0(T_2)$. Thus we assume without loss of generality that $u_1(T_1)=-\hat{u}_1(T_2)$.

Acting by a time translation and $\tau_1 \circ \tau_2 \circ \tau_3$ or $\tau_1 \circ \tau_4$, respectively, define
\[
\Psi_2^\pm(t) = \left( \hat{\lambda}(\tau), \mp \hat{u}_0(\tau), -\hat{u}_1(\tau), \hat{u}_2(\tau), \pm\hat{v}_0(\tau), \hat{v}_1(\tau), -\hat{v}_2(\tau)\right),
\]
where $\tau=T_1+T_2-t$ for $T_1 \leq t \leq T_1+T_2$. By our assumptions either $\Psi_1 (T_1) = \Psi_2^+(T_1)$ and we define a smooth solution of \eqref{eq:nk:odes:t}
\begin{equation}
\label{eq:Matching1}
\Psi (t) = \left\{ \begin{array}{ll}
 \Psi_1(t), & 0 \leq t \leq T_1,\\
 \Psi_2^+(t), & T_1\leq t \leq T_1+T_2,\\
\end{array}\right .
\end{equation}
or $\Psi_1\left(T_1 \right) = \Psi_2^-\left( T_1\right)$ and we consider
\begin{equation}
\label{eq:Matching2}
\Psi (t) = \left\{ \begin{array}{ll}
 \Psi_1(t), & 0 \leq t \leq T_1,\\
 \Psi_2^-(t), & T_1\leq t \leq T_1+T_2.\\
\end{array}\right .
\end{equation}
We deduce the following two lemmas.

\begin{lemma}[Doubling Lemma]
\label{lem:Doubling}
$ $ 

\begin{enumerate}
\item Suppose that there exists $a \in (0,\infty)$ such that $\alphaW (a)$ lies in the portion of the boundary of $W$ with $\lambda =1, \mu >1$. Then \eqref{eq:Matching1} with $\Psi_1=\Psi_2=\Psi_{a}$ defines a smooth invariant \nk structure on $S^2 \times S^4$.
\item Suppose that there exists $a \in (0,\infty)$ such that $\alphaW (a)$ lies in the portion of the boundary of $W$ with $\lambda =\mu >1$. Then \eqref{eq:Matching2} with $\Psi_1=\Psi_2=\Psi_{a}$ defines a smooth invariant \nk structure on $\CP^3$.
\item Suppose that there exists $b \in (0,\infty)$ such that $\betaW (b)$ lies on the boundary of $W$ and $\betaW (b) \neq (1,1)$. Then \eqref{eq:Matching1} or \eqref{eq:Matching2} with $\Psi_1=\Psi_2=\Psi_{b}$ defines a smooth invariant \nk structure on $S^3 \times S^3$.
\end{enumerate}
\proof
The point $(1,1) \in W$ corresponds to the standard Sasaki--Einstein structure on $\Ndiag$ and therefore has to be excluded.

In view of the previous discussion we have only to explain the topology of the underlying manifold $M$ obtained by the gluing construction of \eqref{eq:Matching1} and \eqref{eq:Matching2}.

Recall that $\tau_3$ is induced by the inner automorphism $\phi_3$ of $\sunitary{2} \times \sunitary{2}$ generated by $\frac{\pi}{4}U^{-}$, \ie $\phi_3$ is conjugation by an element of the normaliser $N$ of $\triangle\unitary{1}$ in $\sunitary{2}\times\sunitary{2}$. It is clear that $N=\unitary{1} \times \unitary{1}$ is the maximal torus of $\sunitary{2} \times \sunitary{2}$. In particular $N$ is connected and $\phi_3$ can be deformed to the identity through a path in $N$. On the other hand, $\tau_4$ is induced by the outer automorphism of $\sunitary{2} \times \sunitary{2}$, which fixes both $\triangle\unitary{1}$ and $\triangle\sunitary{2}$ but exchanges $\unitary{1} \times \sunitary{2}$ and $\sunitary{2}\times\unitary{1}$. It is then straightforward to deduce the group diagrams of the resulting cohomogeneity one manifolds. The result follows from Table \ref{tab:Group:diagrams}.
\endproof
\end{lemma}

\begin{remark*}
Observe that the \nk structures of the lemma are obtained by ``doubling'' a solution $\Psi_a$ or $\Psi_b$ across its maximal volume orbit. In particular, the two singular orbits are spheres of the same dimension and have the same volume. Because of the non-trivial action of the discrete symmetries, however, the maximal volume orbit is not totally geodesic. By \cite[Lemma 2.1]{Fernandez:nearly:hypo} this can only occur for the ``rotated'' \se structure in a \nk sine-cone.
\end{remark*}

\begin{lemma}[Matching Lemma]
\label{lem:Matching}
$ $

\begin{enumerate}
\item Suppose that there exists $a<a' \in (0, \infty)$ such that $\alphaW (a)=\alphaW (a')$. Set $\Psi_1 = \Psi_{a}$ and $\Psi_2 = \Psi_{a'}$. Then either \eqref{eq:Matching1} yields a smooth \nk structure on $S^2 \times S^4$ or \eqref{eq:Matching2} defines a smooth invariant \nk structure on $\CP^3$. 
\item Suppose that there exists $b<b' \in (0, \infty)$ such that $\betaW (b)=\betaW (b')$. Set $\Psi_1 = \Psi_{b}$ and $\Psi_2 = \Psi_{b'}$. Then either \eqref{eq:Matching1} or \eqref{eq:Matching2} defines a smooth invariant \nk structure on $S^3 \times S^3$.
\item Suppose that there exist $a, b \in (0,\infty)$ such that $\alphaW (a)=\betaW (b)$. Then either \eqref{eq:Matching1} or \eqref{eq:Matching2} with $\Psi_1=\Psi_{a}$ and $\Psi_2=\Psi_{b}$ defines a smooth invariant \nk structure on $S^6$.
\end{enumerate}
\proof
In all cases $\Psi_2 \neq \Psi_1, \widetilde{\Psi}_1$ and therefore one of the conditions of \eqref{eq:Matching1} or \eqref{eq:Matching2} are satisfied. The topology of the resulting $6$--manifold follows from Table \ref{tab:Group:diagrams}.
\endproof
\end{lemma}


\section{Limits of the two $1$--parameter families of smoothly closing \nk solutions as $a$ or $b$ tend to $0$}\label{sec:limits}

In order to be able to satisfy the conditions of the Doubling or Matching Lemmas \ref{lem:Doubling} and \ref{lem:Matching} we will need 
information about the behaviour of the two curves $\alpha$ and $\beta$ parameterising the maximal volume orbits of the two $1$--parameter families
 $\{ \Psi_a \}_{a>0}$ and $\{ \Psi_b\}_{b>0}$ respectively. In this section we establish properties about the families 
$\{ \Psi_a \}_{a>0}$ and $\{ \Psi_b\}_{b>0}$ in the limit where the size of the singular orbit tends to zero, \ie $a$ or $b \to 0$.
This information will suffice for our applications to constructing a new complete cohomogeneity one \nk structure on $S^3 \times S^3$ 
in Section \ref{sec:S3xS3}; however, to prove the existence of a new complete cohomogeneity one \nk structure on $S^6$ 
we will also need some understanding of the family $\{ \Psi_a\}_{a>0}$ in the limit where $a \to \infty$. 
We turn to this latter problem in Section \ref{sec:S6}.

The four compact $\sunitary{2}\times\sunitary{2}$--manifolds of Table \ref{tab:Group:diagrams} can each be thought of as desingularisation of the sine-cone of Example \ref{Sine:Cone}, where a neighbourhood of each conical singularity is replaced with a copy of either $\mathcal{O}(-1)\oplus\mathcal{O}(-1)\ra S^2$ or of $T^\ast S^3$. By Theorem \ref{thm:AC:CY} both of these carry complete Calabi--Yau structures. 
Proposition \ref{prop:Bubbling:AC:CY} establishes that in the limit where the size of the singular orbit tends to zero, 
suitably rescaled $\Psi_a$ and $\Psi_b$ converge to the Calabi--Yau structures on the small resolution and the conifold respectively, and 
thus confirms our earlier expectation that the two $1$--parameter families $\{ \Psi_a\}_{a>0}$ and $\{ \Psi_b\}_{b>0}$ are \nk deformations 
of the two Calabi--Yau desingularisations of the conifold. Further confirmation of the desingularisation intuition is provided by the main result of this section, 
Theorem \ref{thm:Convergence:sine:cone}: this  establishes that 
on every compact set of $(0,\pi)$ the local \nk structures $\Psi_{a}$ and $\Psi_{b}$ converge to the sine-cone of Example \ref{Sine:Cone} as $a$ and $b \ra 0$.
The proof of this result uses the limiting behaviour of the rescalings of $\Psi_a$ or $\Psi_b$ as $a$ or $b \to 0$ as described above, along with the properties of the B\"ohm functional $\bohm$ 
and results from the previous section about the space of maximal volume orbits $\calv$.

\subsection{Bubbling-off of asymptotically conical Calabi--Yau $3$--folds}

Let $\{ \Psi_a \}_{a>0}$ and $\{ \Psi_b\}_{b>0}$ be the two $1$--parameter families of solutions to \eqref{eq:nk:odes:t} given by Theorems \ref{thm:Singular:IVP:S^2} and \ref{thm:Singular:IVP:S^3}, respectively. Set $\epsilon=a$ or $\epsilon=b$ and consider the scaling
\[
(\omega, \Omega) \longmapsto (\epsilon ^{-2}\omega, \epsilon^{-3}\Omega).
\]
Under such a transformation, the two $1$--parameter families define two families of local \nk structures with $\Scal = 30\epsilon^2$ defined in a neighbourhood of a singular orbit of fixed size. In terms of the parametrisation of \eqref{eq:invt:nh}, a solution $\Psi_{\epsilon}(t) = \left( \lambda(t), u(t), v(t) \right)$ of \eqref{eq:nk:odes:t} defines a solution
\[
\widetilde{\Psi}_\epsilon = \left( \epsilon^{-1}\lambda(\epsilon t), \epsilon^{-2} u(\epsilon t), \epsilon^{-3} v(\epsilon t) \right)
\]
of the ODE system
\begin{subequations}
\label{eq:nk:odes:scaled}
\begin{gather}
\lambda \dot{u}_0 + 3\epsilon v_0=0,\\
\lambda \dot{u}_1 + 3\epsilon v_1 - 2\lambda^2=0,\\
\lambda \dot{u}_2 + 3\epsilon v_2=0,\\
\dot{v}_0 -4\epsilon\lambda u_0=0,\\
\dot{v}_1-4\epsilon\lambda u_1=0,\\
\epsilon\lambda \dot{v}_2 - 4\epsilon^2\lambda^2u_2 + 3u_2=0,\\
\epsilon\lambda^2\abs{u}^2\dot{\lambda} +2\epsilon\lambda^4 u_1 + 3u_2 v_2=0,
\end{gather}
\end{subequations}
together with the constraints 
\begin{subequations}
\label{eq:constraints:scaled}
\begin{gather}
\langle u,v \rangle = 0,\\
u_2=-\epsilon\lambda \mu \label{eq:constraints:scaled:1}\\
\lambda ^2 \abs{u}^2=\abs{v}^2,\\
v_1 = \epsilon\mu^2,
\end{gather}
\end{subequations}
where $\mu^2:=\abs{u}^2$. 

\begin{prop}[Rescaled \nk ``bubbles'' converge to Calabi--Yau structures]\hfill
\label{prop:Bubbling:AC:CY}
\begin{enumerate}
\item For every $a \geq 0$ there exists a unique smooth solution $\widetilde{\Psi}_{a}$ to \eqref{eq:nk:odes:scaled} with $\epsilon=a$ satisfying the initial conditions
\[
\lambda (0)=0=u_2 (0), \qquad u_0 (0)=1=u_1(0), \qquad v_i(0)=0,\quad i=0,1,2.
\]
\item For every $b \geq 0$ there exists a unique smooth solution $\widetilde{\Psi}_{b}$ to \eqref{eq:nk:odes:scaled} with $\epsilon=b$ satisfying the initial conditions
\[
\lambda (0)=1,\qquad u_i (0)=0,\quad i=0,1,2, \qquad v_0(0)=-\frac{2}{3}=-v_2(0),\qquad v_1(0)=0.
\]
\end{enumerate}
Moreover, $\widetilde{\Psi}_{a}$ and $\widetilde{\Psi}_{b}$ depend continuously on $a, b \in [0, \infty)$ and $\lim_{a \ra 0}\widetilde{\Psi}_{a}$, $\lim_{b \ra 0}\widetilde{\Psi}_{b}$ is the asymptotically conical Calabi-Yau structure of Theorem \ref{thm:AC:CY}(i) and (ii),  respectively.
\proof
The proof is analogous to the one of Theorems \ref{thm:Singular:IVP:S^2} and \ref{thm:Singular:IVP:S^3} and we only prove (ii).

For better comparison with Theorem \ref{thm:AC:CY}(ii) and since $\lambda (0)=1$, we introduce a new variable defined by $\frac{ds}{dt}=\frac{1}{\lambda}$. Write
\[
\begin{gathered}
u_0(s)=s y_1 (s), \qquad u_1 (s) = s y_2 (s),\qquad u_2(s) = \epsilon s y_3(s)\\
v_0 (s) = -\frac{2}{3}+s^2 y_4(s), \qquad v_1(s)=s^2 y_5(s), \qquad v_2(s) =\frac{2}{3} +s^2 y_6(s), \qquad \lambda^2 (s) = y_7(s).
\end{gathered}
\]
In terms of the new variables the system \eqref{eq:nk:odes:scaled} becomes
\begin{alignat*}{2}
\dot{y}_1 &=-\frac{y_1 -2b}{s}-3b sy_4,&\qquad  &\dot{y}_4 =-\frac{2y_4 - 4b y_1 y_7}{s},\\
\dot{y}_2 &=-\frac{y_2 -2y_7}{s} -3b sy_5,&\qquad  &\dot{y}_5 =-\frac{2y_5 - 4b y_2 y_7}{s},\\
\dot{y}_3 &=-\frac{y_3 + 2}{s} - 3sy_6,&\qquad  &\dot{y}_6 =-\frac{2y_6 - 4b^2 y_3 y_7 + 3y_3 }{s},
\end{alignat*}
\[
\dot{y}_7=-\frac{4\left( y_2 y_7^2 + y_3 \right)}{s\left( -y_1^2+y_2^2+b^2 y_3^2\right)} -\frac{6s y_3 y_6}{y_1^2+y_2^2+b^2 y_3^2},
\]
which has the form $\dot{y}=\frac{1}{s}M_{-1}(y)+M(s,y)$ of \eqref{eq:Singular:IVP}. Condition (i) in Theorem \ref{thm:Singular:IVP} fixes the initial condition
\[
y_0 = \left( 2b, 2, -2, 4b^2, 4b, 3-4b^2, 1 \right).
\]
Then the linearisation of $M_{-1}$ at $y_0$ is 
\[
d_{y_0}M_{-1}=\left( \begin{array}{ccccccc}
-1 & 0 & 0 & 0 & 0 & 0 & 0\\
0 & -1 & 0 & 0 & 0 & 0 & 2\\
0 & 0  & -1  & 0 & 0 & 0 & 0\\
4b & 0 & 0 & -2 & 0 & 0 & 8b^2\\
0 & 4b & 0 & 0 & -2 & 0 & 8b\\
0 & 0 & 4b^2-3 & 0 & 0 & -2 & -8b^2\\
0 & -1 & -1 & 0 & 0 & 0 & -4\\
\end{array}\right).
\]
Since
\[
\det{\left( h \text{Id} - d_{y(0)}A \right)}=(h+1)^2(h+2)^4(h+3) \neq 0
\]
for all $h \geq 1$ the hypotheses of Theorem \ref{thm:Singular:IVP} are satisfied and the existence of a continuous $1$--parameter family of solutions to \eqref{eq:nk:odes:scaled} follows.

Taking the limit $b\ra 0$ in the equations and initial conditions we immediately obtain that $\widetilde{\Psi}_0$ satisfies $u_0 = u_2 =v_1 =0$ (and therefore $\mu = u_1$), $u'_1=2\lambda^2$ and $v_0 = -\frac{2}{3}$. Moreover, the constraint \eqref{eq:constraints:scaled:1} implies that $sy_3=-\lambda\mu$ and therefore $(\lambda\mu)'=3v_2$ and $v'_2=3\lambda\mu$. Here $'$ denotes derivative with respect to $s$. Comparing the parametrisations \eqref{eq:hypo:Stenzel} and \eqref{eq:invt:nh}, by Theorem \ref{thm:AC:CY}(ii), $\widetilde{\Psi}_0$ is the unique invariant complete Calabi--Yau structure on $T^\ast S^3$ normalised so that $\lambda (0)=1$.
\endproof
\end{prop}

\begin{remark}\label{rmk:Power:series:S^3}
With respect to the new variable $s$ introduced in the proof of the proposition the Taylor series of $\widetilde{\Psi}_{b}$ at $s=0$ is
\begin{subequations}
\begin{gather*}
\lambda^2(s) = 1 - \tfrac{9}{5} (b^2-1) s^2 + \tfrac{27}{35}(b^2-1)(2b^2-1) s^4 + \dots,\\
u_0(s) = 2bs - 4b^3 s^3 + \tfrac{6}{25}b^3(19b^2-9)s^5 + \dots,\\
u_1(s) = 2 s -\tfrac{2}{5}(13b^2-3)s^3 +\tfrac{6}{175} (172b^4-111b^2+9)s^5 + \dots,\\
u_2(s) = -2b s + b(4b^2-3)s^3 - \tfrac{3}{100}b (152b^4-192b^2+45) s^5 \dots, \\
v_0(s)= -\frac{2}{3} +4b^2s^2 -\frac{2}{5}b^2(19b^2-9)s^4+\dots,\\
v_1(s)=4b s^2 - \frac{4}{5}b(11 b^2-6)s^4+\dots,\\
v_2(s)=\frac{2}{3}-(4b^2-3)s^2+\frac{1}{20}(152b^4-192b^2+45)s^4+\dots.
\end{gather*}
\end{subequations}
One checks that as $b \ra 0$ these are exactly the first terms in the Taylor series at $s=0$ for the Calabi--Yau structure of Theorem \ref{thm:AC:CY}(ii).
\end{remark}


\subsection{Singular limit: the sine-cone}
Since the Calabi--Yau manifolds of Theorem \ref{thm:AC:CY} are asymptotic to the Calabi--Yau cone over the standard \se structure on $\Ndiag$, Proposition \ref{prop:Bubbling:AC:CY} implies that as $\epsilon = a, b$ converges to zero we can find some time $t_\epsilon >0$ such that the \sunitary{2}--structure $\Psi_{\epsilon} (t_\epsilon)$ is arbitrarily close to the standard invariant \se structure. We will now deduce from this that $\Psi_{\epsilon}$ converges to the sine-cone over the standard \se structure on $\Ndiag$ away from the singular orbit(s) as $\epsilon \ra 0$.

The main tool is a functional $\bohm$, introduced by B\"ohm in his work \cite{Bohm:Complete} on cohomogeneity one Einstein metrics, which plays the role of a Lyapunov function for the system \eqref{eq:nk:odes:t}.

Recall from Section \ref{sec:Einstein:odes} that the Einstein equations for a family of equidistant hypersurfaces $\R \times N$ is a first-order system for a pair $(g,L)$ of a Riemannian metric $g$ on $N$ and a symmetric endomorphism of $TN$ satisfying, among other constraints, (\ref{eqn:Mean:Curvature:odes}b). We restrict here to the case where the pair $(g,L)$ is homogeneous with respect to the action of a compact Lie group $G$. In \cite[Equation (8)]{Bohm:Complete} B\"ohm considers the function
\[
\bohm (g,L) = V(g)^{\frac{2}{n}}\left( |\mathring{L}|^2 + \Scal (g) \right),
\]
where $n$ is the dimension of $N$, $V(g)$ the volume of $N$ with respect to the metric $g$, $\mathring{L}$ is the traceless part of $L$ and $|\mathring{L}|^2=\trace{( \mathring{L}^2 )}$. Using (\ref{eqn:Mean:Curvature:odes}b), $\bohm (g,L)$ can be rewritten as
\begin{equation}\label{eqn:Lyapunov:function}
\bohm (g,L) = V(g)^{\frac{2}{n}} \left( (n-1)\Lambda + l^2 \right),
\end{equation}
where $l = \trace{L}$. Observe that the power of $V$ makes $\bohm$ scale-invariant. Moreover, suppose that $(g,L)$ is an integral curve of the ODE system defining an Einstein metric of cohomogeneity one on $(-\delta, \delta) \times N$, \ie $L = \frac{1}{2}g^{-1}g'$ and $(g,L)$ satisfies \eqref{eqn:Einstein:odes} and \eqref{eqn:Mean:Curvature:odes}. Then $\bohm$ is decreasing in $t$ whenever $l \geq 0$.

\begin{lemma}[{\cite[Proposition 2.2]{Bohm:Complete}}]\label{lem:Monotonicity:Bohm:functional}
If $(g_{t},L_t)$ is a $1$--parameter family of $G$--invariant pairs defining an Einstein metric of cohomogeneity one on $(-\delta, \delta ) \times N$ with Einstein constant $\Lambda$ then
\[
\frac{d}{dt}\bohm (g,L) = -2 \frac{n-1}{n} V^{\frac{2}{n}}(g) \trace{L} |\mathring{L}|^2.
\] 
\proof
Differentiate \eqref{eqn:Lyapunov:function} using (\ref{eqn:Mean:Curvature:odes}a) and $V'=lV$.
\endproof
\end{lemma}
 
We now specialise to the case of a homogeneous pair $(g,L)$ on $\Ndiag$ induced by a hypo or nearly hypo structure. Thus $n=5$ and $\Lambda =0$ or $\Lambda = 5$, respectively. By abuse of notation we write $\bohm (\psi)$ when the pair $(g,L)$ is determined by a (nearly) hypo structure $\psi$.

\begin{lemma}\label{lem:Traceless:2nd:ff}
Let $\Psi_\epsilon$ be one of the solutions of Theorems \ref{thm:Singular:IVP:S^2} and \ref{thm:Singular:IVP:S^3} with $\epsilon =a, b$, respectively.
\begin{enumerate}
\item $\bohm|_{\Psi_\epsilon}$ attains its minimum on the maximal volume orbit;
\item $\bohm (\psi) \geq 20$ for every invariant \nh structure $\psi$ with $l=0$ and equality holds if and only if $\psi$ is a rotated \se structure.
\end{enumerate}
\proof
By Proposition \ref{lem:Scalar:curvature:nearly:hypo}(iii) if $\mathring{L}=0$, then in particular $w_1=w_2=0$ and Corollary \ref{cor:characterisation:SC} implies that $\psi$ is an invariant hypersurface of the sine-cone. Hence $|\mathring{L}|>0$ on the solutions of Theorems \ref{thm:Singular:IVP:S^2} and \ref{thm:Singular:IVP:S^3} and Lemma \ref{lem:Monotonicity:Bohm:functional} implies (i). By \eqref{eqn:Lyapunov:function}, restricted to the space of invariant \nh structure with a fixed value of $l$, $\bohm$ is proportional to a fixed power of the volume. The statement in (ii) then follows from Proposition \ref{prop:Maximal:volume:orbits}.
\endproof
\end{lemma}
Observe that in view of part (ii) and the scale invariance of $\bohm$, we have $\bohm\equiv 20$ both on the sine-cone and the conifold. We now deduce the main result of this section.

\begin{theorem}\label{thm:Convergence:sine:cone}
As $a,b \ra 0$, $\Psi _{a}$ and $\Psi_{b}$ converge to the sine-cone over the standard Sasaki--Einstein structure on $\Ndiag$ on every relatively compact neighbourhood of the maximal volume orbit $t=\frac{\pi}{2}$ in $(0,\pi) \times \Ndiag$.
\proof
Consider first the functional $\bohm$ restricted to one of the Calabi--Yau structures of Theorem \ref{thm:AC:CY}. Since they are asymptotic to the conifold, $\bohm$ approaches the value $20$ asymptotically.

By the scale-invariance of $\bohm$ and Proposition \ref{prop:Bubbling:AC:CY} for every $\epsilon=a$ or $b$ sufficiently small we can then find some time $t_\epsilon>0$ such that $l(t_\epsilon)>0$ and $\bohm \left( \Psi _{\epsilon}(t_\epsilon) \right)$ is arbitrarily close to $20$.

Lemma \ref{lem:Traceless:2nd:ff} then implies that as $\epsilon \ra 0$ the maximal volume orbit $\Psi_{\epsilon}(t_{\max, \epsilon})$ converges to the ``rotated'' standard \se structure. The theorem is now a consequence of the continuous dependence on the initial conditions for solutions of \eqref{eq:nk:odes:t} (and their time reversals) starting from a principal orbit, \cf Proposition \ref{prop:nk:odes}.
\endproof
\end{theorem}

\section{An exotic \nk structure on $S^3 \times S^3$}\label{sec:S3xS3}

In this section we prove the existence of an inhomogeneous \nk structure on $S^3 \times S^3$. According to part (iii) of the Doubling Lemma \ref{lem:Doubling} every point of intersection of the curve $\betaW$ of Definition \ref{def:Maximal:volume:orbits} with the boundary of $W$ (recall Proposition \ref{prop:Maximal:volume:orbits}) 
corresponds to a smooth invariant \nk structure on $S^3 \times S^3$.

Our strategy is inspired by B\"ohm's work in \cite[\S 4]{Bohm:Spheres}: we will consider a function $f$ on the space of invariant \nh structures such that every solution $\Psi_b$ of Theorem \ref{thm:Singular:IVP:S^3} intersects the level set $f^{-1}(0)$ transversally and such that the intersection of $f^{-1}(0)$ with the space of maximal volume orbits $\calv_0$ lies in the preimage of the boundary of $W$ under the projection of Proposition \ref{prop:Maximal:volume:orbits}. Studying how the number of zeroes of $f$ before the maximal volume orbit of $\Psi_b$ varies  as a function of $b>0$ will allow us to detect intersection points of $\betaW$ with the boundary of $W$.

There are in fact various possible choices for the function $f$. Our choice is motivated by the fact that we will need to know more than the existence of an intersection point of $\betaW$ with the boundary of $W$. Indeed $\betaW (1)$, the maximal volume orbit of the homogeneous \nk structure on $S^3 \times S^3$, lies on the portion of $\partial W$ where $\lambda = 1$. For our applications in the next section, we will need to know that $\betaW$ has at least one intersection point with the second line $\lambda = \mu$ defining the boundary of $W$. By (\ref{eq:V:crit}b), $\lambda = \mu$ is equivalent to $v_0 =0$ on a maximal volume orbit. We will then count zeroes of $v_0$ (equivalently, critical points of $u_0$) and show that there exists $b \in (0,1)$ such that $v_0=0$ on the maximal volume orbit of $\Psi_b$.

\subsection{Counting zeroes of $v_0$}

Let $\Psi = (\lambda, u, v)$ be a solution of \eqref{eq:nk:odes:t}. The pair $(u_0,v_0)$ then satisfies the system
\[
\lambda\dot{u}_0 = -3v_0, \qquad \dot{v}_0=4\lambda u_0.
\]
An immediate consequence is that critical points of $u_0$ (equivalently, zeroes of $v_0$) are non-degenerate unless $u_0=0=v_0$. By Corollary \ref{cor:characterisation:SC} this can only occur on an invariant hypersurface of the sine-cone of Example \ref{Sine:Cone}.

For all $b \in (0,\infty)$ let $T_{b}$ denote the time of the maximal volume orbit of $\Psi_{b}$.

\begin{definition}
\label{def:count:critical:points}
For all $b \in (0,\infty)$ such that $v_0 (T_{b}) \neq 0$ let $\calc (b)$ be the number of zeroes of $v_0$ in $(0,T_{b})$.
\end{definition}

The following properties of $\calc(b)$ follow easily from the non-degeneracy of critical points of $u_0$, \cf \cite[Lemmas 4.4 and 4.5]{Bohm:Spheres}.

\begin{lemma}$ $
\label{lem:count:critical:points}
\begin{enumerate}
\item Given $0<b'<b''$, suppose that $v_0 (T_{b}) \neq 0$ for all $b \in [b',b'']$. Then $\calc (b)$ is constant on $[b',b'']$.
\item Suppose that $b_\ast>0$ is the unique value of $b \in [b_\ast-\delta, b_\ast+\delta]$ such that $v_0(T_{b})=0$. Then
\[
\left| \calc (b')- \calc (b'') \right| \leq 1
\]
for all $b', b'' \in [b_\ast-\delta, b_\ast+\delta]$ with $b' < b_\ast < b''$.
\end{enumerate}
It follows that for all $0<b'<b''$ there exist at least $|\calc (b')- \calc (b'')|$ values of $b \in [b',b'']$ such that $v_0 (T_{b})=0$.
\end{lemma}

Recalling Example \ref{S3xS3} and Remark \ref{rmk:a:b:homogeneous}, when $b=1$ then $v_0 = -6\cos{(2\sqrt{3}t)}$ and $T_{b}=\frac{\pi}{2\sqrt{3}}$. The existence of a second \nk structure on $S^3 \times S^3$ is now a consequence of Lemma \ref{lem:count:critical:points} and the following result.

\begin{prop}
\label{prop:count:critical:points:small:b0}
For all $b>0$ sufficiently small $\calc (b) \geq 2$.
\end{prop}

In the rest of the section we prove this proposition. 

For every $b>0$ consider the solution $\Psi_{b}$ of \eqref{eq:nk:odes:t} given by Theorem \ref{thm:Singular:IVP} and observe that $u_0$ is a solution of the second order ODE
\begin{equation}\label{eq:ODE:u0}
(\lambda u'_0)' + 12\lambda u_0 =0,
\end{equation}
with initial conditions 
$$u_0(0)=0, \quad u'_0(0)=2b^2>0.$$ 
By Theorem \ref{thm:Convergence:sine:cone}, $\lambda (t) \to \sin{t}$ uniformly on compact sets of $(0,\pi)$ as $b \ra 0$. Since we are interested in the behaviour of $u_0$ for small $b$, it makes sense to compare $u_0$, which satisfies \eqref{eq:ODE:u0}, with a solution of the simpler limiting equation
\begin{equation}
\label{eq:linearisation:SC}
(\sin{t}\, \xi')'  + 12\sin{t}\, \xi =0.
\end{equation}

\begin{remark*}
In fact \eqref{eq:linearisation:SC} naturally arises when considering the linearisation of \eqref{eq:nk:odes:t} on the sine-cone of Example \ref{Sine:Cone}. Indeed, since $u_0=0=v_0$ on the sine-cone it is not difficult to see that the space of solutions of the linearised equations consists of the vector field corresponding to time-translation along the sine-cone and the $2$--parameter family of solutions to \eqref{eq:linearisation:SC}.
\end{remark*}

Equation \eqref{eq:linearisation:SC} is one of Legendre's equations. Its solutions take the form
\begin{equation}\label{eq:linearisation:SC:sol}
\xi_0(t) = C_1 \,\xi_{\text{reg}} + C_2\, \xi_\text{sing}
\end{equation}
for constants $C_1, C_2 \in \R$ where 
\[
\xi_{\text{reg}} = 5\cos^3{t} - 3\cos{t}, \quad  \xi_\text{sing}= \frac{5}{2}\cos^2{t} + \frac{1}{8}\cos{t}\,(4\cos^2{t}-6\sin^2{t})\log{\frac{1-\cos{t}}{1+\cos{t}}}-\frac{2}{3}\,.
\]

\begin{lemma}\label{lem:linearisation:SC:sol}
There exists a solution $\xi_0$ of \eqref{eq:linearisation:SC} with the following property: there exists $0<t_1<t_2<t_3<\frac{\pi}{2}$ such that $\xi_0(t_1)=0=\xi_0(t_2)$, $\xi_0 \geq 0$ on $[t_1,t_2]$ and $\xi_0$ has a negative minimum at $t_3$.
\proof
The function $\xi_0$ in \eqref{eq:linearisation:SC:sol} with $C_1=0, C_2=1$ has all the required properties except that $t_3=\frac{\pi}{2}$. Since the Legendre polynomial $5\cos^3{t} - 3\cos{t}$ vanishes at $t=\frac{\pi}{2}$ and has strictly positive derivative there, by taking $C_1>0$ small enough we ensure that the qualitative behaviour of $\xi_0$ is unchanged and at the same time its minimum occurs at $t_3 < \frac{\pi}{2}$. See Figure~1.
\endproof
\end{lemma}
The solution $\xi_0$ in the previous lemma is singular at $t=0$ and $\pi$, but in the following we will only consider its restriction to the interval $[t_1,\pi-t_1]$.

\begin{figure} 
\label{fig:legendre}
\begin{center} 
\includegraphics[scale=0.65]{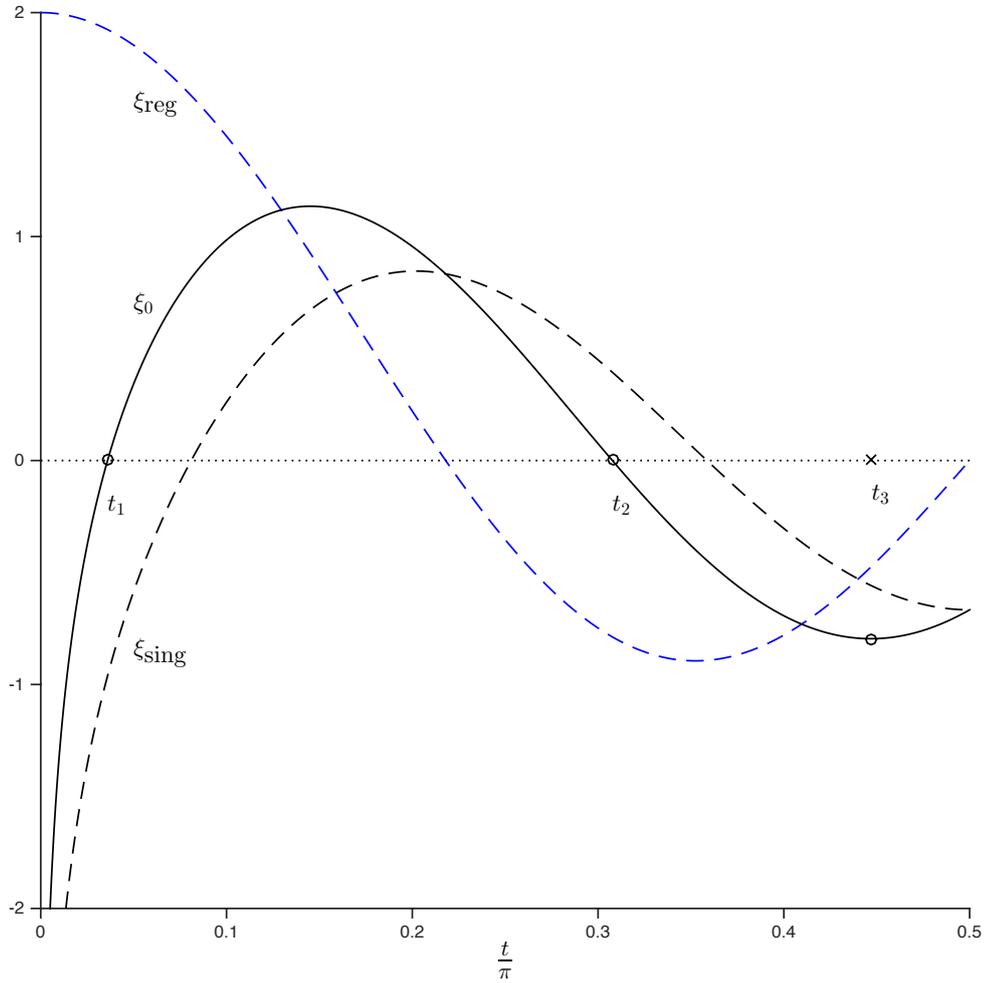}
\end{center}  
\caption{\small \sl Solution $\xi_0$ of \eqref{eq:linearisation:SC} as in Lemma \ref{lem:linearisation:SC:sol}}
\end{figure}

\subsection{A Sturm comparison argument}

We are now going to compare the function $u_0$ for a solution $\Psi_{b}$ of \eqref{eq:nk:odes:t} with $b$ small enough, with the solution $\xi_0$ of the Legendre equation \eqref{eq:linearisation:SC} given in the previous lemma. The comparison result we need is the following generalisation of the Sturm comparison theorem for Sturm--Liouville equations.

\begin{lemma}\label{lem:Sturm:Liouville}
Let $\lambda_1, \lambda_2, q_1, q_2$ be continuous functions on $[t_1, t_3] \subset \R$ such that
\[
\lambda _1 \geq \lambda_2 >0, \qquad q_2 \geq q_1,
\]
for all $t \in [t_1, t_3]$. Suppose that $\xi$ is a solution of
\[
( \lambda_1 \xi ')'+q_1\xi=0
\]
such that $\xi(t_1)=0=\xi(t_2)$ for some $t_1 < t_2<t_3$, $\xi (t) \geq 0$ for all $t \in [t_1,t_2]$ and $\xi$ has a strict negative minimum at $t_3$. Then any solution $u$ of
\[
( \lambda_2 u ')'+q_2u=0
\]
has a strict negative minimum in $(t_1, t_4)$ unless $u(t_1)<0$ and $u'(t_1) \geq 0$.
\end{lemma}
The proof is analogous to that of the classical Sturm comparison theorem, \cf \cite[\mbox{Theorem 3}, Chapter 10]{Birkhoff:Rota} and \cite[Proposition 5.9]{Bohm:Spheres}.

\begin{prop}
\label{prop:Minimum}
There exists $\epsilon>0$ such that for all $b < \epsilon$ the function $u_0$ has a strict negative minimum before the maximal volume orbit.
\proof
Fix $0<t_1<t_3<\frac{\pi}{2}$ as in Lemma \ref{lem:linearisation:SC:sol}.

By Theorem \ref{thm:Convergence:sine:cone} for all $\delta>0$ there exists $\epsilon>0$ such that if $b<\epsilon$ then $|\lambda(t) - \sin{t}|<\delta$ for all $t\in [t_1,\pi-t_1]$.

Consider the Sturm--Liouville equation
\begin{equation}\label{eq:linearisation:SC:deformation}
\left( (\sin{t}+\delta)\,\xi'\right)'+12(\sin{t}-\delta)\,\xi=0
\end{equation}
on $[t_1, \pi-t_1]$ with initial conditions $\xi(t_1)=\xi_0(t_1)$, $ \xi'(t_1)=\xi_0'(t_1)$, where $\xi_0$ is the solution of \eqref{eq:linearisation:SC} given by Lemma \ref{lem:linearisation:SC:sol}. Since the coefficients of \eqref{eq:linearisation:SC:deformation} converge uniformly to the ones of \eqref{eq:linearisation:SC} and $\sin{t}+\delta \geq \sin{t_1}>0$, by choosing $\delta>0$ small enough we can guarantee that $\xi$ has the same behaviour as $\xi_0$, \ie $\xi$ has two zeroes $t_1, t'_2$ and a negative minimum at $t_3'<\frac{\pi}{2}$.

We can then apply Lemma \ref{lem:Sturm:Liouville} with $\lambda_1=\sin{t}+\delta$, $q_1=\sin{t}-\delta$, $\lambda_2=q_2=\lambda$ to compare $u_0$ to $\xi$ on $[t_1,t_3']$: either $u_0$ has a strict negative minimum in $(t_1, t_3')$ or $u_0(t_1)<0$ and $u'_0(t_1) \geq 0$. In the former case, by taking $\epsilon$ smaller if necessary we can make sure that the maximal volume orbit of $\Psi_{b}$ occurs at some time $t > t_3'$ for all $b < \epsilon$. In the latter case, because of the initial conditions, $u_0$ must have a negative minimum in $[0,t_1]$.  
\endproof
\end{prop}

Because of the initial conditions, $u_0$ must also have a positive maximum before achieving the minimum in the proposition. Thus Proposition \ref{prop:count:critical:points:small:b0} is proved.

\begin{remark}\label{rmk:a:non:doubling}
The whole argument can also be carried out in the case of the solutions $\Psi_a$ of Theorem \ref{thm:Singular:IVP:S^2}: Definition \ref{def:count:critical:points}, Lemma \ref{lem:count:critical:points} and Proposition \ref{prop:Minimum} still hold. However in this case $u_0$ is a solution of \eqref{eq:ODE:u0} with initial conditions $u_0(0)=a^2>0$, $u'_0(0)=0$. We therefore conclude that $\calc (a) \geq 1$ for $a>0$ sufficiently small. When $a^2=3$, \ie $\Psi_a$ is the homogeneous \nk structure on $S^6$ given in Example \ref{S6}, one checks that $\calc(a)=0$. Thus there must exist at least one value of $a \in (0,\sqrt{3})$ that satisfies the hypothesis of Lemma \ref{lem:Doubling}(ii). On the other hand, by Remark \ref{rmk:a:b:homogeneous}, $a=\frac{\sqrt{3}}{2}$ already gives such a solution, the homogeneous \nk structure on $\CP^3$.  
\end{remark}

\begin{theorem}
\label{thm:exotic:S3xS3}
There exists $b \in (0,1)$ such that $\Psi_{b}$ defines an inhomogeneous \nk structure on $S^3 \times S^3$.
\proof
By Lemma \ref{lem:count:critical:points} and Proposition \ref{prop:count:critical:points:small:b0} there exists at least one $b \in (0,1)$ such that the maximal volume orbit of $\Psi_{b}$ lies on the portion of the boundary of $W$ with $\lambda = \mu >1$. By part (iii) of the Doubling Lemma \ref{lem:Doubling} this is enough to guarantee that $\Psi_{b}$ defines a smooth \nk structure on $S^3 \times S^3$. It remains to show that this is not homogeneous and therefore defines a new \nk structure.

Consider the Riccati equation \eqref{eqn:Riccati}. Since $L=\frac{1}{2}g^{-1}\dot{g}$, the component in the direction of the Reeb vector field $U^-$ gives
\[
\frac{\ddot{\lambda}}{\lambda}+\hat{R}(U^-,U^-)=0.
\]
Suppose that the new \nk structure is homogeneous. We first show that $\hat{R}(U^-,U^-)$ must be constant. Let $\hat{g}$ be the metric on $S^3\times S^3$ induced by the new \nk structure and $\nabla$ denote its Levi-Civita connection. Observe that $\hat{R}(U^-,U^-)=\hat{R}(\partial_t,J\partial_t,\partial_t,J\partial_t)$ because $J\partial_t$ is parallel to the Reeb vector field $U^-$. Since $\partial_t$ is the unit tangent vector of a geodesic in $(S^3 \times S^3,\hat{g})$ we have $\nabla _{\partial_t}\partial_t=0$. Moreover, the \nk property implies that $(\nabla _X J)X=0$ for every tangent vector $X$ (see  \cite[p. 517]{Bar} for a simple derivation from the \gtwo--holonomy perspective). Therefore $\nabla _{\partial_t}(J\partial_t)=0$. By Butruille's classification of homogeneous \nk $6$--manifolds \cite{Butruille}, if $(S^3\times S^3,\hat{g})$ is homogeneous it must be a $3$--symmetric space. Then part (ii) in \cite[Proposition 4.3]{Gray:1972} implies that
\[
(\nabla_X R)(X,JX,X,JX)=0
\]
for every tangent vector $X$. We conclude that $\hat{R}(U^-,U^-)$ is constant.

Now, since $\lambda (0)=\lambda(T)>0$, where $T$ is the maximal existence time of $\Psi_{b}$, and $\lambda$ is even at $t=0$ and $t=T$, the only possibility would be $\hat{R}(U^-,U^-)=0$ and $\lambda$ a constant. However, it is easy to check that the unique solution of \eqref{eq:nk:odes:t} with $\lambda=\text{const}$ is $\Psi_{b}$ with $b=1$, \ie the homogeneous \nk structure on $S^3 \times S^3$ of Example \ref{S3xS3}. For example, this follows from the power series expansions of Remark \ref{rmk:Power:series:S^3}.
\endproof
\end{theorem}

\begin{remark}
A similar argument shows that the solutions of Theorems \ref{thm:Singular:IVP:S^2} and \ref{thm:Singular:IVP:S^3} are mutually non-isometric. Indeed, using the constancy of $\hat{R}(U^-,U^-)$ on \nk $3$--symmetric spaces and the Taylor series of Remarks \ref{rmk:Power:Series:a0} and \ref{rmk:Power:series:S^3} one can show that $\Psi_a$ and $\Psi_b$ cannot be homogeneous unless $a=\frac{\sqrt{3}}{2},\sqrt{3}$ and $b=1,\frac{3}{2}$, the known homogeneous examples of Remark \ref{rmk:a:b:homogeneous}. On the other hand, assume that $f$ is an isometry between, say, the metric induced by $\Psi_a$ and $\Psi_{a'}$ for $a \neq a'$. Since $f$ cannot preserve the $\sunitary{2}\times\sunitary{2}$--orbits, the tangent space of a point is spanned by Killing fields and $\Psi_a, \Psi_a'$ would then be homogeneous. 
\end{remark}

\section{An exotic \nk structure on $S^6$}\label{sec:S6}

In this section we prove the existence of an inhomogeneous \nk structure on $S^6$. By part (iii) of the Matching Lemma \ref{lem:Matching} we have to show that there exist two values $a, b \in (0,\infty)$ such that the two curves $\alphaW, \betaW$ parametrising the maximal volume orbits of the solutions of Theorems \ref{thm:Singular:IVP:S^2} and \ref{thm:Singular:IVP:S^3} up to discrete symmetries intersect. One intersection point is already known to exist: by Remark \ref{rmk:a:b:homogeneous} the choice $a=\sqrt{3}$ and $b=\frac{3}{2}$ yields the standard \nk structure on $S^6$. We will show that there exists a second intersection point.
The key new ingredient is Proposition \ref{prop:Limit:large:a0} which gives us some control over the solution $\Psi_a$ as $a \to \infty$. 

\begin{theorem}
\label{thm:exotic:S6}
There exists $a \neq \sqrt{3}$ and $b \in (0,1)$ such that $\Psi_{a}$ and $\Psi_{b}$ satisfy the conditions of part (iii) of the Matching Lemma \ref{lem:Matching} and therefore 
define an inhomogeneous \nk structure on $S^6$.
\end{theorem}
The rest of the section contains the proof of the theorem, which consists of various steps.

We first give an alternative parametrisation of the space of maximal volume orbits $\calv$ to that of Proposition \ref{prop:Maximal:volume:orbits}. Recall that $\calv = \calv_0 \times \sorth{2}$, where $\calv_0=\calu \cap l^{-1}(0) \subset \lorentz{2}$. Denote by $\pi$ the natural projection $\lorentz{2}\ra \lorentz{2}/\sorth{2}$ to $\calv_0$. We identify $\lorentz{2}/\sorth
{2}$ with the upper hyperboloid
\[
H=\{ w=(w_0,w_1,w_2)\in \R^{1,2},\, |w|^2=-1, w_0>0 \}
\]
and take $(w_1,w_2) \in \R^2$ as global coordinates on $H$.

\begin{lemma}\label{lem:Maximal:volume:orbits}
The projection $\pi\co \calv_0 \ra H$ is a homeomorphism.
\proof
Since $H$ is endowed with the quotient topology we only have to show that $\pi$ is a bijection.

In the notation of Proposition \ref{prop:Invariant:nearly:hypo:structures}, let
\[
B=\left( \begin{array}{ccc} w_0 & x_0 & y_0\\ w_1 & x_1 & y_1\\ w_2 & x_2 & y_2 \end{array}\right) \in \calu 
\]
parametrise an invariant \nh structure with $\theta=0$ as usual. Thus $x_2<0$, $y_1>0$ and $B \in \lorentz{2}$. The projection $\pi\co\lorentz{2}\ra H$ is then the map $B\mapsto w \in H$. Let $\pi^{-1}(w)$ be the circle fibre \eqref{eq:Circle:fibre:proj:H} of $\pi$ through $B$, parametrised by an angle variable $\phi$. Since $\pi^{-1}(w) \cap \calu$ is an interval containing $0$ of length at most $\pi$, we take $\tan{\phi}$ as a coordinate on $\pi^{-1}(w) \cap \calu$. The range of $\tan{\phi}$ is then the connected interval containing $0$ where $x_1\tan{\phi}+y_1 > 0$ and $x_2-y_2\tan{\phi}< 0$.

By Lemma \ref{lem:Scalar:curvature:nearly:hypo}(ii) the restriction of $l$ to $\pi^{-1}(w) \cap \calu$ is
\[
l(\tan{\phi}) = 2\frac{x_1-y_1\tan{\phi}}{x_1\tan{\phi}+y_1} - 3\frac{x_2\tan{\phi}+y_2}{x_2-y_2\tan{\phi}}.
\]

Now, on one hand
\[
l'(\tan{\phi}) = -2\frac{x_1^2+y_1^2}{(x_1\tan{\phi}+y_1)^2}-3\frac{x_2^2+y_2^2}{(x_2-y_2\tan{\phi})^2}<0
\]
and therefore $\pi\co\calv_0\ra H$ is injective. On the other hand $l(\tan{\phi})$ approaches values of opposite sign as $\tan{\phi}$ converges to the endpoints of its range, as can be easily checked by working out the precise range of $\tan{\phi}$ according to the sign of $x_1, y_2$. More precisely, $l(\tan{\phi})$ approaches $\pm \infty$ at the boundary of its range unless (i) $x_1,y_2>0$ or (ii) $x_1, y_2<0$.

In case (i) $\max{\left( -\frac{y_1}{x_1},\frac{x_2}{y_2}\right)}<\tan{\phi}<\infty$ and $-2\frac{y_1}{x_1}+3\frac{x_2}{y_2}<l(\tan{\phi})<\infty$, while in case (ii) $-\infty<\tan{\phi}<\min{\left( -\frac{y_1}{x_1},\frac{x_2}{y_2}\right)}$ and $-\infty<l(\tan{\phi})<-2\frac{y_1}{x_1}+3\frac{x_2}{y_2}$.
\endproof
\end{lemma}

The lemma implies that the two continuous curves $\alpha$ and $\beta$ parametrising the maximal volume orbits of $\Psi_a$,  $\Psi_b$ can also be regarded as curves in $H$.
\begin{definition}\label{def:Maximal:volume:orbits:H}
Let $\alphaH, \betaH\co (0,\infty) \ra H$ be the two continuous curves in $H\simeq \calv_0$ parametrising the maximal volume orbits of the solutions 
$\{ \Psi_a \}_{a>0}$ and $\{ \Psi_b\}_{b>0}$ of Theorems \ref{thm:Singular:IVP:S^2} and \ref{thm:Singular:IVP:S^3}.
\end{definition}

We collect properties of $\alphaH, \betaH$ that are readily deduced from results of the previous sections.

\begin{lemma}\label{lem:Properties:curves:max:vol:orbits}$ $
\begin{enumerate}
\item The curves $\alphaH, \betaH$ do not self-intersect.
\item The curves $\alphaH$ and $\betaH$ cannot intersect for positive values of the parameters $a,b>0$.
\item $\lim_{a\ra 0^+}{\alphaH (a)}=\lim_{b\ra 0^+}{\betaH (b)}=(0,0)$ and $\alphaH (a)$, $\betaH (b)$ are distinct from the origin for $a, b>0$.
\end{enumerate}
\proof
Parts (i) and (ii) follow from the uniqueness of solutions to \eqref{eq:nk:odes:t} with given initial conditions since $H \simeq \calv_0$ by Lemma \ref{lem:Maximal:volume:orbits}. Part (iii) follows from Theorem \ref{thm:Convergence:sine:cone} and Corollary \ref{cor:characterisation:SC}.
\endproof
\end{lemma}

It will be important to understand the induced action of the discrete symmetries of Proposition \ref{prop:symmetries} on $H$. Observe that in terms of the parametrisation $(\lambda, u, v)$ of \eqref{eq:invt:nh} 
\begin{equation}\label{eq:Projection:transverse:unit:volume:metric}
w_0 = \frac{u_1 v_2-u_2 v_1}{V}, \qquad w_1 = \frac{u_0 v_2-u_2 v_0}{V}, \qquad w_2 = \frac{u_1 v_0-u_0 v_1}{V},
\end{equation}
where $V=\lambda\mu^2$ and $\mu^2=|u|^2$, since $w$ is the Minkowski ``cross-product'' of the two orthogonal space-like vectors $\frac{u}{|u|}, \frac{v}{|v|} \in \R^{1,2}$. An immediate consequence of Proposition \ref{prop:symmetries} and \eqref{eq:Projection:transverse:unit:volume:metric} is the following lemma.

\begin{lemma}\label{lem:Symmetries:curves:max:vol:orbits} Set $\epsilon =a,b > 0$ and let $\Psi_\epsilon$ be the solution to \eqref{eq:nk:odes:t} given by Theorems \ref{thm:Singular:IVP:S^2} or \ref{thm:Singular:IVP:S^3}, respectively. The image of $\alphaH$ or $\betaH$ under the involutions
\[
(w_1,w_2)\mapsto (-w_1,-w_2), \qquad (w_1,w_2)\mapsto (-w_1,w_2), \qquad (w_1,w_2)\mapsto (w_1,-w_2)\]
parametrises the maximal volume orbit of, respectively,
\[
\tau_2\circ\tau_3\circ\tau_4(\Psi_\epsilon),\qquad \tau_1\circ\tau_2\circ\tau_3(\Psi_\epsilon),\qquad \tau_1\circ\tau_4(\Psi_\epsilon).
\]
\end{lemma}

Finally, the results of Section \ref{sec:S3xS3} allow us to deduce the following crucial property of the curve $\betaH$. 

\begin{lemma}\label{lem:behaviour:betaH}
There exists $0<b'<b''\leq 1$ such that the arc $\betaH (b)$, $b' \leq b \leq b''$, and its image under the involutions of Lemma \ref{lem:Symmetries:curves:max:vol:orbits} form the boundary of a bounded closed set $D \subset H$ which contains the origin in its interior.
\proof
By the proof of Proposition \ref{prop:count:critical:points:small:b0} for $b>0$ sufficiently small the function $u_0$ in $\Psi_{b}$ has at least two critical points and one zero before the maximal volume orbit. On the other hand, by Remark \ref{rmk:a:b:homogeneous} the solution $\Psi_{b}$ with $b=1$ is the homogeneous \nk structure on $S^3\times S^3$. Thus for $b=1$ the function $u_0$ has a unique maximum before the maximal volume orbit and a unique zero, which occurs at the maximal volume orbit. We do not know whether the number of critical points or zeroes of $u_0$ before the maximal volume orbits is monotone in $b$. However, the observations above guarantee the existence of an interval $0<b'<b''\leq 1$ such that $u_0$ has a unique maximum and a unique zero before the maximal volume orbit for all $b'\leq b \leq b''$, a minimum on the maximal volume orbit when $b=b'$ and a zero on the maximal volume orbit when $b=b''$. 

By \eqref{eq:V:crit} and \eqref{eq:Projection:transverse:unit:volume:metric}, on a maximal volume orbit $w_1=0$ if and only if $v_0=0$ and similarly $w_2=0$ if and only if $u_0=0$. Thus the arc $\betaH (b)$, $b' < b < b''$ is contained in an \emph{open} quadrant of the $(w_1,w_2)$--plane. We conclude that the arc $\betaH (b)$, $b' \leq b \leq b''$ together with its image under the involutions of Lemma \ref{lem:Symmetries:curves:max:vol:orbits} form a continuous closed curve $\gamma$ in $H \simeq \R^2$. This curve is simple by Lemma \ref{lem:Properties:curves:max:vol:orbits}(i) and does not contain the origin by part (iii) of the same Lemma. The existence of the domain $D$ follows from the Jordan curve theorem. By the construction of $\gamma$ the origin is contained in the interior of $D$. 
\endproof
\end{lemma}

By Lemma \ref{lem:Properties:curves:max:vol:orbits}(ii) the boundary of $D$ cannot contain the points $\alphaH (\tfrac{\sqrt{3}}{2})$, $\alphaH (\sqrt{3})$ (the maximal volume orbits of the homogeneous \nk structures on $\CP^3$ and $S^6$ respectively by Remark \ref{rmk:a:b:homogeneous}) nor their image under the group generated by reflections along the axes. If $\alphaH (\tfrac{\sqrt{3}}{2})$ or $\alphaH (\sqrt{3})$ do not belong to $D$ then the proof of Theorem \ref{thm:exotic:S6} is complete, because the curve $\alphaH$ must intersect the boundary of $D$. The bad case is therefore when $\alphaH (\tfrac{\sqrt{3}}{2})$ and $\alphaH (\sqrt{3})$ both lie in the interior of $D$.

\begin{prop}\label{prop:Limit:large:a0}
The curve $\alphaH$ exits any compact set of $\calv_0$ as $a\ra\infty$.
\proof
In order to understand the behaviour of $\Psi_{a}$ as $a \ra \infty$, we observe that the Taylor series of Remark \ref{rmk:Power:Series:a0} suggest that we consider the rescaling
\begin{equation}\label{eq:Scaling:large:a0}
\widetilde{\Psi}_{a}(t)=\left( \lambda (t), \frac{u_0(t)}{a^2},\frac{u_1(t)}{a^2},\frac{u_2(t)}{a},\frac{v_0(t)}{a^2},\frac{v_1(t)}{a^2},\frac{v_2(t)}{a}\right).
\end{equation}
Observe that $\widetilde{\Psi}_a(t)$ does not satisfy the constraints \eqref{eq:constraints} and therefore does not define an \sunitary{2}--structure.

We now want to derive differential equations satisfied by $\widetilde{\Psi}
_a$ which are well-behaved in the limit $a \ra +\infty$. Using the conserved quantities \eqref{eq:constraints} satisfied by $\Psi_a$, we can rewrite the last equation of \eqref{eq:nk:odes:t} as
\begin{equation}\label{eq:lambda:modified:equation}
\dot{\lambda}=-\frac{2\lambda^2u_1}{v_1}-\frac{3v_2}{u_2}.
\end{equation}
Then $\widetilde{\Psi}_{a}$ is a solution of the ODE system
\begin{alignat*}{2}
\lambda \dot{u}_0 &+ 3v_0=0,&\quad &\dot{v}_0 -4\lambda u_0=0,\\
\lambda \dot{u}_1 &+ 3v_1 - 2\epsilon^2\lambda^2=0,&\quad &\dot{v}_1 -4\lambda u_1=0,\\
\lambda \dot{u}_2 &+ 3v_2=0,&\quad &\lambda \dot{v}_2 - 4\lambda^2u_2 + 3u_2=0,
\end{alignat*}
where $\epsilon = \frac{1}{a}$, together with  \eqref{eq:lambda:modified:equation}.

The conditions of Lemma \ref{lem:Smooth:extension:S2} and of Theorem \ref{thm:Singular:IVP:S^2} suggest that we write
\[
\begin{gathered}
u_0(t)=1 + t^2 y_1 (t), \qquad u_1 (t) = 1 + t^2 y_2 (t),\qquad u_2(t) = t^2 y_3(t),\\
v_0 (t) = t^2 y_4(t), \qquad v_1(t)=t^2 y_5(t), \qquad v_2(t) = t^2 y_6(t), \qquad \lambda (t) = ty_7(t).
\end{gathered}
\]
Then $y=(y_1, \dots, y_7)$ satisfies
\begin{alignat*}{2}
\dot{y}_1 &=-\frac{1}{t}\left( 2y_1 + 3\frac{y_4}{y_7} \right),&\qquad  &\dot{y}_4=-\frac{1}{t}\left( 2y_4 - 4 y_7 \right) + 4t y_1 y_7,\\
\dot{y}_2 &=-\frac{1}{t}\left( 2y_2 + 3\frac{y_5}{y_7}-2\epsilon^2y_7\right),&\qquad &\dot{y}_5=-\frac{1}{t}\left( 2y_5 - 4 y_7 \right) + 4t y_2 y_7,\\
\dot{y}_3 &=-\frac{1}{t}\left( 2y_3 + 3\frac{y_6}{y_7}\right),&\qquad &\dot{y}_6=-\frac{1}{t}\left( 2y_6 + 3\frac{y_3}{y_7} \right) + 4t y_3 y_7,
\end{alignat*}
\[
\dot{y}_7=-\frac{1}{t}\left( y_7 + 2\frac{y_7^2}{y_5} + 3\frac{y_6}{y_3}\right) -2t\frac{y_2y_7^2}{y_5},
\]
and the initial condition
\[
y_0 = \left( -3, -3+\frac{3}{2}\epsilon^2, -\frac{3\sqrt{3}}{2}, \,3,\, 3, \frac{3\sqrt{3}}{2}, \frac{3}{2} \right).
\]

Thus $y$ is a solution of a singular initial value problem of the form $\dot{y}=\frac{1}{t}M_{-1}(y)+M(t,y)$. It is immediate to check that $M_{-1}\left( y_0 \right) =0$ and that the linearisation of $M_{-1}$ at $y_0$
\[
d_{y_0}M_{-1}=\left( \begin{array}{ccccccc}
-2 & 0 & 0 & -2 & 0 & 0 & 4\\
0 & -2 & 0 & 0 & -2 & 0 & 4+2\epsilon^2\\
0 & 0  & -2  & 0 & 0 & -2 & 2\sqrt{3}\\
0 & 0 & 0 & -2 & 0 & 0 & 4\\
0 & 0 & 0 & 0 & -2 & 0 & 4\\
0 & 0 & -2 & 0 & 0 & -2 & -2\sqrt{3}\\
0 & 0 & \frac{2}{\sqrt{3}} & 0 & \frac{1}{2} & \frac{2}{\sqrt{3}} & -3\\
\end{array}\right)
\]
satisfies
\[
\det{\left( h \text{Id} - d_{y_0}M_{-1}\right)}=h(h+1)(h+2)^3(h+4)^2 \neq 0 
\]
for all integer $h \geq 1$. Theorem \ref{thm:Singular:IVP} then implies the existence of a $1$--parameter family $\widetilde{\Psi}_{a}$ depending continuously on $\epsilon=\frac{1}{a}\geq 0$.

Hence as $a \ra \infty$, $\widetilde{\Psi}_{a}$ approaches a well-defined smooth limit $\widetilde{\Psi}_\infty$ defined on an interval $0\leq t<T$. In particular, reversing the scaling \eqref{eq:Scaling:large:a0}, there exist smooth functions $l_\infty, V_\infty$ such that the mean curvature $l$ and the orbital volume function $V$ of $\Psi_{a}$ are $C^0$--close to $l_{\infty}$ and $a^4 V_{\infty}$, respectively, for $a$ sufficiently large. Fix $0<t_\ast<\min{(T,\tfrac{\pi}{2})}$ such that $l_\infty (t_\ast), V_\infty (t_\ast)>0$ (the initial conditions guarantee the existence of such $t_\ast$). Then the comparison results of Proposition \ref{prop:crit:vol} imply that for $a$ sufficiently large the volume $V_{\text{max}}$ of the maximal volume orbit of $\Psi_{a}$ satisfies
\[
V_{\text{max}} \geq a^4 \frac{V_\infty (t_\ast)}{\sin^5{(t_\ast+\tfrac{\pi}{2})}}-\delta,
\]
where $\delta>0$ can be choosen arbitrarily small as $a \ra \infty$. Thus the maximal volume orbit of $\Psi_a$ has unbounded volume as $a \ra \infty$. Recalling that $V=\lambda\mu^2$, the result follows immediately from the parametrisation of $\calv_0$ as a branched $4$--fold cover of the wedge $W$ in the $(\lambda,\mu)$--plane given in Proposition \ref{prop:Maximal:volume:orbits}.
\endproof
\end{prop}

\proof[Proof of Theorem \ref{thm:exotic:S6}]
By Proposition \ref{prop:Limit:large:a0} the curve $\alphaH$ intersects the boundary of $D$. By Lemmas \ref{lem:Properties:curves:max:vol:orbits}(ii) and \ref{lem:Symmetries:curves:max:vol:orbits} such an intersection point can only occur on the image of the arc $\betaH (b)$, $b \in (b',b'')$, under the symmetries $(w_1,w_2)\mapsto (-w_1,w_2)$ or $(w_1,w_2)\mapsto (w_1,-w_2)$. Part (iii) of the Matching Lemma \ref{lem:Matching} then implies the existence of a smooth \nk structure on $S^6$. It remains only to show that this is not homogeneous.

As in the proof of Theorem \ref{thm:exotic:S3xS3} we look at the Riccati equation \eqref{eqn:Riccati} in the direction of the Reeb vector field $U^-$. If the constructed \nk structure on 
$S^6$ were homogeneous then it would have to be the standard \nk structure on $S^6$. In particular, $\hat{R}(U^-,U^-)=1$ and $\lambda = C_1\cos{t}+C_2\sin{t}$ for 
some constants $C_1$ and $C_2$. Without loss of generality assume that the singular orbit $S^3$ occurs at $t=0$. Since $\lambda$ must be even in $t$ we have $C_2=0$. The singular orbit $S^2$ must then occur at $t=\frac{\pi}{2}$ (the first zero of $\lambda$) and the Taylor series of Remark \ref{rmk:Power:Series:a0} (or the condition $y_7 (0)=\frac{3}{2}$ in the proof of Theorem \ref{thm:Singular:IVP:S^2}) imply that $C_1=\frac{3}{2}$. This however is impossible since $\lambda (0)=b<b''\leq 1$ by assumption.
\endproof

\section{Conjectures and numerical results}\label{sec:numerics}

Theorems \ref{thm:exotic:S3xS3} and \ref{thm:exotic:S6} guarantee the existence of at least one complete inhomogeneous 
\nk structure both on $S^3 \times S^3$ and on  $S^6$ (as stated in the Main Theorem). In fact we make the following: 

\begin{conjecture*}
The Main Theorem yields all (inhomogeneous) complete cohomogeneity one \nk structures on simply connected manifolds. In particular, $S^2 \times S^4$ does not admit any cohomogeneity one \nk structure and $\CP^3$ admits only its homogeneous one.
\end{conjecture*}

This conjecture is motivated by a systematic numerical study of the ODE system \eqref{eq:nk:odes:t}. In this final  less formal section we discuss numerical results in support of the Conjecture and provide some numerical information about the \nk structures of the Main Theorem. A more detailed account of the numerics may appear elsewhere.

\subsection{The numerical scheme}
The proofs of Theorems \ref{thm:Singular:IVP:S^2} and \ref{thm:Singular:IVP:S^3}, where existence of the two $1$--parameter families $\{ \Psi_a \}_{a>0}$ and $\{ \Psi_b\}_{b>0}$ is established, can be turned into a constructive numerical scheme useful in the study of the system \eqref{eq:nk:odes:t}. 
These proofs  showed the existence of recurrence relations that uniquely determine the coefficients of the Taylor series of $\Psi_a$ and $\Psi_b$  at $t=0$  
once initial conditions are fixed. 
The initial conditions are uniquely determined by the choice of $a$ or $b$ respectively, \eg see \eqref{eq:initial:condition:y:a0} for the initial conditions in terms of $a$. 
After computing the first several terms of the Taylor series by hand we made these recurrence relations explicit and then computed 
the first 50 nonzero terms in these Taylor series symbolically in MATLAB using its Symbolic Math Toolbox. The first few terms of these power series expansions are recorded in Remarks \ref{rmk:Power:Series:a0} and \ref{rmk:Power:series:S^3}. 

The main problem in using numerical methods to study the existence of 
new complete cohomogeneity one \nk structures is the inevitability  of singularities in the coefficients of the equations  \eqref{eq:nk:odes:t}.
We overcome this problem as follows.  First using the symbolic polynomials associated to the two families $\{ \Psi_a \}_{a>0}$ and $\{ \Psi_b\}_{b>0}$
described above we find (very high order) approximations to regular 
initial conditions for the ODE system  \eqref{eq:nk:odes:t} simply by evaluating these polynomials at some sufficiently small positive value $t_*$ of $t$. 
Now, by Proposition \ref{prop:Existence:maximal:volume:orbit}, we know that every solution $\Psi_a$ or $\Psi_b$ has a unique maximal volume 
orbit that it attains at some time $t_{\textrm{vmax}}$ before the solution develops its second singularity. Moreover, the maximal volume orbit is characterised algebraically 
by the equality \eqref{eq:crit:vol}. Therefore for each positive $a$ or $b$ we approximate the solution $\Psi_a$ or $\Psi_b$ on the interval 
$[t_*,t_{\textrm{vmax}}]$ by using one of the standard MATLAB ODE solvers (we found ODE45 to be suitable) to integrate numerically 
equation \eqref{eq:nk:odes:t} beginning with nonsingular initial conditions at $t=t_*$  and detecting $t_{\textrm{vmax}}$ by evolving the solution numerically until a zero of $2\lambda^4u_1-3u_2v_2$ occurs. 

In particular, this numerical scheme allows us to obtain very accurate numerical approximations to the two 
curves $\alpha$ and $\beta: (0,\infty) \to \calv$ parameterising the (unique)
maximal volume orbits of the two $1$--parameter families $\{ \Psi_a \}_{a>0}$ and $\{ \Psi_b\}_{b>0}$ respectively. 
We can therefore use these numerical approximations of $\alpha$ and $\beta$ to study when the conditions 
of the Doubling and the Matching Lemmas \ref{lem:Doubling} and \ref{lem:Matching} can be satisfied. 
This has the great advantage that we never have to solve numerically toward an unknown final time when the solution becomes 
singular again (some numerical experimentation makes it clear that in practice that more naive strategy is very unstable.)

\subsection*{Conjectures based on numerics}
The properties of the numerical approximations to the curves $\alpha$ and $\beta$ so obtained suggest a number of concrete conjectures on the behaviour of $\Psi_a$ and $\Psi_b$ as functions of the parameters, that one might hope to establish analytically.

\begin{conjecture}
The volume of the maximal volume orbit of $\Psi_a$ and $\Psi_b$ is strictly increasing in $a$ and $b$, respectively.
\end{conjecture}

\begin{remark*}
The fact that the volume of the maximal volume orbit of $\Psi_a$ is eventually strictly increasing, \ie strictly increasing for all $a$ sufficiently large, follows from 
the rescaling argument employed in the proof of Proposition \ref{prop:Limit:large:a0}. 
A different rescaling argument based on further contemplation of the power series expansions for $\Psi_b$ might establish the same result for 
the family $\{ \Psi_b\}_{b>0}$.
\end{remark*}

An immediate consequence of the verification of this conjecture would be: the curves $\alphaW$ and $\betaW$ of Definition \ref{def:Maximal:volume:orbits} can never self-intersect and hence parts (i) and (ii) of the Matching Lemma \ref{lem:Matching} can never be applied. In particular, for any complete cohomogeneity one \nk structure on $\CP^3$, $S^2 \times S^4$ and $S^3 \times S^3$ both singular orbits must have the same volume, \ie is obtained by ``doubling'' some member 
of one of the two families $\{ \Psi_a \}_{a>0}$ and $\{ \Psi_b\}_{b>0}$.

We next consider how many complete \nk structures arise by applying the Doubling Lemma \ref{lem:Doubling}. By \eqref{eq:V:crit} the conditions of the lemma are satisfied if and only if either $u_0$ or $v_0$ has a zero on the maximal volume orbit. In Definition \ref{def:count:critical:points} we considered the number $\mathcal{C}(b)$ of zeroes of $v_0$ before the maximal volume orbit of $\Psi_b$. According to Remark \ref{rmk:a:non:doubling} it is possible to extend this definition to the family $\{ \Psi_a \} _{a>0}$ as well, and we write this count as $\mathcal{C}(a)$.
\begin{conjecture}
The count $\mathcal{C}$ of zeroes of $v_0$ before the maximal volume orbit satisfies:
\begin{enumerate}
\item $\mathcal{C}(a)$ and $\mathcal{C}(b)$ are decreasing in $a$ and $b$, respectively;
\item for $a>0$ sufficiently small $\mathcal{C}(a)=1$; for $b>0$ sufficiently small $\mathcal{C}(b)=2$;
\item $\mathcal{C}(a)=0$ for all $a \geq \sqrt{3}$ and $\mathcal{C}(b)=1$ for all $b \geq 1$.
\end{enumerate}
\end{conjecture}
Based on numerical evidence, we also conjecture analogous properties for the count of zeroes of $u_0$ before the maximal volume orbit.

As a result, the conditions of the Doubling Lemma \ref{lem:Doubling} are satisfied only three times: 
once in the family $\{ \Psi_a \}_{a>0}$, corresponding to the homogeneous \nk structure on $\CP^3$;  
twice in the family $\{ \Psi_b \}_{b>0}$, yielding the inhomogeneous \nk structure on $S^3 \times S^3$ of the Main Theorem and the homogeneous \nk structure of Example \ref{S3xS3}.

\begin{figure}
\label{fig:alpha:beta:disk}
\begin{center} 
\includegraphics[scale=0.64]{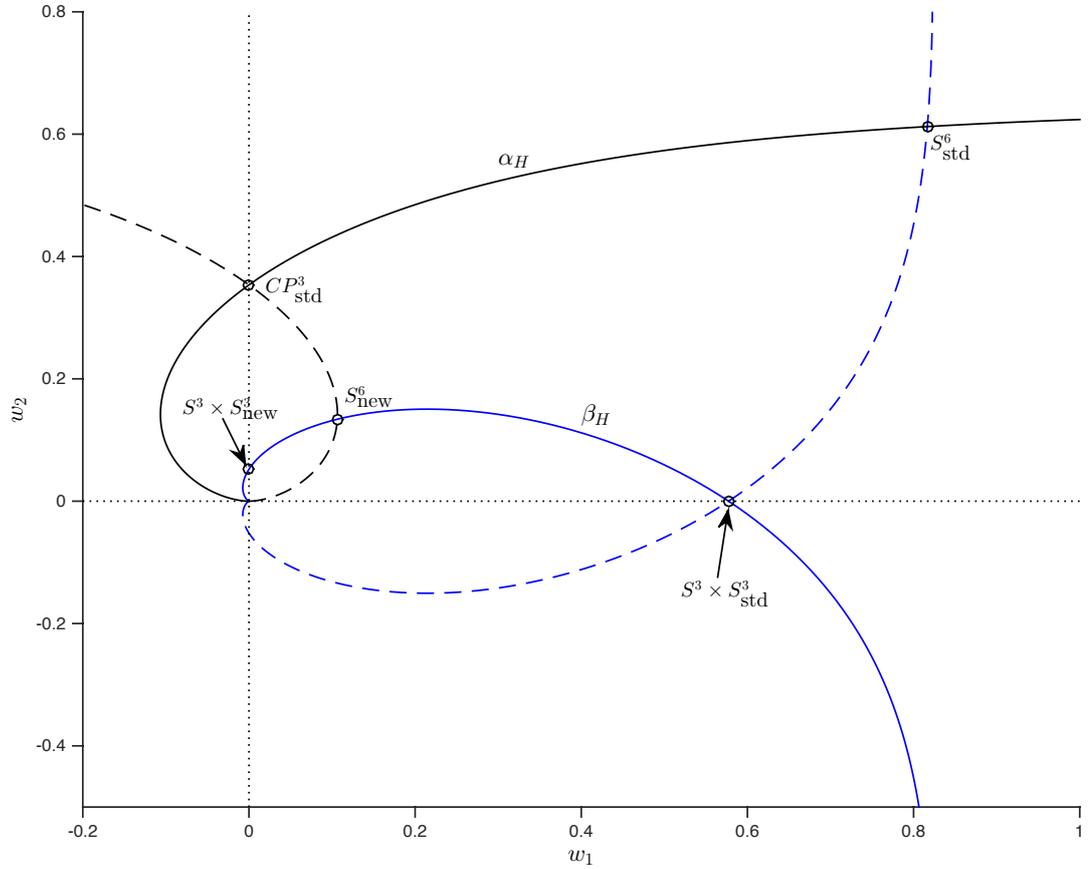}
\end{center}  
\caption{\small \sl Plots of $\alphaH$ and $\betaH$ showing locations of the 5 complete cohomogeneity 
one \nk structures}  
\end{figure}

Finally, Figure 2 plots the numerical approximations to the curves $\alphaH$ and $\betaH$ of Definition \ref{def:Maximal:volume:orbits:H}. This provides further numerical evidence that the Doubling Lemma \ref{lem:Doubling} yields exactly three cohomogeneity one \nk structures;  these correspond 
to the three points of intersection of the curves $\alphaH$ and $\betaH$ with the axes (the origin parametrises the circle of \se structures on $\Ndiag$ and must therefore be excluded). 
The plot also suggests that $\alphaH$ intersects the image of $\betaH$ under reflections in the two axes exactly twice; these two intersection points yield the inhomogeneous \nk structure on $S^6$ given by the Main Theorem and the homogeneous one of Example \ref{S6}. 

\subsection{The volumes of the cohomogeneity one \nk structures}

According to the Conjecture above, there exist exactly two new complete cohomogeneity one \nk structures.  
The numerical analysis of the previous subsection provides more concrete information about these two new solutions 
than the abstract existence proofs of Theorems \ref{thm:exotic:S3xS3} and \ref{thm:exotic:S6},
and gives a way to compare them quantitatively with the known homogeneous \nk structures and with the sine-cone.

In Table \ref{tab:Comparison:coho1:nk} we consider various quantities describing the geometry of a complete cohomogeneity one \nk $6$--manifold: the size of the two singular orbits $O_1$ and $O_2$ in terms of the parameters $a$ or $b\,$; the maximum $V_{\textup{max}}$ of the orbital volume function; and the total volume $\vol$ normalised so that $\vol (S^6_{\textup{std}})=1$.

The values of these quantities on the two new inhomogeneous solutions are numerical approximations. For a more accurate result, we cut $S^6$ and $S^3 \times S^3$ along the maximal volume orbit of their inhomogeneous \nk structures and compute numerically the 
total volume on each half separately. 
Information about the sine-cone and the homogeneous \nk structures on $S^6$, $S^3\times S^3$ and $\CP^3$ is computed analytically from the explicit solutions of Examples \ref{Sine:Cone}--\ref{CP3}. The values in the tables are all obtained directly from those expressions. Since the orbital volume function is $V=\lambda\mu^2$, the total volume is easily deduced by integration. For example, for the homogeneous \nk structure on $S^6$ we have $V(t) =  \frac{27}{2} \sin^2{t}\cos^3{t}$ and hence $\textup{Vol}(S^6_{\textup{std}}) =\frac{9}{5}V_0$, where $V_0$ is the volume of $\Ndiag$ with respect to the standard \se metric. Since the volume of the $6$--sphere with respect to the round metric of curvature $1$ is $\frac{16}{15}\pi^3$, we must have $V_0 = \frac{16}{27}\pi^3$.

\begin{table}[!h]
\centering
\begin{tabular}{ccccc}
$M$ &  $O_1$  &  $O_2$  &  $V_{\textup{max}}$  & $\vol$ \\ \hline
\\
 sine-cone & $0$ & $0$ & $1$ & $\frac{16}{27} \approx 0.5926$\\
 \\
 $S^3\times S^3_{\textup{new}}$ & $b=0.3736$ & $b=0.3736$ & $1.0041$ & $0.5929$\\
\\
 $S^6_{\textup{new}}$ & $a=0.5646$ & $b=0.5985$ & $1.0385$ & $0.5752$\\
\\
$\CP^3$ & $a=\frac{\sqrt{3}}{2}$ & $a=\frac{\sqrt{3}}{2}$ & $\frac{27\sqrt{2}}{32} \approx 1.1932$ & $\frac{5}{8}$\\
\\
$S^3 \times S^3_{\textup{std}}$ & $b=1$ & $b=1$ & $\frac{4}{3}$ & $\frac{10\pi}{27\sqrt{3}} \approx 0.6718$\\
\\
$S^6_{\textup{std}}$ & $a=\sqrt{3}$ & $b=\frac{3}{2}$ & $\frac{81\sqrt{3}}{25\sqrt{5}} \approx 2.5097$ & $1$\\
\\
\end{tabular}
\caption{Cohomogeneity one \nk manifolds}
\label{tab:Comparison:coho1:nk}
\end{table}

The quantities in Table \ref{tab:Comparison:coho1:nk} all give a measure of how both inhomogeneous \nk structures, in particular the one on $S^3 \times S^3$, are much closer to the sine-cone than the homogeneous ones. Observe that the total volume $\vol$ is greater than the volume of the sine-cone in all cases except for the inhomogeneous \nk structure on $S^6$.

\bibliographystyle{amsinitial}
\bibliography{nkrefs}
\end{document}